\pdfoutput=1                      

\documentclass[10pt,a4paper]{article}

\usepackage[accsupp]{axessibility}    

\usepackage{mathtools}

\usepackage{latexsym,amsfonts,amsmath,amsthm,amssymb, graphicx, mathrsfs, fancyhdr, 
amscd, 
fontenc,verbatim,stix,amsfonts,amsthm,indentfirst}

\numberwithin{equation}{section} 
\usepackage[all]{xy}
\usepackage{braket}

\usepackage{color,xcolor}

\usepackage[draft=false,setpagesize=false,pdfstartview=FitH,
colorlinks=true,citecolor=blue,pagebackref=true]{hyperref}

\theoremstyle{definition}
\newtheorem{definition}{Definition}[section]

\theoremstyle{plain}
\newtheorem{proposition}[definition]{Proposition}
\newtheorem{lemma}[definition]{Lemma}
\newtheorem{theorem}[definition]{Theorem}
\newtheorem{corollary}[definition]{Corollary}

\newtheorem{conjecture}{Conjecture}
\newtheorem*{conjecture*}{Conjecture}
\newtheorem{question}[conjecture]{Question}
\theoremstyle{definition}
\newtheorem{remark}[definition]{Remark}

\DeclareMathOperator{\Lip}{Lip}
\newcommand{\Rm}{\mathbb{R}^m}

\newcommand{\diver}{\mathrm{div}}
\newcommand{\rig}{\rightarrow}
\newcommand{\la}{\left\langle}
\newcommand{\ra}{\right\rangle}

\newcommand{\cM}{\mathcal{M}}

\newcommand{\cX}{\mathcal{Y}}

\newcommand{\cI}{\mathcal{I}}
\newcommand{\cJ}{\mathcal{J}}
\newcommand{\scrF}{\mathscr{F}}

\def\ov#1{\overline{#1}}
\newcommand{\e}{\varepsilon}
\newcommand{\RN}{\mathbb{R}^m}

\DeclareMathOperator{\sh}{\mathrm{sh}}
\DeclareMathOperator{\ch}{\mathrm{ch}}

\def\wt#1{\widetilde{#1}}

\newcommand{\eps}{\varepsilon}

\newcommand{\di}{\mathrm{d}}
\newcommand{\rd}{\mathrm{d}}
\newcommand{\R}{\mathbb R}

\newcommand{\N}{\mathbb N}

\newcommand{\metric}{\langle \, , \, \rangle}

\newcommand{\lip}{\mathrm{Lip}}
\newcommand{\loc}{\mathrm{loc}}
\newcommand{\disp}{\displaystyle}

\newcommand{\LL}{\mathbb{L}}

\newcommand{\BI}{\mathrm{BI}}
\newcommand{\M}{\mathrm{MS}}
\newcommand{\FF}{\mathscr{F}}

\newcommand{\supp}{\operatorname{supp}}

\newcommand{\SF}{\operatorname{II}}

\newcommand{\diam}{\mathrm{diam}}
\newcommand{\haus}{\mathscr{H}}

\newcommand{\ol}{\overline}

\newcommand{\Spa}{\mathcal{S}}

\newcommand{\measrest}{%
  \,\raisebox{-.127ex}{\reflectbox{\rotatebox[origin=br]{-90}{$\lnot$}}}\,%
}
\newcommand{\PP}{\mathscr{P}}

\allowdisplaybreaks[1]

\begin{document}

\author{Jaeyoung Byeon \and Norihisa Ikoma \and Andrea Malchiodi \and Luciano Mari}
\title{\textbf{
Existence and regularity for prescribed Lorentzian mean curvature hypersurfaces, 
and the Born--Infeld model}}
\date{\today}
\maketitle

\scriptsize \begin{center} Department of Mathematical Sciences, KAIST, \\ 
291 Daehak-ro, Yuseong-gu, Daejeon 305-701, Republic of Korea\\
E-mail: byeon@kaist.ac.kr
\end{center}

\scriptsize \begin{center} 
Department of Mathematics, Faculty of Science and Technology, Keio University,\\ 
Yagami Campus: 3-14-1 Hiyoshi, Kohoku-ku, Yokohama, Kanagawa 2238522, Japan\\
E-mail: ikoma@math.keio.ac.jp
\end{center}

\scriptsize \begin{center}
Scuola Normale Superiore,\\
Piazza dei Cavalieri, 7, 56126 Pisa, Italy\\
E-mail: andrea.malchiodi@sns.it
\end{center}

\scriptsize \begin{center} Dipartimento di Matematica, Universit\`a degli Studi di Torino,\\
Via Carlo Alberto 10, I-10123 Torino, Italy\\
E-mail: luciano.mari@unito.it
\end{center}

\normalsize

\begin{abstract}
Given a measure $\rho$ on a domain $\Omega \subset \R^m$, we study spacelike graphs over $\Omega$ in Minkowski space 
with Lorentzian mean curvature $\rho$ and Dirichlet boundary condition on $\partial \Omega$, 
which solve
	\begin{equation}\label{eq:abst-BI}
		-\diver \left( \frac{Du}{\sqrt{1-|Du|^2}} \right) = \rho  \qquad \text{ in } \, \Omega \subset \R^m. 
		\tag{$\mathcal{BI}$}
	\end{equation}
The graph function also represents the electric potential generated by a charge $\rho$ in electrostatic Born-Infeld's theory. Even though there exists a unique minimizer $u_\rho$ of the associated action
	\[
	I_\rho(\psi) \doteq \int_{\Omega} \Big( 1 - \sqrt{1-|D\psi|^2} \Big) \di x - \langle \rho, \psi \rangle
	\]
among functions $\psi$ satisfying $|D\psi| \le 1$, 
by the lack of smoothness of the Lagrangian density for $|D\psi| = 1$ one cannot guarantee that $u_\rho$ satisfies the Euler-Lagrange equation \eqref{eq:abst-BI}.  
A chief difficulty comes from the possible presence of light segments in the graph of $u_\rho$. 
In this paper, we investigate the existence of a solution for general $\rho$. 
In particular, we give sufficient conditions to guarantee that $u_\rho$ solves \eqref{eq:abst-BI} and 
enjoys $\log$-improved energy and $W^{2,2}_\loc$ estimate. 
Furthermore, we construct examples 
which suggest a sharp threshold for the regularity of $\rho$ to ensure the solvability of \eqref{eq:abst-BI}. 

\end{abstract}
%
%



\tableofcontents

\section{Introduction}

Spacelike maximal and constant mean curvature hypersurfaces in Lorentzian manifolds, possibly with singularities, and more generally spacelike hypersurfaces $M$ with prescribed Lo\-ren\-tzian mean curvature $\rho$, play a prominent role in General Relativity. For instance, their use is substantial in connection to positive energy theorems and to the initial value problem for solutions to the Einstein field equation (see \cite{mt} and the references therein). Therefore, guaranteeing their existence and grasping their qualitative properties, either in entire space or in a bounded domain with a given boundary condition, helps to understand the mathematics behind Einstein's theory. Once $\rho$ is prescribed, one of the core issues making the existence and regularity problem challenging is the possibility that the constructed hypersurface $M$ ceases to be spacelike somewhere. For example, $M$ may contain a light segment (``go null" in the terminology of \cite{mt}), a feature that, as we shall see, can actually occur even for reasonably well-behaved $\rho$. The appearance of a light segment $\overline{xy}$ forces to properly justify the sentence ``$M$ has mean curvature $\rho$ in a neighborhood of $\overline{xy}$", and also suggests to allow for possibly nonsmooth mean curvature functions. Indeed, the behavior of $M$ near a light segment is a problem which is, to the best of our knowledge, mostly open.


The purpose of the present paper is to investigate the prescribed Lorentzian mean curvature problem for general $\rho$. 
For simplicity, we shall restrict to hypersurfaces $M$ in the Lorentz-Minkowski ambient space
	$$
	\LL^{m+1} \doteq \R \times \RN\qquad \text{with Lorentzian metric} \quad - \di x^0 \otimes \di x^0 + \sum_{i=1}^m \di x^i \otimes \di x^i,
	$$
although most of our results likely allow for extension to more general Lorentzian manifolds. 
The spacelike condition ensures that $M$ is the multigraph, 
over some open subset $\Omega$ of the totally geodesic slice $\RN \doteq \{x^0 = 0\}$, 
of a function $u$ with $|Du| < 1$. Hereafter, we shall only consider single valued graphs. 
We treat the problem in a bounded domain $\Omega$ as well as in the entire $\RN$. 
Following the convention in the literature, we say that the graph of $u \in W^{1,\infty}(\Omega)$ is
\begin{itemize}
\item[-] \emph{weakly spacelike} if $|Du| \le 1$ on $\Omega$;
\item[-] \emph{spacelike} if $|u(x)-u(y)| < |x-y|$ whenever $x,y \in \Omega$, $x \neq y$ and the line segment $\overline{xy}$ is 
contained in $\Omega$;
\item[-] \emph{strictly spacelike} if $u \in C^1(\Omega)$ and $|Du|< 1$ in $\Omega$.
\end{itemize}

Given $\phi \in C(\partial \Omega)$, a spacelike hypersurface with Lorentzian mean curvature $\rho$ and boundary (the graph of) $\phi$ is the graph of a solution $u : \overline{\Omega} \to \R$ to 
\begin{equation}\label{borninfeld}\tag{$\mathcal{BI}$}
\left\{ \begin{aligned}
-\diver \left( \frac{Du}{\sqrt{1-|Du|^2}} \right) &= \rho  & &\text{on } \, \Omega \subset \R^m, \\[0.3cm]
u&= \phi  & &\text{on } \, \partial \Omega,
\end{aligned}\right.
\end{equation}
where $D$ and $|\cdot|$ are the connection (gradient) and norm in $\R^m$. If $\Omega = \R^m$, the boundary assumption in \eqref{borninfeld} is replaced by $u(x) \to 0$ as $|x| \to \infty$, 
accomplished by a proper choice of the function space for $u$. 

  The name \eqref{borninfeld} is an acronym for ``Born-Infeld". 
In fact, our second main motivation for studying \eqref{borninfeld} comes from Born-Infeld's theory 
for Electromagnetism, proposed by Born and Infeld in \cite{borninfeld_1,borninfeld_2} 
to overcome the failure of the principle of finite energy occurring in Maxwell's model. 
The reader is referred to the Appendix, where we describe the problem in more detail. 
In the electrostatic case, according to Born-Infeld's theory 
the solution $u_\rho$ to \eqref{borninfeld} represents the electric potential generated 
by a charge $\rho$, and by setting 
	\[
	w_\rho \doteq \frac{1}{\sqrt{1-|Du_\rho|^2}},
	\]
the quantity $w_\rho - 1 + \rho u_\rho$ is the energy density of the field generated by $u_\rho$ 
(with the obvious meaning for $\rho u_\rho$ if $\rho \not \in L^1$). 
The role of $w_\rho$ will be essential in our investigation 
and, hereafter, slightly abusing notation we name $w_\rho$ the \emph{energy density} of $u_\rho$.

By the above description, \eqref{borninfeld} is a nonlinear replacement for Poisson's equation 
	\begin{equation}\label{eq-Po}
		-\Delta u = \rho 
	\end{equation}
occurring in Maxwell's model. For $\rho \in L^1 (\R^3) \cap L^{6/5} (\R^3)$, 
solutions to \eqref{eq-Po} (which decay at infinity) do not violate the principle of finite energy 
since $|Du| \in L^2(\R^3)$, see \cite[Remark 1]{fop}. 
However, this is not the case for general sources including point charges, which are natural to describe electrons 
and motivate the introduction of Born-Infeld's theory. 
Thus, it is meaningful to consider \eqref{borninfeld} for $\rho$ a measure.

%
It was observed in \cite{bartniksimon,ba_1,bpd,mt} that a variational approach to \eqref{borninfeld} by minimizing the action 
	\begin{equation}\label{def_Irho_intro}
	I_\rho(v) \doteq \int_\Omega \left( 1 - \sqrt{1- |Dv|^2} \right) \di x - \la \rho, v\ra
	\end{equation}
($\la \cdot, \cdot \ra$ stands for the duality pairing) may not lead to a solution to \eqref{borninfeld} because of the lack of smoothness of the Lagrangian density when $|Dv| =1$. In particular, the possible presence of light segments in the graph of the (unique) minimizer $u_\rho$ is a serious obstacle. Hereafter, with a slight abuse of notation, we call $\overline{xy} \subset \overline{\Omega}$ a \emph{light segment for $u_\rho$} if 
$x \neq y$ and 
\[
|u_\rho(y)- u_\rho(x)| = |x-y|.
\]
Since $|Du_\rho| \le 1$, in such a case $u_\rho$ is linear with slope $1$ on $\overline{xy}$. To the present, the existence and regularity problem for solutions to \eqref{borninfeld} was considered only for very restricted classes of measures $\rho$. In fact, $L^\infty$ sources and juxtaposition of point charges were treated in bounded domains, while for $\Omega = \R^m$ the case $\rho \in L^q(\R^m)$, $q>m$, was recently included.
Notice that solvability of \eqref{eq-Po} is known at least for
	\[
		\rho = \sum_{i=1}^k a_i \delta_{x_i} + f \quad \text{with} \quad f \in L^q, \ q > 1,
	\]
where $\delta_x$ is the Dirac delta at $x$ and $a_i \in \R$. In this paper, we shall see whether similar properties for \eqref{borninfeld} hold or not. 
Among other results in this paper, we here stress the following:
\begin{itemize}
\item[$(i)$] If $m \ge 3$ and $\rho \in L^2_\loc(\R^m) \cap L^p(\R^m), p \in [1,2m/(m+2)]$, then $u_\rho$ solves \eqref{borninfeld} on $\R^m$ (Theorem \ref{teo_BI_global});
\item[$(ii)$] If $m=2$, $\Omega$ is bounded and 
	\[
		\rho = \sum_{i=1}^k a_i \delta_{x_i} + f \quad \text{with} \quad f \in L^2_{\rm loc}(\Omega) \cap L^1(\Omega), \ \ x_i \in \Omega,
	\]
then  $u_\rho$ solves \eqref{borninfeld} and the graph of $u_\rho$ does not contain light segments. 
Moreover, we may replace point charges by any singular measure $\rho_{\rm S}$ supported in $\Omega$  
with $\haus^1 (\supp \rho_{\rm S}) = 0$ (Theorem \ref{teo_BI_local_s2});
\item[$(iii)$] If $m \ge 3$ and $q < m-1$, there exist $f \in L^q(\R^m)$ and $x, y \in \R^m$ such that the minimizer $u_\rho$ for $\rho = \delta_y - \delta_x + f$ does not solve \eqref{borninfeld}. 
Here the graph of $u_\rho$ has a light segment over $\ol{xy}$ 
(Theorem \ref{teo_nolight_intro} and Corollary \ref{prop_nosolve}). 
\end{itemize}
Towards these goals, we need to develop new tools. Indeed, in the aforementioned cases studied in the literature, the solvability of \eqref{borninfeld} depends on proving that $u_\rho$ is \emph{strictly spacelike} in the domain of consideration. When $\rho \in L^2_\loc$ not even spacelikeness is guaranteed, and in fact compactly supported solutions to \eqref{borninfeld} on $\R^m$, with a light segment and source $\rho \in L^2(\R^m)$, do exist in dimension $m \ge 4$ (cf. Theorem \ref{ex-light-seg_intro} below). Therefore, a core point of our investigation regards the structure of the set of light segments of $u_\rho$ and its interplay with the solvability of \eqref{borninfeld}. The existence of a solution will be obtained through new integral bounds (\lq\lq{}log\rq\rq{} improvement) for the energy density $w_\rho$ and for the second fundamental form of the graph of $u_\rho$. Since an efficient way to localize estimates reveals to be one of the major problems to overcome, the case of bounded $\Omega$ is particularly challenging.


\vspace{0.5cm}

\noindent \textbf{Notation and agreements}.\\
Hereafter, we write $\omega_{m-1}$ for the volume of the unit sphere $\mathbb{S}^{m-1}$, 
and indicate with $\mathbb{1}_A$ the characteristic function of a set $A$. 
The subscript $\delta$ will denote quantities referred to the Euclidean metric on $\R^m$: 
$\di_\delta$ will be the Euclidean distance, $\diam_\delta(E)$ the diameter of a set $E \subset \R^m$ and 
$|\cdot|_\delta, \haus^k_\delta$ the volume and $k$-dimensional Hausdorff measure in $\di_\delta$. 
Given $x,y \in \R^m$ with $x \neq y$, we let $\overline{xy}$ be the closed segment joining $x$ and $y$. 
If $\Omega \subset \R^m$ is an open set, we denote by $\cM(\Omega)$ the set of all finite (signed) Borel measures on $\Omega$ 
equipped with the total variation norm $\| \cdot \|_{\cM (\Omega)}$. 
The set $\lip_c(\Omega)$ will denote the set of Lipschitz functions with compact support in $\Omega$, 
and we write $\Omega' \Subset \Omega$ when $\Omega'$ has compact closure in $\Omega$.

\subsection{Known results for bounded domains}

After works for maximal hypersurfaces $(\rho = 0)$ due to Flaherty \cite{flaherty} and Audounet and Bancel \cite{ab}, 
solutions to \eqref{borninfeld} in bounded domains $\Omega$ and for sources $\rho \in L^\infty(\Omega)$ were studied 
in depth in the influential work by Bartnik and Simon \cite{bartniksimon}. To describe the main result therein, for $\phi \in C(\partial \Omega)$, we define 
	\begin{equation}\label{def_XphiOmega}
	\cX_\phi(\Omega) \doteq \Big\{ \text{$u \in W^{1,\infty}(\Omega)$ : $u$ weakly spacelike, $u = \phi$ on $\partial \Omega$}\Big\}.
	\end{equation}
\begin{remark}\label{rem_defboundary}
We assumed no regularity of $\partial \Omega$, so the boundary condition has to be intended as in \cite{bartniksimon}: $u = \phi$ on $\partial \Omega$ iff, for each $x \in \partial \Omega$ and any straight line $\gamma : (0,1) \to \Omega$ with $\gamma(0^+) = x$, it holds $u(\gamma(t)) \to \phi(x)$ as $t \to 0^+$. In Proposition \ref{teo_compaembe} below, we will prove that this definition suffices to guarantee that functions $u \in \cX_\phi(\Omega)$ can be extended continuously on $\partial \Omega$ with value $\phi$.     
\end{remark}
The class of boundary data for which $\cX_\phi(\Omega) \neq \emptyset$ was characterized in \cite[p. 149]{bartniksimon} in terms of the function 
	\begin{equation}\label{def_domega}
	\di_{\overline{\Omega}}(x,y) \doteq \inf\left\{ \haus^1_\delta(\gamma) \ : \ \gamma \in \Gamma_{x,y} \right\} \le + \infty \qquad \forall \, x,y \in \overline{\Omega},
 	\end{equation}
where 
	\[
	\Gamma_{x,y} = \Big\{ \gamma \in C([0,1], \overline{\Omega}) \ : \ \gamma((0,1)) \subset \Omega, \ \text{ $\gamma$ piecewise affine and $\gamma(0)=x, \gamma(1) = y$} \Big\},
	\]
the infimum is defined to be $+\infty$ if $\Gamma_{x,y} = \emptyset$, and $\gamma$ is called piecewise affine if it consists of finitely many intervals where it is affine. In fact, it is shown in \cite[p. 149]{bartniksimon} that 
	\[
	\cX_\phi(\Omega) \neq \emptyset \qquad \Longleftrightarrow \qquad |\phi(x)-\phi(y)| \le \di_{\overline\Omega}(x,y) \ \ \ \forall \, x,y \in \partial \Omega.
	\]
Note that the restriction $\di_\Omega$ of $\di_{\overline\Omega}$ to $\Omega \times \Omega$ gives the intrinsic metric on $\Omega$.  Remarks on the relation between $\di_{\overline\Omega}(x,y)$ for $x,y \in \partial \Omega$ and the distance in the metric completion of $(\Omega,\di_\Omega)$ will be given in Subsection \ref{subsec_boundary}.

	Next, we introduce a class of weak solutions to \eqref{borninfeld} in bounded domains. 
	\begin{definition}\label{def-weak-sol}
		Let $\Omega$ be a bounded domain in $\R^m$. 
		For $\rho \in W^{1,\infty}(\Omega)^*$, a \emph{weak solution to \eqref{borninfeld}} is a function $u \in \cX_\phi(\Omega)$ such that
		\[
		\begin{array}{ll}
			{\rm (i)} & \disp w \doteq \frac{1}{\sqrt{1-|Du|^2}} \in L^1_\loc(\Omega) \qquad \text{and} \\[0.5cm]
			{\rm (ii)} & \disp \int_\Omega \frac{D u \cdot D\eta}{\sqrt{1-|Du|^2}} \di x = \langle \rho, \eta \rangle 
			\qquad \forall \,  \eta \in	 \lip_c(\Omega).
		\end{array}
		\] 
		Given a subdomain $\Omega' \subset \Omega$, we say that $u$ \emph{weakly solves} \eqref{borninfeld} on $\Omega'$  
		if $w \in L^1_{\rm loc} (\Omega')$ and {\rm (ii)} holds for $\eta \in \lip_c(\Omega')$.
	\end{definition} 

	Equation \eqref{borninfeld} is formally the Euler-Lagrange equation for the functional 
\begin{equation}\label{def_I}
I_\rho \ : \ \cX_\phi(\Omega) \to \R, \qquad I_\rho(v) \doteq \disp \int_\Omega \Big( 1-\sqrt{1-|Dv|^2}\Big) \di x - \langle \rho, v \rangle.
\end{equation}
Although, for $\rho$ lying in a large subset of $W^{1,\infty}(\Omega)^*$, the variational problem for $I_\rho$ admits a unique minimizer $u_\rho$ (cf. Subsection \ref{subsec_functional}), the example of a hyperplane with slope $1$ and $\rho = 0$ indicates that the requirement $\cX_\phi(\Omega) \neq \emptyset$ does not suffice to guarantee that $u_\rho$ solves \eqref{borninfeld} (see Ecker \cite{Ecker}). 
In this respect, note that any solution to \eqref{borninfeld} is easily seen to coincide with the minimizer $u_\rho$ (cf. Proposition \ref{prop_localglobal} below). In \cite[Theorem 4.1 and Corollaries 4.2, 4.3]{bartniksimon}, the authors obtained the following striking result:

\begin{theorem}[\cite{bartniksimon}]\label{teo_bartniksimon_2}
Let $\Omega \subset \R^m$ be a bounded domain, and let $\phi \in C(\partial \Omega)$. The following properties are equivalent:
	\begin{itemize}
	\item[{\rm (i)}] $\phi$ admits a spacelike extension on $\Omega$, that is, there exists $\bar \phi \in \cX_\phi(\Omega)$ which is spacelike on $\Omega$; 
	\item[{\rm (ii)}] $|\phi(x)-\phi(y)| < \di_{\overline\Omega}(x,y) \ $ for every $x,y \in \partial \Omega$, $x \neq y$;
	\item[{\rm (iii)}] for each $\rho \in L^\infty(\Omega)$, there exists $u \in C^1(\Omega) \cap W^{2,2}_{\rm loc}(\Omega)$, 
	which is strictly spacelike and weakly solves \eqref{borninfeld}. 
	\end{itemize}
	\end{theorem}

We therefore define the set
 \begin{equation*}
 	\Spa(\partial \Omega) \doteq \Big\{ \phi \in C(\partial \Omega) \ : \ \text{ any among (i), (ii), (iii) in Theorem \ref{teo_bartniksimon_2} holds} \Big\}.
 \end{equation*}
	
\begin{remark}
No regularity of $\Omega$ is assumed in Theorem \ref{teo_bartniksimon_2}.
This is quite a contrast with the linear problem $-\Delta u = \rho$ in $\Omega$, $u = \phi$ on $\partial \Omega$, for which we need certain regularity properties of $\partial \Omega$, and comes from the strong restriction $u \in W^{1,\infty}(\Omega)$ for \eqref{borninfeld}.
\end{remark}

\begin{remark}
In a broader setting, the equivalence (i) $\Leftrightarrow$ (ii) was  studied in \cite[Theorem 1]{KM93}. 
\end{remark}	

Theorem \ref{teo_bartniksimon_2} does not contain the full generality of the statements in \cite{bartniksimon}. Indeed, under the only assumption $\cX_\phi(\Omega) \neq \emptyset$ the authors showed that the minimizer $u_\rho$ is strictly spacelike on the complement of the set of light segments
\begin{equation}\label{def_lightseg}
K^\rho_\phi \doteq \overline{\bigcup \Big\{ \overline{xy} \ : \ x,y \in \Omega, \ x \neq y, \ \overline{xy} \subset \Omega, \ |u_\rho(x)-u_\rho(y)| = |x-y| \Big\}}, 
\end{equation}
hence it solves \eqref{borninfeld} on $\Omega \setminus K^\rho_{\phi}$. 
A key fact proved in \cite[Theorem 3.2]{bartniksimon} is that when $\rho \in L^\infty(\Omega)$, 
every light segment has to extend up to $\partial \Omega$, a property called there the \emph{anti-peeling theorem}.  	
The proof depends on a comparison argument that is not applicable to more general sources $\rho$, a case for which the relation between singularities of $\rho$ and properties of light segments, including their existence, is currently unknown.

For the study of hypersurfaces with $\rho \in L^\infty$ in more general ambient Lorentzian manifolds, we suggest to consult the works of Gerhardt \cite{G83} and Bartnik \cite{ba_1}. As yet, to the best of our knowledge, more singular $\rho$ on bounded domains were only addressed in a series of works by Klyachin and Miklyukov \cite{KM93,KM95,klyachin_desc}, where the authors treated in depth the case of juxtaposition of point charges. Particularly relevant for us is \cite[Theorem 2]{KM95}, that we rephrase as follows:


\begin{theorem}[\cite{KM95}]\label{thm_KM}
Let $\Omega \subset \R^m$ be a domain such that $(\Omega, \di_\Omega)$ has compact completion, and let $\phi \in \Spa(\partial \Omega)$. Fix a $k$-tuple of points $\mathscr{P} = (x_1,\ldots, x_k) \in \Omega \times \ldots \times \Omega$. Then, there exists a constant $M_m(\phi,\mathscr{P})$ such that, for each $a \doteq (a_1,\ldots, a_k) \in \R^k$ satisfying $|a| < M_m(\phi, \mathscr{P})$, the minimizer $u_\rho$ with source 
	\[
	\rho = \sum_{j=1}^k a_j \delta_{x_j}
	\] 
solves \eqref{borninfeld} and it is strictly spacelike (hence, smooth) on $\Omega \backslash \mathscr{P}$. Furthermore, $M_2(\phi,\mathscr{P}) = +\infty$. 
\end{theorem}

The above result also contains a lower bound for $M_m(\phi,\mathscr{P})$ when $m \ge 3$, which depends on the solution to \eqref{borninfeld} with $\rho = 0$, on $\{x_1,\ldots,x_k\}$ and on the geometry of $\Omega$.

The case $m=2$ is rather special and, indeed, maximal surfaces with singularities in $\LL^3$ 
were also studied from a different point of view by using complex-analytic tools. 
An Enneper--Weierstrass type representation for maximal surfaces was used by Pryce \cite{Pryce} 
to provide explicit examples in some special cases, 
including a maximal surface with two point singularities and, 
notably, one with a light segment $\ol{xy}$ \cite[Example XI]{Pryce}. 
Regarding the latter, it would be very interesting to determine the singular charge $\rho$ with $\supp \rho \subset \ol{xy}$ 
generated by such a solution. 
In \cite{esturome,fls}, the geometric structure of maximal surfaces with singularities was investigated, 
while \cite{koba,uy,fsuy} described in detail classes of maximal surfaces whose singular set is suitably controlled. 
Before stating further known results, 
it should be pointed out that, in the works cited below, 
the authors consider the equation 
\[
(1-|Du|^2)^{3/2} \textrm{div} \left( \frac{Du}{\sqrt{1-|Du|^2}} \right) 
= (1-|Du|^2)^{3/2} \rho 
\]
for which the role of light segments may be different. Examples of maximal surfaces in $\LL^3$ whose singular set contains an entire light line were constructed 
in \cite{various,uy_light,auy}, while an investigation of points at which $Du_\rho$ is light-like can be found in \cite{klyachin_mixed,uy_light_0,uy_light}. The behavior near isolated singularities of surfaces with nonconstant, smooth $\rho$ was characterized in \cite{gjm}. 
To the best of our knowledge, whether or not the singular sets described in the above mentioned references induce a singular measure 
in the mean curvature $\rho$, and which kind of measure, is a problem that has not been considered yet.


\subsection{Our contributions for bounded domains}

From a variational point of view, 
even though the minimizer $u_\rho$ for $I_\rho$ in \eqref{def_I} may not solve \eqref{borninfeld} weakly, 
if $\phi \in \Spa(\partial \Omega)$ then $u_\rho$ enjoys nice properties for each reasonably well-behaved source $\rho$, 
including signed Radon measures. 
Inspired by \cite{bpd}, we prove in Proposition \ref{lem_basicL2} that the energy density of $u_\rho$ is locally integrable, namely 
	\begin{equation}\label{eq_energylocal}
	w_\rho = \frac{1}{\sqrt{1-|Du_\rho|^2}} \in L^1_\loc(\Omega),
	\end{equation}
and in particular $|Du_\rho| < 1$ a.e. on $\Omega$; moreover, 
	\begin{equation}\label{eq_ineq_intro}
	\disp \int_{ \Omega } \frac{ Du_\rho \cdot ( Du_\rho - D \psi ) }{\sqrt{1-|Du_\rho|^2}} \rd x
		\leq \la \rho, u_\rho - \psi \ra \qquad \forall \, \psi \in \cX_\phi(\Omega),
	\end{equation}
where the integrand in the LHS is shown to belong to $L^1(\Omega)$. 
As we shall see in Proposition \ref{prop_localglobal}, $u_\rho$ weakly solves \eqref{borninfeld} if and only if equality holds in \eqref{eq_ineq_intro}, a fact that is not obvious in view of the lack of regularity of $\partial \Omega$ and of $\phi$.

Next, we investigate the relation between the integrability of $\rho$ and the possible existence of a light segment in the graph of $u_\rho$. In Section \ref{sec_counterexamples} (Theorem \ref{ex-light-subsp}), we construct the following example.

\begin{theorem}\label{ex-light-seg_intro}
For each $m \ge 3$ and $\ell \in \{1,\ldots, m-2\}$, there exists a function $u \in C^\infty_c(\R^m)$ with the following properties:	
\begin{itemize}
\item[(i)] the set $K$ of light segments of $u$ is a closed cylinder $\overline{B}^{\ell-1} \times [a,b]$ in a totally geodesic $\ell$-plane of $\R^m$ (in particular, if $\ell = 1$ it is a single light segment), and $|Du|< 1$ on $\R^m \backslash K$; 
\item[(ii)] $u$ satisfies  
			\[
				\int_{\R^m} \frac{Du \cdot D \eta}{\sqrt{ 1 - |Du|^2 }} \, \rd x 
				= \int_{\R^m} \rho_u \eta \, \rd x
				\qquad \forall \, \eta \in \lip_c(\R^m),
			\]
where $\rho_u \in L^q(\R^m)$ for each $q < m-\ell$. In particular, if $\Omega \Subset \Rm$ is a bounded domain containing the support of $u$, then the restriction of $u$ to $\Omega$ weakly solves \eqref{borninfeld} with $\phi \equiv 0$ and $\rho = \rho_u$;
\item[(iii)] for each $q < m-\ell$, it holds
	\[
	w, \ \ \ w |D^2 u|, \ \ \ w^2 |D^2u \left( Du, \cdot  \right)|, \ \ \ w^3 D^2 u \left( Du, Du \right) \quad \in L^q(\R^m),
	\]
where $w = (1-|Du|^2)^{-1/2}$ is the energy density of $u$. 	
	\end{itemize}
\end{theorem}

\begin{remark}\label{rem:smooth-mc-ls}
Notice that $u$ is \emph{smooth}, so the presence of a (maximally extended) light segment entirely contained in $\Omega$ may not prevent the minimizer to be regular. As pointed out by one of the referees, this also occurs when the light segment extends up to $\partial \Omega$: consider the function $u(y,x) = (1-|y|^4)x$ in a small ball $\Omega \subset \R^2$ centered at the origin. Then, $u \in C^\infty(\overline\Omega)$ even though there is a light segment $\gamma$ joining antipodal points of $\Omega$, and by direct computation, if $\Omega$ is small enough, the mean curvature $\rho$ of $u$ on $\Omega \backslash \gamma$ extends smoothly to the entire $\overline{\Omega}$. Furthermore, we may check that $u$ minimizes $I_\rho$ with respect to its own boundary values\footnote{The minimizer $u_\rho$ must coincide with $u$ on $\partial \Omega \cup \gamma$ and be strictly spacelike away from $\gamma$ by \cite[Corollary 4.2]{bartniksimon}, hence $u=u_\rho$ follows by standard comparison on each half-ball determined by $\gamma$.}. However, the energy density $w_\rho$ violates \eqref{eq_energylocal}, so the condition $\phi \in \Spa(\partial \Omega)$ is necessary to get \eqref{eq_energylocal} even when $|Du_\rho|<1$ almost everywhere.
\end{remark}

	Our idea for the construction is as follows. 
Assuming for simplicity $\ell = 1$ and using coordinates $( y , x_m) \in \R^{m -1} \times \R$, 
we first perturb the function $(y,x_m) \mapsto x_m$, whose slope is equal to $1$, 
in the $y$-direction in the following form: 
	\[
		U(y,x_m) = \left( 1 - |y|^2 \right) x_m. 
	\]
The core part of the argument is then to devise a suitable cut-off so that $U$ becomes strictly spacelike outside of a small segment of 
type $\{0\} \times [-t,t]$, it weakly solves \eqref{borninfeld} and the mean curvature of the resulting graph has the desired regularity. 
This is quite subtle and, because of technicalities, the proof is deferred to Section \ref{sec_counterexamples}. 
We think that the above regularity bounds for $\rho$ are sharp, see Subsection \ref{subsec_open} for a list of open problems.

	The existence of a weak solution with a light segment allows us to exhibit what, to our knowledge, 
is the first example of minimizer with spacelike boundary condition that \emph{does not solve \eqref{borninfeld}}, 
even though the corresponding source is rather mild. This is a consequence of the following result, Theorem \ref{teo_nolight} below.

\begin{theorem}\label{teo_nolight_intro}
Let $\Omega \subset \R^m$ be either a bounded domain or $\Omega = \R^m$. In the first case, let $\phi \in \Spa(\partial \Omega)$. Let $u_\rho$ be a minimizer for $I_\rho$ and assume that $u_\rho$ has a light segment $\overline{xy} \subset\Omega$ with $u_\rho(y) - u_\rho(x) = |y-x|$. Then, for each $\alpha > 0$, $u_\rho$ also minimizes the functional $I_{\rho_\alpha}$ with 
	\[
	\rho_\alpha = \rho + \alpha( \delta_y - \delta_x)
	\]
but it does not solve \eqref{borninfeld} weakly for $\rho_\alpha$.
\end{theorem}

Indeed, applying Theorem \ref{teo_nolight_intro} to the example in Theorem \ref{ex-light-seg_intro} with $\ell =1$, we infer

\begin{corollary}\label{prop_nosolve}
Let $m \geq 3$. Then, there exist a smooth domain $\Omega \Subset \R^m$, a function $u \in C^\infty_c(\Omega)\cap \cX_0(\Omega)$, points $x,y \in \Omega$ with $x \neq y$ and a function $\rho_{\mathrm{AC}} \in L^q(\Omega)$ for any $q < m-1$, such that the following properties hold:  
\begin{itemize}
\item[(i)] $\overline{xy}$ is a light segment for $u$, and $|Du|<1$ on $\Omega \backslash \overline{xy}$;
\item[(ii)] $u$ minimizes $I_\rho$ with source 
	\[
	\rho = \alpha(\delta_y - \delta_x) + \rho_{\mathrm{AC}}, \qquad \text{for each fixed } \, \alpha \in \R^+,
	\]
but it does not solve \eqref{borninfeld} weakly. 	 
\end{itemize}	
\end{corollary}

Observe that Corollary \ref{prop_nosolve} makes it impossible to extend Theorem \ref{thm_KM} (i.e. \cite[Theorem 2]{KM95}) for dimension $m \ge 3$ to more general sources of the type
	\[
	\rho = \sum_{j=1}^k a_j \delta_{x_j} + \rho_{\mathrm{AC}} \qquad \text{with } \, \rho_{\mathrm{AC}} \in L^q(\Omega), \ q < m-1.   
	\]
	
We next move to results that guarantee the solvability of \eqref{borninfeld}. 
To get elliptic estimates, our boundary datum shall be restricted to compact subsets $\FF\subset \Spa(\partial \Omega)$ with respect to uniform convergence. 
Examples of $\FF$ include a singleton $\{\phi\}$ and sets of uniformly bounded, $c$-Lipschitz functions on $\partial \Omega$ with respect to $\rd_\delta$, with $c<1$. 
Under the assumption that the metric space $(\Omega,\di_\Omega)$ has compact completion, a more general example will be studied in Subsection \ref{subsec_boundary}.

 
We first consider the $2$-dimensional case.

\begin{theorem}\label{teo_BI_local_s2}
Assume that $\Omega \subset \R^2$ is a bounded domain, and 
let $\Sigma \Subset \Omega$ be a compact subset satisfying $\haus_\delta^1(\Sigma) = 0$. 
Suppose that $\rho \in \cM(\Omega)$ decomposes as 
	\begin{equation*}
	\rho = \rho_{\mathrm{S}} + \rho_{\mathrm{AC}}, \qquad \text{with } \, \left\{ \begin{array}{l}
	{\rm supp } \, \rho_{\mathrm{S}} \subset \Sigma \\[0.2cm]	
	\rho_{\mathrm{AC}} \in L^1(\Omega) \cap L^2_\loc(\Omega \backslash \Sigma).
	\end{array}\right.
	\end{equation*}
Then, 
\begin{itemize}
\item[{\rm (i)}] for each $\phi \in \Spa(\partial \Omega)$, 
the minimizer $u_\rho \in \cX_\phi(\Omega)$ weakly solves \eqref{borninfeld} in $\Omega$ and does not have light segments;
\item[{\rm (ii)}] for any given compact set $\FF \subset \Spa(\partial \Omega)$, 
$\mathcal{I}_1, \mathcal{I}_2, \e >0$, $q_0 \geq 0$, and any given open set $\Omega' \Subset \Omega \backslash \Sigma$ satisfying 
	\[
	\|\rho\|_{\cM(\Omega)} \le \mathcal{I}_1, \qquad \|\rho\|_{L^2(\Omega')} \le \mathcal{I}_2, 
	\]
there exists a constant
$ \mathcal{C} = \mathcal{C}\big(\Omega,\FF, q_0, \diam_\delta(\Omega), \mathcal{I}_1, \mathcal{I}_2, \e, \di_\delta(\Omega',\partial \Omega),\Omega'\big)$
such that, for each $\phi \in \FF$, it holds
	\begin{equation*}
	\begin{aligned}
	&\int_{\Omega'_{\e}} (1+ \log w_\rho)^{q_0} \biggl\{ w_\rho |D^2u_\rho|^2 + w_\rho^3 \left| D^2 u_\rho  \left( D u_\rho, \cdot \right) \right|^2 
	\\
	&\qquad\qquad 
	+ w_\rho^5 \left[D^2u_\rho(Du_\rho,Du_\rho)\right]^2\biggr\}\di x 
	+ \int_{\Omega'_\e} w_\rho (1+ \log w_\rho)^{q_0+1}\di x \le \mathcal{C},
	\end{aligned}
	\end{equation*}
where $\Omega'_{\e} \doteq \{x \in \Omega' : \di_\delta(x, \partial \Omega') > \e\}$;
\item[{\rm (iii)}] 
if $\Omega' \Subset \Omega \setminus \Sigma$ and $\rho \in L^\infty(\Omega')$, then $u_\rho$ is strictly spacelike and $u_\rho \in C^{1,\alpha}_\loc(\Omega')$, for some $\alpha>0$. 
In particular, if $\rho \in C^\infty(\Omega')$ so is $u_\rho$.
	\end{itemize}
\end{theorem}		
\begin{remark}
If $\rho_{\mathrm{S}}$ is a sum of Dirac deltas and $\rho_{\mathrm{AC}} = 0$, 
we recover the result by Klyachin-Miklyukov (see Theorem \ref{thm_KM}). 
However, we stress that our proof is completely different. Indeed, the clever proof in \cite{KM95} is quite specific to Dirac delta singularities, and it seems difficult to extend to sources whose absolutely continuous part is not in $L^\infty$.
\end{remark}
\begin{remark}
Regarding the second order regularity of $u$, for general $\rho$ one cannot expect $u_\rho \in W^{2,q}_\loc$ for $q \ge 1$, 
see the beginning of Subsection \ref{subsec_locSF}. 
\end{remark}

We briefly overview the strategy of the proof, that relies on several steps. 
We refer to $\Omega,\FF,\mathrm{diam}_\delta(\Omega),\mathcal{I}_1, \mathcal{I}_2,\di_\delta(\Omega',\partial\Omega)$ in (ii) as being the \emph{data} of our problem, and fix $\e>0$. 
Hereafter, a constant $\mathcal{C}$ will be assumed to depend on the data. 
We proceed by approximating $\rho$ via convolution to get $\rho_j \rightharpoonup \rho$ 
weakly in $\cM(\Omega)$, let $u_j \in \cX_\phi(\Omega)$ minimize $I_{\rho_j}$ and denote by $w_j \doteq (1-|Du_j|^2)^{-1/2}$ its energy density. First, we show the following two properties:

\begin{itemize}
\item[$(\PP 0_1)$] Proposition \ref{prop_inte_estimate} and Corollary \ref{cor_secondfund} (\textbf{local second fundamental form estimate}): the squared norm of the second fundamental form $\SF_j$ for the graph of $u_j$ over $\Omega$ satisfies
	\[
	\int_{\Omega'_{\e/2}} \|\SF_j\|^2 w_j^{-1} \di x \le \mathcal{C};
	\]
\item[$(\PP 0_2)$] Lemma \ref{lem_growthinte} (\textbf{energy estimate}): 
on Euclidean balls $B_r$ contained in $\Omega_{\e/2}$,
	\[
	\int_{B_r} w_j \di x \le \mathcal{C}r.
	\]
\end{itemize}	
Properties $(\PP 0_1)$ and $(\PP 0_2)$ hold in any dimension $m \ge 2$. 
We stress that, writing $\SF_j$ in terms of $u_j$ as in \eqref{norm_second} in the following, 
$(\PP 0_1)$ implies bounds on the derivative of the energy density $w_j$. For the surface case $m=2$, $(\PP 0_1)$ and $(\PP0_2)$ imply 
\begin{itemize}
\item[$(\PP 1)$] Theorem \ref{teo_higher_m2} (\textbf{higher integrability for $m=2$}): 
	\[
	\int_{\Omega'_\e} w_j \log w_j  \rd x\le \mathcal{C}.
	\]
\end{itemize}
The uniform integrability of $\{w_j\}$ granted by $(\PP 1)$ enables us to show 
\begin{itemize}
\item[$(\PP 2)$] Step 2 in Proof of Theorem \ref{teo_BI_local_s2}
(\textbf{no-light-segment}): $u_\rho$ has no light segments in $\Omega'$ (the statement is quantitative in terms of the data).
\end{itemize}
With the aid of $(\PP 2)$, we can then refine the integral estimates as follows:
\begin{itemize}
\item[$(\PP 3)$] Theorem \ref{teo_higherint} (\textbf{higher integrability and second fundamental form estimates}): for each $q_0\ge 0$,  
	\begin{equation}\label{eq_P3_intro}
	\int_{\Omega'_{\e}} \Big\{ w_j \log w_j + \|\SF_j\|^2 w_j^{-1}\Big\} \log^{q_0} w_j  \di x \le \mathcal{C},
	\end{equation}
where $\mathcal{C}$ also depends on $q_0$ (and on $\Omega'$ in a subtler way). Item (ii) in Theorem \ref{teo_BI_local_s2} follows from \eqref{eq_P3_intro}, which is technically one of the core parts of the paper. It is important to notice that $(\PP 3)$ holds in a given dimension $m$ provided that so does $(\PP 2)$, and in particular, the higher integrability of $w_j$ does not depend on $(\PP 1)$. To the present, we are able to prove $(\PP 2)$ only in dimension $m=2$, and the example in Theorem \ref{ex-light-seg_intro} shows the possible failure of $(\PP 2)$ in dimension $m \ge 4$ when $\rho \in L^2(\Omega')$. 
\end{itemize}
Also, item (iii) in Theorem \ref{teo_BI_local_s2} follows from $(\PP 2)$ by applying arguments in \cite{bartniksimon}. To prove Item (i) we need one last piece of information. Clearly, $(\PP 2)$ and the fact that $\haus^1_\delta(\Sigma) = 0$ guarantee that $u_\rho$ does not have light segments on the entire $\Omega$. However, the  local uniform integrability of $\{w_j\}$ on each $\Omega'\Subset \Omega \backslash \Sigma$ implies
	\[
	\int_\Omega w_\rho Du_\rho \cdot D\eta = \la \rho, \eta \ra \qquad \forall \, \eta \in \lip_c(\Omega \backslash \Sigma). 
	\]
To extend the above identity to test functions $\eta \in \lip_c(\Omega)$, we shall prove the following removable singularity property, which holds in any dimension.
\begin{itemize} 
\item[$(\PP 4)$] Theorem \ref{teo_removable} (\textbf{removable singularity}): if $\{w_j\}$ is locally uniformly integrable on $\Omega \backslash \Sigma$ and $\haus^1_\delta(\Sigma) = 0$, then $u_\rho$ solves weakly \eqref{borninfeld}.
\end{itemize}
%
%
%
%
As we shall see in Remark \ref{rem_sharp_remov}, condition $\haus^1_\delta(\Sigma) = 0$ cannot be weakened to $\haus_\delta^1(\Sigma) < \infty$.

In higher dimensions, the possible failure of $(\PP 2)$ makes it necessary to investigate the set of light segments $K^\rho_\phi$ of $u_\rho$. However, with the aid of the higher integrability Theorem \ref{teo_higherint}, we can still get nontrivial information. Namely, for $\rho,\Sigma$ as in Theorem \ref{teo_BI_local_s2} but in dimension $m \ge 3$, the \emph{only} possible obstruction to the solvability of \eqref{borninfeld} is when $K^\rho_\phi$ is non-empty and touches $\Sigma$ or $\partial \Omega$, as the next Theorem shows.

\begin{theorem}\label{teo_BI_local_sm}
Let $m \geq 3$ and $\Omega \subset \R^m$ be a domain, $\Sigma \Subset \Omega$ be compact and $\rho \in \cM(\Omega$) satisfy  
$\haus_\delta^1(\Sigma) = 0$ and 
	\begin{equation*}
	\rho = \rho_{\mathrm{S}} + \rho_{\mathrm{AC}}, \qquad \text{with } \, \left\{ \begin{array}{l}
	{\rm supp } \, \rho_{\mathrm{S}} \subset \Sigma, \\[0.2cm]	
	\rho_{\mathrm{AC}} \in L^1(\Omega) \cap L^2_\loc(\Omega \backslash \Sigma).
	\end{array}\right.
	\end{equation*}
Given $\phi \in \Spa(\partial \Omega)$, consider the set of light segments $K^\rho_\phi$ of the minimizer $u_\rho \in \cX_\phi(\Omega)$ defined in \eqref{def_lightseg}. Then, 
\begin{itemize}
\item[{\rm (i)}] $u_\rho$ weakly solves \eqref{borninfeld} on $\Omega \backslash K_\phi^\rho$.\\
Moreover, if $K_\phi^\rho \cap (\partial \Omega \cup \Sigma) = \emptyset$, then $u_\rho$ weakly solves \eqref{borninfeld} on the entire $\Omega$.
\item[{\rm (ii)}] For each $\Omega' \Subset \Omega \backslash (\Sigma \cup K^\rho_\phi)$ and $q_0 \ge 0$,
	\[
	\begin{array}{lcl}
	\disp \int_{\Omega'} (1+ \log w_\rho)^{q_0} \left\{ w_\rho |D^2u_\rho|^2 + w_\rho^3 \left| D^2 u_\rho  \left( D u_\rho, \cdot \right) \right|^2 
	+ w_\rho^5 \left[D^2u_\rho(Du_\rho,Du_\rho)\right]^2\right\}\di x \\[0.5cm]
	\disp + \int_{\Omega'} w_\rho (1+ \log w_\rho)^{q_0+1}\di x < \infty.
	\end{array}
	\]
\item[{\rm (iii)}] If $\Omega' \Subset \Omega \setminus (\Sigma \cup K_\phi^\rho)$ and $\rho \in L^\infty(\Omega')$, 
then $u_\rho$ is strictly spacelike and $u_\rho \in C^{1,\alpha}_\loc(\Omega')$, for some $\alpha>0$. 
In particular, if $\rho \in C^\infty(\Omega')$ so is $u_\rho$.
\end{itemize}
\end{theorem}		

\begin{remark}\label{rem_nosolve}
Corollary \ref{prop_nosolve} shows that, in dimension $m \ge 4$, there exists $\rho_{\mathrm{AC}} \in L^2(\Omega)$ and $\rho_{\mathrm{S}} = \delta_y - \delta_x$ such that $u_\rho \in \cX_0(\Omega)$ does not solve \eqref{borninfeld} weakly on the entire $\Omega$. Notice that the support $\Sigma = \{x,y\}$ of $\rho_{\mathrm{S}}$ satisfies $\Sigma \subset K_\phi^\rho$, and therefore  condition $K_\phi^\rho \cap \Sigma = \emptyset$ in (i) of Theorem \ref{teo_BI_local_sm} cannot be removed. 
\end{remark}

\subsection{Known results for $\Omega = \R^m$}

The picture for constant $\rho$ on the entire $\R^m$ is by now well understood. 
Thanks to Calabi \cite{calabi}, Cheng and Yau \cite{chengyau} and Bartnik (Ecker \cite[Theorem F]{Ecker}), 
we know that if $u : \R^m \to \R$ minimizes $I_0$ (i.e. $\rho = 0$) on each open subset $\Omega \Subset \R^m$ 
with respect to compactly supported variations in $\Omega$, then $u$ is a hyperplane, possibly with slope $1$. 
Note that no growth conditions on $u$ are imposed a-priori. 
On the contrary, many examples of smooth spacelike graphs with constant $\rho \neq 0$ were constructed in \cite{stumbles,treibergs}.

	In view of applications to Born-Infeld's theory, we study $I_\rho$ in $\R^m$ with $m \ge 3$ and for functions decaying at infinity to zero, 
taking advantage of the different functional settings described by Kiessling in \cite{Kiessling} 
and Bonheure, d'Avenia and Pomponio in \cite{bpd}. 
For our purposes, we mildly modify their frameworks and define in Subsection \ref{subsec_functional} a Banach space $\cX(\R^m)$ 
in such a way that $I_\rho$ is well defined on 		
	\begin{equation}
	\label{def_cX0}
	\cX_0(\R^m) \doteq \Big\{ v \in \cX(\R^m) \ : \ \|Dv\|_\infty \le 1 \Big\}, 	
	\end{equation}
and so that the latter is closed (and convex) in $\cX(\R^m)$. 
Our choice does not affect the functional properties of $I_\rho$ shown in \cite{bpd}: in particular, following \cite[Lemma 2.2]{bpd}, $I_\rho$ has a unique minimizer $u_\rho \in \cX_0(\R^m)$ which, by \cite[Proposition 2.7]{bpd} (cf. also Proposition \ref{lem_basicL2} herein), satisfies 
	\begin{equation}\label{eq_ham_Rm}
	w_\rho \in L^1_\loc(\R^m), \qquad 0 \le w_\rho -1 \le \frac{|Du_\rho|^2}{\sqrt{1-|Du_\rho|^2}} \in L^1(\R^m) 
	\end{equation}
and the variational inequality 
	\begin{equation}\label{eq_ineq_intro_Rm}
	\disp \int_{\R^m} \frac{ Du_\rho \cdot ( Du_\rho - D \psi ) }{\sqrt{1-|Du_\rho|^2}} \rd x
		\leq \la \rho, u_\rho - \psi \ra \qquad \forall \, \psi \in \cX_0(\R^m).
	\end{equation}
We then say that \emph{$u_\rho$ weakly solves \eqref{borninfeld}} if  
	\[
	\int_{\R^m} \frac{ Du_\rho \cdot D\eta}{\sqrt{1-|Du_\rho|^2}} \rd x
		= \la \rho, \eta \ra \qquad \forall \, \eta \in \lip_c(\R^m).	
	\]
In \cite{bpd}, $u_\rho$ was shown to solve \eqref{borninfeld} weakly whenever $\rho \in \cX(\R^m)^*$ satisfies any of the following assumptions:
%
	\begin{itemize}
	\item[(i)] $\rho$ is radial (\cite[Theorem 1.4]{bpd});
	\item[(ii)] $\rho \in L^\infty_\loc(\RN)$ (\cite[Theorem 1.5]{bpd}). In this case, $u_\rho$ is locally strictly spacelike and thus $u_\rho \in C^{1,\alpha}_\loc(\RN)$ for some $\alpha>0$, by the regularity theory for quasilinear equations. 
	\end{itemize}	
A constructive proof of classical solutions for $\rho \in C^\alpha_0(\RN) \cap L^1(\RN)$, under a smallness condition on $\rho$ in $C^\alpha_0(\RN)$, can be found in \cite{CaKi15} by Carley and Kiessling. 
We next discuss the two further cases considered so far. \\[0.2cm]
\noindent \textbf{The case of point charges.}\\
The problem for
    \begin{equation}\label{pointcharges_intro}
    \rho = \sum_{i=1}^k a_i \delta_{x_i}
    \end{equation}
was treated in \cite{bcf, bpd}: in particular, see \cite[Theorem 1.2]{bcf}, $u_\rho$ was shown to be locally strictly spacelike (hence, smooth) away from the charges $\{x_i\}$ provided that the points $x_i$ are sufficiently far away depending on the sizes $a_i$, in the quantitative way
recalled in Remark \ref{rem_bcf} below.
In this case, $u_\rho$ weakly (indeed, classically) solves \eqref{borninfeld} on $\R^m \setminus \{x_1,x_2,\ldots, x_k\}$. However, in \cite{bcf, bpd} the authors did not prove equality in \eqref{eq_ineq_intro_Rm} for test functions which do not vanish at $x_i$, see \cite[Remark 4.4]{bpd} for more detailed comments.

	In \cite{Kiessling}, Kiessling claimed that for $\rho$ as in \eqref{pointcharges_intro} $u_\rho$ satisfies \eqref{borninfeld} without any restriction on the charges $a_i$, 
but in \cite{bpd} it was pointed out that his argument has a subtle flaw. 
Kiessling subsequently published the erratum \cite{Kiessling}, where he supplied a proof of his claim using a dual approach to circumvent the obstacle in the direct method pointed out in \cite{bpd}. 
From a mathematical point of view it is still desirable to have a proof with a direct use of the functional $I_{\rho}$, though.
%
%
\\[0.2cm]
\noindent \textbf{The case $\rho \in L^q$ for large $q$.}\\
It is natural to seek a sharp condition on $\rho$ that guarantees both the strict spacelikeness of $u_\rho$ and its local $C^{1,\alpha}$ regularity, for some $\alpha \in (0,1)$. 
The evidence coming from the radial case in \cite[Section 3]{bpd}, further motivated by the detailed discussion in the Introduction of \cite{bi_ARMA}, led Bonheure and Iacopetti to formulate the following
	
\begin{conjecture*}[{\cite[Conjecture 1.4]{bi_ARMA}}]
If $m \ge 3$ and $\rho \in \cX^* \cap L^q_\loc(\R^m)$ with $q>m$, then $u_\rho$ is strictly spacelike on $\R^m$ and $u_\rho \in C^{1,\alpha}_\loc(\R^m)$ for some $\alpha \in (0,1)$.
\end{conjecture*}

Here, $\cX^*$ is the dual of a functional space $\cX$ where $\cX_0(\R^m)$ embeds as a closed, convex set, a typical case being  
	\[
	\rho \in L^p(\R^m) \qquad \text{for } \, p \in [1, 2_*],
	\]
where, in what follows, 
	\[
	2_* \doteq \frac{2m}{m+2}
	\]
is the conjugate exponent of the Sobolev one $2^*$. In fact, in the stated assumptions on $\rho$, $C^{1,\alpha}_\loc$ regularity easily follows from strict spacelikeness by standard theory of quasilinear equations. 
 
To the present, a complete answer to the conjecture is still unknown. After a first partial result in \cite{bi_ARMA}, which is in itself remarkable, an almost exhaustive positive answer was given by the combined efforts of Haarala \cite{haarala} and Bonheure--Iacopetti \cite{bi_new}: 
\begin{theorem}[{\cite[Theorem 1.3]{haarala} and \cite[Theorems 1.4 and 1.5]{bi_new}}]
Assume $m \ge 3$ and $\rho \in L^q(\R^m) \cap L^p(\R^m)$ with $p \in [1,2_*]$ and $q>m$. Then, $u_\rho$ is strictly spacelike and 
	\[
	u_\rho \in C^{1,1-\frac{m}{q}}_\loc(\R^m) \cap W^{2,q}_\loc(\R^m).
	\]
Furthermore, $u_\rho$ weakly solves \eqref{borninfeld}.
\end{theorem}	
%
The proof of the theorem is deep, and combines different ingredients that are of independent interest. 
We emphasize that the global $L^q$ integrability of $\rho$ is fundamental at various stages of the proofs in \cite{haarala,bi_new}, 
and hence, the case $\rho \in L^q_\loc(\R^m)$ remains an open problem.

\subsection{Our contributions for $\Omega = \R^m$}
We first address the problem with a superposition of point charges. 
With the aid of Theorem \ref{teo_removable} (removable singularity) and Theorem \ref{teo_higherint} (higher integrability), we can refine the results in \cite{bcf,bpd,Kiessling} and prove
\begin{theorem}\label{teo_bocofo}
Let $\rho$ be as in \eqref{pointcharges_intro}, and let $u_\rho \in \cX(\R^m)$ minimize $I_\rho$. Then, the following hold. 
\begin{itemize}
\item[$(i)$] Light segments of $u_\rho$ cannot cross each other, that is, no point can lie in the interior of two distinct light segments.
\item[$(ii)$] Any maximally extended light segment of $u_\rho$ must be of the type $\overline{x_ix_j}$ for some $i,j$ such that $a_i a_j < 0$. In particular, if $a_ia_j > 0$ for each $i,j$, then $u_\rho$ does not have light segments.
\item[$(iii)$] If $u_\rho$ does not have any light segment, then it weakly solves \eqref{borninfeld} and is smooth, strictly spacelike outside of $\{x_1,\ldots, x_k\}$. Furthermore, around $x_i$, $u_\rho$ is asymptotic to a light cone in the sense of \cite{Ecker}, where the cone is future (respectively, past) pointing provided that $a_i <0$ (respectively, $a_i>0$). 
\end{itemize}
\end{theorem}

	\begin{remark}\label{rem_bcf}
According to \cite[Proof of Theorem 1.2]{bcf}, 
$u_\rho$ has no light segments whenever  
	\begin{equation}\label{eq_quanticharges}
	\left( \frac{m}{\omega_{m-1}} \right)^{\frac{1}{m-1}} \frac{m-1}{m-2} \left[ \left( \sum_{i \in I_-} |a_i| \right)^{\frac{1}{m-1}} + \left(\sum_{i \in I_+} |a_i| \right)^{\frac{1}{m-1}} \right] < \min_{i\neq j} |x_i - x_j|,
	\end{equation}
where $I_+$ ($I_-$) is the set of indices for which $a_i>0$ ($a_i<0$). 
	\end{remark}

Item $(i)$ in Theorem \ref{teo_bocofo} is indeed a general fact that holds for every $u$ with $\|D u\|_\infty \le 1$, see Proposition \ref{prop_nocross}. Item $(ii)$ complements \cite[Corollary 3.2 and Proposition 3.1]{Kiessling}, where $u_\rho$ was assumed a priori to satisfy \eqref{borninfeld}. Notice that our proof is different and 
essentially suggested by one of the referees. The last part of $(iii)$ in Theorem \ref{teo_bocofo} needs some comments, too. In \cite{Ecker}, Ecker defined an \emph{isolated singularity} for 
	\[
	\diver \left( \frac{Du}{\sqrt{1-|Du|^2}} \right) = 0 \qquad \text{on an open set } \, B
	\]
as being a point $x_0 \in B$ such that $u$ minimimizes $I_0$ on any $\Omega' \Subset B \backslash \{x_0\}$ (that is, among functions in $\cX_{u_\rho}(\Omega')$), but not on the entire $B$. 
He then proves in \cite[Theorem 1.5]{Ecker} that an isolated singularity is asymptotic to a future or past pointing light cone centered at $x_0$. As a direct application of Ecker's result, in \cite[Theorem 3.5]{bcf} 
(see also \cite[Theorem 1.5]{bpd}) the authors claim that, for $\rho$ as in \eqref{pointcharges_intro} and $\{x_i\},\{a_i\}$ matching \eqref{eq_quanticharges}, near $x_i$, $u_\rho$ is asymptotic to a light cone which is upward or downward pointing according to whether $a_i<0$ or $a_i>0$. 
However, without knowing the validity of the Euler-Lagrange equation around $x_i$, it is not clear to us how to exclude the possibility that $u_\rho$ also minimizes $I_0$ in a neighborhood of $x_i$. Our statement that $u_\rho$ solves \eqref{borninfeld} in $(iii)$ suffices to guarantee that this does not happen, and therefore to fully justify the conclusions in \cite{bpd,bcf}.\\[0.2cm]
Next, we consider the behavior of $u_\rho$ for sources $\rho \in L^2_\loc(\R^m)$, and obtain the next

\begin{theorem}\label{teo_BI_global}
Let $m \ge 3$ and 
	\[
	\rho \in \big( L^1(\R^m) + L^p(\R^m)\big) \cap L^2_\loc(\R^m), \qquad \text{for some } \, p \in (1,2_*].
	\]
Then, the minimizer $u_\rho$ weakly solves \eqref{borninfeld}. Moreover, for a given $\mathcal{I} \in \R^+$, there exists a positive constant $\mathcal{I}_0 = \mathcal{I}_0(m, p, \mathcal{I})$ 
with the following property: if
	\[
	\|\rho\|_{L^1(\R^m) + L^p(\R^m)} \le \mathcal{I},
	\]
then for any pair of open sets $\Omega'' \Subset \Omega' \Subset \Rm$ 
with $\di_\delta(\Omega'', \partial \Omega') \ge \mathcal{I}_0$, any $\mathcal{I}_2 > 0$ with 
	\[
	\|\rho\|_{L^2(\Omega')} \le \mathcal{I}_2, 
	\]
and any $q_0 \geq 0$, there exists a constant 
$\mathcal{C} = \mathcal{C}(q_0, m, p, \mathcal{I}, \mathcal{I}_0, \mathcal{I}_2, |\Omega'|_\delta)$ such that
	\begin{equation}\label{eq_bonito_intro}
	\begin{array}{lcl}
	\disp \int_{\Omega''} (1+ \log w)^{q_0} \left\{ w_\rho |D^2u_\rho|^2 + w_\rho^3 \left| D^2 u_\rho  \left( D u_\rho, \cdot \right) \right|^2 
	+ w_\rho^5 \left[D^2u_\rho(Du_\rho,Du_\rho)\right]^2\right\}\di x \\[0.5cm]
	\disp + \int_{\Omega''} w_\rho (1+ \log w_\rho)^{q_0+1}\di x \le \mathcal{C}.
	\end{array}
	\end{equation}
\end{theorem}

Some comments are in order. First, the example in Theorem \ref{ex-light-seg_intro} shows that the set of light segments
\begin{equation}\label{def_lightseg_Rm}
K^\rho \doteq \ov{\bigcup \Big\{ \overline{xy} \ : \ x,y \in \R^m, \ x \neq y,  \ |u_\rho(x)-u_\rho(y)| = |x-y| \Big\}}
\end{equation}
may be non-empty, at least if $m \ge 4$. As in the case of bounded domains, the existence of light segments for solutions in $\R^3$ is unknown. Second, the enhanced second fundamental form estimate \eqref{eq_bonito_intro} holds provided that the  inequality
	\begin{equation}\label{eq_weaker}
	\int_{\Omega'} \rho^2 \frac{(1+\log w_\rho)^{q_0+2}}{w_\rho} \di x \le \mathcal{I}_1
	\end{equation}
is satisfied, which is trivially implied by $\rho \in L^2(\Omega')$. Whether \eqref{eq_weaker} may be satisfied by less regular sources $\rho$ is an open problem.

We conclude by studying the case when $\rho \in L^2_\loc$ away from a small compact set, where the singular part of $\rho$ is supported. Then, Theorem \ref{teo_higherint} guarantees that the sole obstruction to the solvability of \eqref{borninfeld} on the entire $\R^m$ may occur when $K^\rho$ intersects the support of the singular measure. 
As we see in Corollary \ref{prop_nosolve} and Remark \ref{rem_nosolve}, 
the case can  happen for certain $\rho = \rho_{AC} + \rho_S$ where $\rho_{AC} \in  L^2_{\loc}$ and 
$\rho_S = \delta_y - \delta_x$ with $m \ge 4,$ 
the corresponding minimizer does not solve \eqref{borninfeld}.

\begin{theorem}\label{teo_BI_global_consing}
Let $m \ge 3$ and let $\Sigma \Subset \Rm$ be a compact set satisfying $\haus_\delta^1(\Sigma) = 0$. Assume  that $\rho$ decomposes as    
	\begin{equation*}
	\rho = \rho_{\mathrm{S}} + \rho_2, \qquad \text{with } \, \left\{ \begin{array}{l}
	\rho_{{\rm S}} \in \cM(\R^m), \  {\rm supp } \,  \rho_{\mathrm{S}} \subset \Sigma, \\[0.2cm]	
	\rho_2 \in \big(L^1(\R^m) + L^p(\R^m)\big) \cap L^2_\loc(\R^m \backslash \Sigma), \ \ p \in (1,2_*],
	\end{array}\right.
	\end{equation*}
Then, the following hold.
\begin{itemize}
\item[{\rm (i)}] The minimizer $u_\rho$ weakly solves \eqref{borninfeld} on $\R^m \backslash K^\rho$, with $K^\rho$ as in \eqref{def_lightseg_Rm}.\\
Moreover, if $K^\rho \cap \Sigma = \emptyset$, then $u_\rho$ weakly solves \eqref{borninfeld} on $\R^m$.
\item[{\rm (ii)}] For each $\Omega' \Subset \R^m \backslash (\Sigma \cup K^\rho)$ and $q_0 \ge 0$, 
	\[
	\begin{array}{lcl}
	\disp \int_{\Omega'} (1+ \log w_\rho)^{q_0} \left\{ w_\rho |D^2u_\rho|^2 + w_\rho^3 \left| D^2 u_\rho  \left( D u_\rho, \cdot \right) \right|^2 
	+ w_\rho^5 \left[D^2u_\rho(Du_\rho,Du_\rho)\right]^2\right\}\di x \\[0.5cm]
	\disp + \int_{\Omega'} w_\rho (1+ \log w_\rho)^{q_0+1}\di x < \infty.
	\end{array}
	\]
\item[{\rm (iii)}] 
If $\Omega' \Subset \Omega \setminus (\Sigma \cup K^\rho)$ and $\rho \in L^\infty(\Omega')$, 
then $u_\rho$ is strictly spacelike and $u_\rho \in C^{1,\alpha}_\loc(\Omega')$, for some $\alpha>0$. 
In particular, if $\rho \in C^\infty(\Omega')$ so is $u_\rho$.
\end{itemize}
\end{theorem}

\subsection{Open problems and outline of the paper}\label{subsec_open}

We first address the existence problem for light segments. We think that the regularity of $\rho_u$ in  Theorem \ref{ex-light-seg_intro} might be sharp, and we are tempted to propose the following 

\begin{conjecture}
If $\phi \in \Spa(\partial \Omega)$ and $\rho \in L^q_\loc(\Omega)$ with $q > m-1$, then the minimizer $u_\rho$ does not have light segments.
\end{conjecture}

The case $q = m-1$, which includes $\rho \in L^2_\loc(\Omega)$ when $m=3$, is particularly subtle. 

\begin{question}\label{question_border}
If $\phi \in \Spa(\partial \Omega)$ and $\rho \in L^{m-1}_\loc(\Omega)$, could the minimizer have light segments?
\end{question}

In view of the techniques developed herein, a negative answer to the above question would be sufficient 
to extend Theorem \ref{teo_BI_local_s2} to dimension $m \ge 3$ and to $\rho_{\mathrm{AC}} \in L^{m-1}_\loc(\Omega \backslash \Sigma)$.

Related to the above problems, and in view of Corollary \ref{prop_nosolve}, we also formulate the following

\begin{question}
If $\phi \in \Spa(\partial \Omega)$ and 
	\[
	\rho = \sum_{i=1}^k a_i \delta_{x_i} + \rho_{\mathrm{AC}} \qquad \text{with } \ \  \rho_{\mathrm{AC}} \in L^q(\Omega), \ q > m-1,
	\]
does the minimizer $u_\rho$ solve \eqref{borninfeld} weakly?
\end{question}

An ambitious goal would be to relate the integrability of $\rho$ to the Hausdorff dimension of the set $K^\rho_\phi$ of light segments. In view of Theorem \ref{ex-light-seg_intro} and of its proof, we may expect that the following holds:
 
\begin{conjecture}
If $m \ge 3$, $\phi \in \Spa(\partial \Omega)$ and $\rho \in L^{q}(\Omega)$ for some $2 \le q \le m-1$, then the Hausdorff dimension of $K_\phi^\rho$ satisfies $\dim_{\haus_{\delta}}(K^\rho_\phi) \le m-q$. 
\end{conjecture}

It might be possible that 
 $\dim_{\haus_{\delta}}(K^\rho_\phi) \le m-q$ could be strengthened to $\haus_{\delta}^{m-q}(K^\rho_\phi) = 0$. If this were true, notice that it would also imply a negative answer to Question \ref{question_border}. If $\rho$ is more singular, we propose the next

\begin{question}
For $\rho \in \cM(\Omega)$, is it true that $\haus_\delta^{m-1}(K_\phi^\rho)=0$ ?  
\end{question}

Still about the set of light segments, it would be important to understand the weak limit 
	\[
	w_j \di x \rightharpoonup \vartheta  \qquad \text{in } \, \cM(\Omega'), \ \ \  \Omega' \Subset \Omega:
	\]
can one characterize the singular part of $\vartheta$, and relate its support to the set $K^\rho_\phi$? Can one characterize the non-negative functional 
	\[
	\la \mathscr{T}, \eta \ra \doteq  \la \rho, \eta \ra - \int_{\Omega} \frac{Du_\rho\cdot D\eta}{\sqrt{1-|Du_\rho|^2}} \qquad \eta \in C^\infty_c(\Omega),
	\]	
describing the loss in \eqref{eq_ineq_intro_Rm}?

Regarding the energy density, we first observe that the integrability of $w_\rho$ in Theorem \ref{ex-light-seg_intro} is much higher than the one that we can prove in Theorem \ref{teo_higherint}. However, the latter is uniform on a sequence of approximated solutions $\{u_{\rho_j}\}$. We can ask the following

\begin{question}
Can one prove a local higher integrability $w_\rho \in L^p_\loc(\Omega)$, for suitable $p>1$, under a local higher integrability of $\rho$, for instance for $\rho \in L^q_\loc(\Omega)$ and $q > m-1$?
\end{question}

Even the case $\rho \in L^q_\loc(\R^m)$ and $q > m$ is currently unknown, cf. \cite{haarala,bi_new}. 

\begin{question}
What about the regularity of $u_\rho$ and $w_\rho$ when $\rho \in L^q$ and $q \in (1,2)$? 
\end{question}

About the higher order regularity for $u_\rho$, $W^{2,q}$ estimates are unknown apart from the case $q=2$, considered in the present paper, and $q>m$ treated in \cite{haarala,bi_new} for $\Omega = \R^m$. We think that there might be an interpolation result, and therefore propose the following

\begin{question}
Can one prove that, for $p \in [2,m]$ and $\rho \in L^p_\loc$, the minimizer $u_\rho$ satisfies $u_\rho \in W^{2,p}_\loc$?
\end{question}

The paper is organized as follows.
Section \ref{sec_prelimgeo} contains some background material from Lorentzian Geometry. 
Section \ref{sec_basic} introduces the functional setting, then moves to discuss the basic properties of $u_\rho$ (convergence under approximation of $\rho$, integrability), together with various equivalent conditions for the solvability of \eqref{borninfeld}. In particular, we mention Propositions \ref{lem_basicL2} and \ref{prop_localglobal}, which may have an independent interest. Though preparatory, most of the material in this section did not appear elsewhere in the literature. In Section \ref{sec-main-tools}, we develop our main new tools: a removable singularity result, 
Theorem \ref{teo_nolight_intro}, a second fundamental form estimate and a higher integrability result. 
These are the bulk of the paper, the techniques therein differ from those in the literature and 
we believe they are applicable beyond the purposes of the present work. Section \ref{prf-thms} then contains the proof of our main existence results. In the concluding Section \ref{sec_counterexamples}, we construct a solution to \eqref{borninfeld} with an $\ell$-dimensional set of light segments and zero boundary condition.

To a certain extent, each of Sections \ref{sec_prelimgeo} to \ref{sec_counterexamples} can be read independently. In particular, the reader acquainted with Lorentzian Geometry and not focusing on the functional analytic setting may directly skip to Section \ref{sec-main-tools}. \\[0.2cm]

\noindent \textbf{A note on constants in elliptic estimates}\\
When constants in our theorems are stated to depend on $\diam_\delta(\Omega)$, $|\Omega'|_\delta$, $\di_\delta(\Omega', \partial \Omega)$, in fact they can be bounded uniformly in terms of, respectively, uniform upper bounds for $\diam_\delta(\Omega)$ and $|\Omega'|_\delta$, and lower bounds for $\di_\delta(\Omega', \partial \Omega)$. Regarding the dependence of $\mathcal{C}$ in Theorem \ref{teo_BI_local_s2} from the domain $\Omega'$ and from $\di_\delta(\Omega',\partial \Omega)$, if $\di_\delta(\Omega', \partial(\Omega\backslash \Sigma))\ge \tau$ and 
	\[
	\|\rho\|_{L^2(U_\tau)} \le \mathcal{I}_2 \qquad \text{where } \, U_\tau = \Big\{ x \in \Omega\backslash \Sigma \ : \ \di_\delta\big(x, \partial (\Omega\backslash \Sigma)\big) \ge \tau \Big\},
	\]
then $\mathcal{C}$ merely depends on $\tau$. On the other hand, anywhere we write $\mathcal{C} = \mathcal{C}(\Omega, \ldots)$ we mean that we did not investigate the stability of the bounds for sequences of open sets $\{\Omega_j\}$ for which the other data are kept uniformly controlled.


\section{Preliminaries from Lorentzian Geometry}\label{sec_prelimgeo}

In this section, we briefly recall some differential-geometric background that will be used henceforth. 
Let $\LL^{m+1}$ be the Lorentz space with coordinates $(x^0,x^1, \ldots, x^m)$ and metric 
	\[
	- \di x^0 \otimes \di x^0 + \sum_{i=1}^m \di x^i \otimes \di x^i, \qquad x \cdot y \doteq -x^0y^0 + \sum_{i=1}^m x^i y^i, \qquad 
	|x|_{\mathbb{L}} \doteq \sqrt{ \left| x \cdot x \right| }.
	\]
Given a smooth function $u : \Omega \subset \R^m \rig \R$, consider the graph map
	\[
	F \ : \ \Omega \rig \LL^{m+1}, \qquad  F(x) \doteq (u(x), x),
	\]
and define $M$ to be the manifold $F(\Omega)$ endowed with the metric induced from $\LL^{m+1}$, equivalently, $M$ is $\Omega$ endowed with the pull-back metric $g \doteq F^*(\cdot)$. 
When convenient, $g$ will also be denoted by $\metric$. Let $\| \cdot \|, \nabla, \Delta_M$ be, respectively, the norm, Levi-Civita connection and Laplace-Beltrami operator associated to $g$. The Hessian of a function $u$ in the metric $g$ will be denoted by $\nabla ^2 u$.  

We identity $\R^m$ with the slice $\{x^0=0\}$, so $\{x^i\}$ are Cartesian coordinates on $\R^m$ with associated vector fields $\{\partial_i\}$. Given an open set $\Omega \subset \R^m$ and $u \in C^\infty(\Omega)$, we let $u_i \doteq \partial_i u$ and $u_{ij} \doteq (D^2 u)_{ij} = \partial^2_{ij}u$. By defining 
$$
X_i \doteq F_* \partial_i = \partial_i + u_i \partial_0, 
$$
the components of $g$ are written as 
$$
g_{ij} \doteq X_i \cdot X_j = \delta_{ij} - u_iu_j.
$$
Hereafter we assume that $g$ is Riemannian (equivalently, $|Du|<1$). The inverse metric has components
$$
g^{ij} = \delta^{ij} + w^2u^iu^j, \qquad \text{with} \quad w \doteq \frac{1}{\sqrt{1-|Du|^2}} ,
$$
where $u^i = \delta^{ij}u_j$ are the components of the gradient $Du$. Then, the volume measure $\di x_g$ of $g$ relates to the measure $\di x$ on $\R^m$ as follows:
\[
\di x_g = w^{-1} \di x.
\]
	The future-pointing, unit normal vector to the graph $M$ is given by $\mathbf{n} \doteq w (\partial_0 + u^i \partial_i)$. Note that ${\bf n} \cdot {\bf n} = -1$ and $w = - {\bf n} \cdot \partial_0$. Let superscripts $\parallel$ and $\perp$ denote, respectively, the projection onto $TM$ and $TM^\perp$ with respect to the inner product $\cdot$ in $\LL^{m+1}$. From the chain of identities
	\[
\langle \partial_0^\parallel, \partial_j \rangle = \partial_0 \cdot F_* \partial_j = - u_j = - \langle \nabla u, \partial_j \rangle, 
	\]
we deduce that
	\begin{equation}\label{partial0parallel}
	\partial_0^\parallel = - \nabla u.
	\end{equation}

	Denoting by $\bar D$ the Levi-Civita connection of $\LL^{m+1}$, we define the second fundamental form of $M$ by   
	\[
	\begin{aligned}
		\disp \SF(\partial_i, \partial_j)  \doteq   \left( \bar D_{X_i} X_j \right)^{\perp} = h_{ij} \mathbf{n}, 
		\qquad 
		\text{thus} \qquad 
		\disp h_{ij} =  - \bar D_{X_i} X_j \cdot {\bf n}  =  \bar D_{X_i} {\bf n} \cdot X_j.
	\end{aligned}
	\]
From the definition of $X_i$ we obtain $h_{ij} = w \, u_{ij}.$ The (unnormalized) scalar mean curvature $H \doteq g^{ij}h_{ij}$ in direction ${\bf n}$ is therefore
$$
H = w\Delta u + w^3 D^2u(Du,Du) 
= \diver \left( \frac{Du}{\sqrt{1-|Du|^2}}\right),
$$
where $\Delta$ is the Laplacian on $\R^m$. Next, since the Christoffel symbols of $g$ are given by $\Gamma^k_{ij} = - w^2 \, u^ku_{ij}$, 
we compute the Hessian and Laplacian of a smooth function $\phi : \Omega \rig \R$ in the graph metric $g$:
	\begin{equation}\label{eq_hessianLapla}
		\begin{aligned}
			& \nabla^2_{ij} \phi = \phi_{ij} + w^2 \, \phi_k u^k u_{ij}; \\
		 	& \Delta_M \phi = g^{ij} \nabla^2_{ij} \phi = \Delta \phi + w^2 D^2\phi(Du,Du) + H w D\phi \cdot Du.
	\end{aligned}
	\end{equation}
In addition, the norm of the second fundamental form $\SF$ of the graph $u$ is given by 
	\begin{equation}\label{norm_second}
		\begin{aligned}
		\|\SF\|^2 & =  \disp g^{ij}g^{kl} h_{ik}h_{jl} = w^2\big( \delta^{ij} + w^2 u^i u^j\big)u_{ik}\big( \delta^{kl} + w^2 u^k u^l\big)u_{jl} \\[0.3cm]
		& =  \disp w^2 |D^2u|^2 + 2 w^4 \left| D^2 u \left( Du, \cdot \right) \right|^2 
		+ w^6 \big[D^2u(Du,Du)\big]^2
		.
		\end{aligned}
	\end{equation}
In particular,
	\begin{equation}\label{eq:1}
		\begin{aligned}
			\nabla_{ij}^2 u = w^2 \, u_{ij} = w \, h_{ij}, \quad 
			\|\nabla^2 u\|^2= w^2\|\SF\|^2, \quad 
			\Delta_Mu = H w \quad {\rm on} \ M.
		\end{aligned}
	\end{equation}

Given $o \in \R^m$, we denote by $r_o : \Omega \to \R$ and $\ell_o : \Omega \to \R$, respectively, the Euclidean distance from $o$ and the Lorentzian distance from $(u(o),o)$ restricted to the graph of $u$, that is, we set 
	\begin{equation}\label{def_ello}
		\begin{aligned}
			r_o(x) & \doteq |x-o|, \\
		l_o(s,x) &\doteq \left| (s,x) - ( u(o),o ) \right|_{\mathbb{L}}
			= \sqrt{ - \left( s - u(o) \right)^2 + \left| x-o \right|^2 }, \\
			\ell_o(x) & \doteq l_o \left( u(x) , x \right). 
		\end{aligned}
	\end{equation}
We also denote the extrinsic Lorentzian ball centered at $o$, and more generally the one centered at a subset $A \subset \R^m$, by 
	\begin{equation}\label{def_LRo}
	L_R(o) \doteq \left\{ x \in \Omega \ : \ \ell_o(x) < R\right\}, \qquad L_R(A) \doteq \bigcup_{o \in A} L_R(o).
	\end{equation}
When it is necessary, we will write $\ell_o^\rho, L_R^\rho$ to emphasize their dependence on the minimizer $u = u_\rho$ of $I_\rho$. By \eqref{eq_hessianLapla}, we get
	\begin{equation}\label{eq_elle2}
	\begin{aligned}
		\bar{D} l^2_o (u(x),x) &= 2 \left( x^j - o^j \right) \partial_j + 2 \left( u(x) - u(o) \right) \partial_0 ;
		\\
		\left\| \nabla \ell_o(x) \right\|^2 &= 
		\left| \bar{D} l_o (u(x),x) \right|^2_{\mathbb{L}} + \left( \bar{D} l_o (u(x),x) \cdot \mathbf{n} \right)^2 \\
		&= 1 + \frac{w^2}{\ell_o^2} \left| Du \cdot (x-o) - \left( u(x) - u(o) \right) \right|^2;
		\\
		\Delta_M \ell^2_o(x)	&= 2m + 2 w H \left[ (x-o) \cdot Du -\left( u(x) - u(o) \right) \right] \\
		& = 2m +  H \left( \bar{D} l_o^2 (u(x),x) \cdot \mathbf{n} \right). 
	\end{aligned}
	\end{equation}

As we shall see in the proof of Theorem \ref{teo_higherint}, 
the construction of cut-off functions based on the Lorentzian distance, instead of those based on the Euclidean one, will be the key to obtain the higher integrability of $u_\rho$ in dimension $m \ge 3$.

\section{Basic properties of $u_\rho$} \label{sec_basic}

In this section, we obtain basic properties of the minimizer $u_\rho$ of $I_\rho$, 
both for $\Omega \subset \Rm$ a bounded domain ($m\geq 2$) and for $\Omega =\R^m$ ($m \geq 3$).

\subsection{Functional setting} \label{subsec_functional}

We first choose our functional spaces. If $\Omega = \R^m$, our treatment mildly departs from those in \cite{Kiessling,bpd}, and is basically designed to get an explicit description of the sources $\rho$ covered by the method. On the other hand, for bounded $\Omega$, subtleties related to a possibly rough boundary $\partial \Omega$ require extra care in the choice of the functional space, which significantly differs from that in \cite{bartniksimon}.
 
	\begin{definition}\label{def-calX}
		Given $m \ge 2$, we fix $p_1 \in (m,\infty)$ and assume also $p_1 \ge 2^*$ for $m = 3$. 
		\begin{enumerate}
		\item[{\rm (i)}] 
		When $m \geq 2$ and $\Omega \subset \Rm$ is a bounded domain, 
		we set 
		\[
		\cX (\Omega) \doteq W^{1,p_1} (\Omega) \cap C(\overline\Omega), \qquad \| v\|_{\cX} \doteq  \max\Big\{ \| v\|_{W^{1,p_1}(\Omega)}, \|v\|_{C(\overline\Omega)}\Big\};
		\]
		\item[{\rm (ii)}]
		When $\Omega = \Rm$ and $m \geq 3$, we set 
			\[
			\cX(\Rm) \doteq \overline{C^\infty_c(\R^m)}^{\|\cdot\|_\cX}, \qquad 
				\| v \|_{\cX} \doteq \max\Big\{ \| Dv \|_2, \| Dv \|_{p_1}\Big\}. 
			\]
		\end{enumerate}
	\end{definition}

Note that, if $\Omega$ is bounded and sufficiently regular (Lipschitz is enough), by Morrey's Embedding Theorem $\cX(\Omega)=W^{1,p_1}(\Omega)$ with the equivalent norm $\|\cdot\|_{W^{1,p_1}(\Omega)}$. 
 

	\begin{remark}\label{rem-entire-m2}
		The case $\Omega = \R^2$ will not be considered in the present paper. 
		We observe that the radially symmetric solution in \cite{borninfeld_2} with a Dirac delta source (cf. \eqref{e:rad-example})
		has a logarithmic behavior at infinity when $m=2$, 
		which calls for a different functional setting. 
		For $\rho$ a superposition of point charges, complete classification theorems for entire solutions in $\R^2$ were obtained by Klyachin \cite{klyachin_desc}, and Fern\'andez, L\'opez and Souam \cite{fls}.
	\end{remark}

The following result can be proved in a similar way as \cite[Lemma 2.1]{bpd}, but we give full details for the sake of completeness. 

	\begin{proposition}\label{prop-cX-emb}
		Assume $m \geq 3$ and $\Omega = \Rm$. Then $(\cX(\R^m), \|\cdot \|_\cX)$ is a reflexive Banach space. Moreover, 
	\begin{equation}\label{eq_embe}
	\cX(\Rm) \hookrightarrow W^{1,q} (\Rm) \qquad \forall \, q \in [ 2^\ast , p_1 ]. 
	\end{equation}
In particular, $\| \cdot \|_{\cX}$ is equivalent to $\| D \cdot \|_2 + \| \cdot \|_{W^{1,p_1}}$, and $\cX(\Rm) \hookrightarrow C_0(\Rm) \doteq \{ u \in C(\Rm)  : \lim_{|x| \to \infty} u(x) = 0 \}  $ holds. 
	\end{proposition}

	\begin{proof}
First, $\|\cdot \|_{\cX}$ is equivalent to the norm $|u|_{\cX} \doteq \sqrt{ \|Du\|_2^2 + \|Du\|_{p_1}^2}$. Hence, to prove the reflexivity of $(\cX(\R^m),\|\cdot \|_\cX)$ it suffices to show that $(\cX(\R^m), |\cdot |_{\cX})$ is uniformly convex. 
This easily follows by using the criterion in \cite[Exercise 3.29]{Br} and the uniform convexity of the norms $\|Du\|_2$ and $\|Du\|_{p_1}$.

	To obtain \eqref{eq_embe}, let $u \in \cX(\R^m)$. From the choice of $p_1$ and H\"older's inequality, the next interpolation inequality holds: 
	\begin{equation}\label{ineq-Du}
		\| Du \|_q \leq \| u \|_{\cX} \quad \text{for all $q \in [2, p_1]$}.
	\end{equation}
Since $m \in [2,p_1)$ and $q^\ast \to \infty$ as $q \to m^{-}$, there exists $\hat{q} \in [2,m)$ so that
$\hat{q}^\ast = p_1$. Thus, Sobolev's inequality and \eqref{ineq-Du} yield 
$\| u \|_{p_1} \leq C \| Du \|_{\hat{q}^\ast} \leq C \| u \|_{\cX}$. 
Hence, $\cX (\Rm) \hookrightarrow W^{1, p_1} (\Rm)$ holds. 
In addition, from $\| u \|_{2^\ast} \leq C \| Du \|_{2} \leq \| u \|_{\cX}$, $2 < 2^\ast \leq p_1$ and \eqref{ineq-Du}, 
we see $\cX(\Rm) \hookrightarrow W^{1,2^\ast} (\Rm)$. 
Therefore, by the interpolation, \eqref{eq_embe} holds.

	The equivalence between $\|\cdot\|_\cX$ and $\|D \cdot\|_2 + \|\cdot\|_{W^{1,p_1}}$ is 
an immediate consequence of \eqref{eq_embe}, 
while $\cX(\Rm) \hookrightarrow C_0(\Rm)$ follows from Morrey's embedding Theorem 
once we observe that $u \in L^{2^*}(\R^m) \cap C^{0,\alpha}(\R^m)$ implies that $u$ vanishes at infinity.
\end{proof}	

\begin{remark}[\textbf{Dual spaces}]\label{rem_duals}
If $q \in (1,\infty)$ and $\Omega \subset \R^m$ is any domain, then it is well-known
that elements in the dual space $W^{1,q}(\Omega)^* = W^{-1,q'}(\Omega)$ can be represented 
as pairs $(v,V) \in L^{q'} (\Omega) \times [L^{q'} (\Omega)]^m  $ where $q' \doteq q/(q-1)$, with the action 
	\[
	\langle \rho, \psi \rangle \doteq \int_\Omega \psi v \di x + \int_\Omega D\psi \cdot V \di x \qquad \forall \, \psi \in W^{1,q} (\Omega),
	\]
see for instance \cite[Theorem 3.9]{AF03}. Furthermore, recall that if $X_1,X_2$ are Banach spaces with $X_1 \cap X_2$ dense in $X_1$ and $X_2$, then $(X_1 \cap X_2)^* = X_1^* + Y_2^*$ with the natural norm
	\[
	\|\rho\|_{X_1^* + X_2^*} = \inf \Big\{ \|\rho_1\|_{X_1^*} + \|\rho_2\|_{X_2^*} : \rho_j \in X_j^*, \ \rho = \rho_1 + \rho_2\Big\},
	\] 
see \cite[Theorem 2.7.1]{BL76}. Indeed, inspecting the proof in \cite{BL76}, one deduces that every functional $\rho \in (X_1\cap X_2)^*$ can be represented as 
	\[
	\rho = \rho_1 + \rho_2 \in X_1^* + X_2^*, \qquad \text{with} \qquad \|\rho_1\|_{X_1^*} + \|\rho_2\|_{X_2*} \le \|\rho\|_{(X_1 \cap X_2)^*}, 
	\]
the representation being unique (with equality between norms) when $X_1 \cap X_2$ is dense in both $X_1$ and $X_2$. Taking the above observations into account,   
%
%
\begin{itemize}
\item[(i)] if $\Omega$ is a bounded domain, every $\rho \in \cX(\Omega)^*$ can be represented as $\rho = \rho_1 + \rho_2 \in W^{-1,p_1'}(\Omega) + \cM(\Omega)$, for some $\rho_1,\rho_2$ satisfying 
	\[
	\|\rho_1\|_{W^{-1,p_1'}} + \|\rho_2\|_{\cM} \le \|\rho\|_{\cX^*}.
	\]
The representation is unique when $C(\overline\Omega) \cap W^{1,p_1}(\Omega)$ is dense in $W^{1,p_1}(\Omega)$, a fact which entails some mild requirement on $\partial \Omega$ such as the segment condition (cf. \cite[Theorem 3.22]{AF03}). However, uniqueness of the representation will not be used in the present work. Notice the continuous inclusion $\cM(\Omega) \hookrightarrow \cX(\Omega)^*$.

\item[(ii)] if $\Omega = \R^m$ and $m \ge 3$, then $\cX(\Rm)^\ast = \mathcal{D}^{1,2} (\Rm)^\ast + W^{-1,p_1'} (\Rm)$, with $\mathcal{D}^{1,2} (\R^m )$ being the closure of $C^\infty_c(\R^m)$ with respect to the norm $\| v\|_{\mathcal{D}^{1,2}} \doteq \|Dv\|_{2}$. In particular, because of Proposition \ref{prop-cX-emb} and Morrey's embedding, $\cM(\Rm) \hookrightarrow \cX(\Rm)^\ast$ and $W^{-1,q'}(\R^m) \hookrightarrow \cX(\R^m)^*$ for each $q \in [2^*,p_1]$. Hence,  
	\begin{equation*}
	\cM(\R^m) + L^{q'}(\R^m) \hookrightarrow \cX(\R^m)^* \qquad \forall \, q \in [2^*,p_1],
	\end{equation*}
where $L^{q'}(\R^m)$ consists of the pairs $(v,0)$. 
\end{itemize}
\end{remark}

Clearly, $\cX_0(\R^m)$ in \eqref{def_cX0} is a closed convex subset of $\cX(\R^m)$. The situation is more subtle for $\cX_\phi(\Omega)$ defined in \eqref{def_XphiOmega}, because of the lack of regularity of $\partial \Omega$. However, as the next result shows, the mild sense in which the boundary condition is considered, see Remark \ref{rem_defboundary}, suffices to guarantee that $\cX_\phi(\Omega) \subset \cX(\Omega)$.

\begin{proposition}\label{teo_compaembe}
Let $\Omega \subset \R^m$ be a bounded domain, let $\mathscr{F} \subset C(\partial \Omega)$ be a relatively compact (resp. compact) subset with respect to uniform convergence, and consider 
	\[
	\cX_{\mathscr{F}}(\Omega) \doteq 
	\left\{ v  \ : \  v \in \cX_\phi (\Omega) \  \text{for some $\phi \in \mathscr{F}$} \right\}.
	\]
Then $\cX_{\mathscr{F}}(\Omega) \subset C(\overline\Omega)$ as a relatively compact (resp. compact) subset, where 
we extend each $v \in \cX_{\mathscr{F}}(\Omega)$ onto $\ov{\Omega}$ by setting $v(x) \doteq \phi(x)$ for $x \in \partial \Omega$.
\end{proposition}	

\begin{proof}
First, observe that if $x \in \Omega$ and $\wt{x} \in \partial \Omega$ is a nearest point to $x$ in the metric $\di_\delta$, the boundary condition in Remark \ref{rem_defboundary} tested on the segment $t x + (1-t) \wt{x} \in \Omega$ for any $t \in (0,1]$ gives, for each $v \in \cX_{\FF}(\Omega)$,  
	\begin{equation}\label{eq_segme}
		\left| v(x) - \phi ( \wt{x}) \right| = \left| v(x) - \lim_{t \to 0^+} v( tx + (1-t) \wt{x} ) \right| 
		\leq \left| x - \wt{x} \right|.
	\end{equation}
The inequality trivially holds also if $x \in \partial \Omega$, by the way $v$ is extended. Whence, 
	\begin{equation}\label{eq_Linftybound}
	\|v\|_{L^\infty(\Omega)} \le \|\phi\|_{C(\partial \Omega)} + \diam_\delta(\Omega) 
	\leq \sup_{ \phi \in \scrF } \| \phi \|_{C(\partial \Omega)} + \diam_{\delta} (\Omega) < \infty,
	\end{equation}
where the last inequality follows since $\scrF$ is relatively compact in $C(\partial \Omega)$. This proves the uniform boundedness of $\cX_{\mathscr{F}}(\Omega)$. 

Next, we shall show $v \in C(\ov{\Omega})$ for each $v \in \cX_{\mathscr{F}} (\Omega)$, and that $\cX_{\mathscr{F}}(\Omega)$ is uniformly equicontinuous. Let $\e > 0$ be arbitrary. 
Since $\mathscr{F}$ is relatively compact in $C(\partial \Omega)$, $\mathscr{F}$ is uniformly equicontinuous on $\partial \Omega$, hence, 
there exists $\wt{\delta}_\e > 0$ such that 
	\[
		\phi \in \scrF, \ x_1,x_2 \in \partial \Omega, \ |x_1-x_2| < \wt{\delta}_\e \quad \Rightarrow \quad 
		\left| \phi(x_1) - \phi (x_2) \right| < \frac{\e}{4}.
	\]
Set 
	\[
		\delta_\e \doteq \frac{1}{4} \min \left\{ \e , \wt{\delta}_\e \right\} > 0,
	\]
and pick $x_1,x_2 \in \overline\Omega$ with $|x_1-x_2|< \delta_\e$. If one among $B_{\delta_\e}(x_1)$ and $B_{\delta_\e}(x_2)$ is contained in $\Omega$, property $v \in\cX_{\phi} (\Omega)$ implies that $v$ is $1$-Lipschitz there, whence  
	\[
		|x_1-x_2| < \delta_\e \quad \Rightarrow \quad \left| v(x_1) - v(x_2) \right| \leq \left| x_1 - x_2 \right| < \delta_\e < \e. 
	\]
We therefore assume that $B_{\delta_\e}(x_j) \cap \partial \Omega \neq \emptyset$ for $j = 1,2$, and choose $\wt{x}_j \in B_{\delta_\e}(x_j) \cap \partial \Omega$ satisfying $|x_j-\wt{x}_j| = \di_\delta (x_j, \partial \Omega )$. From $|x_1-x_2| < \delta_{\e}$ and $|x_j-\wt{x}_j| < \delta_\e$ for each $j$, the triangle inequality implies $|\wt{x}_1-\wt{x}_2| < 3 \delta_\e < \wt{\delta}_\e$ and therefore, by using \eqref{eq_segme},  
	\[
		\begin{aligned}
			\left| v (x_1) - v(x_2) \right| 
			&\leq \left| v(x_1) - \phi (\wt{x}_1) \right| + \left| \phi(\wt{x}_1) - \phi (\wt{x}_2) \right| +  \left| \phi(\wt{x}_2) - v(x_2) \right| 
			\\
			&\leq \left| x_1 - \wt{x}_1 \right| + \frac{\e}{4} + \left| \wt{x}_2 - x_2 \right| < 2\delta_\e + \frac{\e}{4} \leq \e.
		\end{aligned}
	\]
Hence, $v \in C(\ov{\Omega})$ and $\cX_{\mathscr{F}} (\Omega)$ is uniformly equicontinuous on $\ov{\Omega}$. The relative compactness of $\cX_{\mathscr{F}}(\Omega)$ in $C(\overline\Omega)$ follows by the Arzel\'a--Ascoli theorem. If $\mathscr{F}$ is compact, then any limit point of a sequence $\{v_j\} \subset \cX_{\mathscr{F}}(\Omega)$ lies in $\cX_{\mathscr{F}}(\Omega)$ since the limit function is $1$-Lipschitz, 
thus $\cX_{\mathscr{F}}(\Omega)$ is compact in $C(\overline\Omega)$. 
\end{proof}

\begin{corollary}\label{cor_closedconvex}
For each bounded domain $\Omega \subset \R^m$ and each $\phi \in C(\partial \Omega)$, $\cX_\phi(\Omega) \subset \cX(\Omega)$ and it is bounded, closed, convex and sequentially weakly compact in $\cX(\Omega)$.
\end{corollary}

\begin{proof}
By Proposition \ref{teo_compaembe}, $\cX_\phi(\Omega) \subset C(\overline\Omega)$ is a compact subset. 
From $\| Du \|_\infty \leq 1$ for any $u \in \cX_\phi (\Omega)$, we deduce that $\cX_\phi (\Omega)  \subset \cX(\Omega)$ is closed, bounded and convex. 
To prove the sequential weak compactness, let $\{v_j\}$ be a sequence in $\cX_\phi(\Omega)$. Then, up to passing to a subsequence, $v_j \to v$ weakly in $W^{1,p_1}(\Omega)$ and strongly in $C(\overline\Omega)$, for some $v \in \cX(\Omega)$. By Remark \ref{rem_duals}, we can represent a given $\rho \in \cX(\Omega)^*$ as $\rho = \rho_1 + \rho_2$ with $\rho_1 \in W^{-1,p_1'}(\Omega)$ and $\rho_2 \in \cM(\Omega)$, whence
	\[
	\la \rho, v_j \ra = \la \rho_1,v_j \ra + \la \rho_2, v_j \ra  \to \la \rho_1,v \ra + \la \rho_2, v \ra = \la \rho,v \ra \qquad \text{as } \, j \to \infty,  
	\]
thus $\{v_j\}$ is weakly convergent. Since $\cX_\phi(\Omega)$ is weakly closed due to the convexity and closedness, 
we see $v \in \cX_\phi (\Omega)$ and complete the proof. 
\end{proof}

Regarding the minimization problem, for the readers' convenience we reproduce the argument in \cite{bpd} to show the existence and uniqueness of the minimizer $u_\rho$ in our functional setting. For $\rho \in \cX(\Omega)^\ast$, we recall that $I_\rho: \cX_{\phi}(\Omega) \to \R$ is defined by
	\[
		I_\rho (v) 
		\doteq 
		\int_{ \Omega } \left( 1 - \sqrt{1 - \left| Dv \right|^2} \right) \rd x - \la \rho , v \ra 
		\quad \text{for $v \in \cX_{\phi} (\Omega)$}.
	\]
The above discussion guarantees that $\cX_{\phi} (\Omega)$ is a closed convex subset of $\cX(\Omega)$ (when $\Omega$ is bounded, we  suppose that $\phi \in C(\partial \Omega)$ is chosen such that $\cX_\phi(\Omega) \neq \emptyset$), and $I_\rho$ is strictly convex 
since $\ov{B_1(0)} \ni p \mapsto 1 - \sqrt{1-|p|^2} \in [0,1]$ is strictly convex. 
Furthermore, from the inequality $1 - \sqrt{1-|p|^2} \leq |p|^2$ for $|p| \le 1$ and 
using Lebesgue's dominated convergence theorem, $I_\rho$ is continuous on $\cX_\phi(\Omega)$. Combining convexity and continuity, we deduce that $I_\rho$ is weakly lower-semicontinuous. 
If $\Omega$ is a bounded domain, by Corollary \ref{cor_closedconvex} the set $\cX_\phi(\Omega)$ is bounded and sequentially weakly compact in $\cX(\Omega)$, so the existence of a minimizer is then obvious by the direct method. On the other hand, if $\Omega = \R^m$, then 
$\| D v \|_q^q \leq \| D v \|_2^2$ holds for every $v \in \cX_0(\Rm)$ and $q \in [2,\infty)$ 
thanks to $\| D v \|_\infty \leq 1$. 
Thus, in view of the identity
 \begin{equation}\label{1}
 1 - \sqrt{1-t} = \sum_{j=1}^\infty b_j t^j \qquad \text{with } \  b_j \doteq \frac{(2j-2)!}{j!(j-1)!2^{2j-1}}, \ \ t \in [0,1],
 \end{equation}
it follows from $3 \leq m < p_1$ that for $v \in \cX_0(\R^m)$, 
  	\begin{equation}\label{coercive}
		\begin{aligned}
		\|v\|_{\cX}^2 
		& \le 
		\left( \|Dv\|_2^2 + \|Dv\|_{p_1}^2 \right) 
		\le 
		\left( \|Dv\|_2^2 + \|Dv\|_2^{4/p_1} \right) 
		\\
		& \le 2 \left( \|Dv\|_2^2 + 1 \right) \le 
		2 \left[ 1 + b_1^{-1} \left( I_\rho(v) + \|\rho\|_{\cX^*} \|v\|_{\cX} \right) \right].
  		\end{aligned}
	\end{equation}
Hence, $I_\rho$ is coercive. Since $\cX(\R^m)$ is reflexive, the existence and uniqueness of $u_\rho$ is then a consequence, for instance, of \cite[Corollary 3.23]{Br}.


\subsection{Large compact subsets of $\Spa(\partial \Omega)$}\label{subsec_boundary}
%
%
Assume that $(\Omega,\di_\Omega)$ has compact metric completion, that following \cite{KM95} we denote by $\Omega_\di$. We set $\partial \Omega_{\di} = \Omega_{\di} \backslash \Omega$. To stress the difference with $\rd_{\ov{\Omega}}$ in \eqref{def_domega}, we write $\di$ instead of $\di_\Omega$ for the metric on $\Omega_\di$. We first recall some topological facts. The identity $i : (\Omega, \di_\Omega) \to (\Omega, \di_\delta)$ extends by density to a distance non-increasing map $\widetilde i : (\Omega_\di,\di) \to (\overline\Omega, \di_\delta)$. Since $\Omega_\di$ is compact and $(\overline\Omega, \di_\delta)$ is Hausdorff, 
$\widetilde{i}$ is a closed map. From $\widetilde{i}(\Omega_\di) \supset \Omega$, 
we deduce that $\widetilde{i}$ is also surjective, hence, $\wt{i}$ is a quotient map. 
Given $\phi \in C(\partial \Omega)$, let $\widetilde{\phi} = \phi \circ \widetilde{i} \in C(\partial \Omega_\di)$ be its lift. 

Let $b \in \R^+$ and $\zeta : \R^+ \to [0,1)$, and define the following subset of $\Spa(\partial \Omega)$:
		 	\begin{equation}\label{def_Spabzeta}
		 	\mathcal{S}_{b,\zeta}(\partial \Omega) \doteq \Big\{ \phi \in \mathcal{S}(\partial \Omega) \ : \ \|\phi\|_\infty \le b, \ \ \sup_{ \footnotesize \begin{array}{c}
		 	x,y \in \partial \Omega_\di, \\
		 	{ \di(x,y)} = t \end{array} } \frac{|\wt\phi(x)- \wt\phi(y)|}{\di(x,y)} \le \zeta(t) \ \ \forall \, t \in \R^+ \Big\}, 
		 	\end{equation}
		where the supremum is defined to be zero if 
		$t > \diam_{\rd_\Omega}(\Omega)$. 
We prove that $\Spa_{b,\zeta}(\partial \Omega)$ is compact in $C(\partial \Omega)$, 
so let $\{\phi_j\} \subset \Spa_{b,\zeta}(\partial\Omega)$. 
By the Arzel\'a--Ascoli Theorem, $\{\widetilde{\phi}_j\}$ is relatively compact in $C(\partial \Omega_\di)$ 
and thus, up to subsequences, $\widetilde{\phi}_j \to  \widetilde{\phi}$ 
for some $ \widetilde{\phi} \in C(\partial \Omega_\di)$ 
which is constant on the fibers of $\wt{i}$, and therefore factorizes as $ \widetilde{\phi} = \phi \circ \wt{i}$.  
Since $\wt{i}$ is a quotient map, $\phi \in C(\partial \Omega)$ 
(see, for instance, \cite[Theorem 22.2]{Mu00}). 
From $\wt{\phi}_j \to  \widetilde{\phi}$ on $\partial \Omega_d$, 
we deduce that $\phi_j \to \phi$ on $\partial \Omega$ and 
$\phi$ satisfies the last two conditions in \eqref{def_Spabzeta}. 
To show that $\Spa_{b,\zeta}(\partial \Omega)$ is compact in $C(\partial \Omega$), 
it suffices to prove that $\phi \in \Spa(\partial \Omega)$. 
Suppose by contradiction that $\phi \not \in \Spa(\partial \Omega)$, 
and take $x,y \in \partial \Omega$, $x \neq y$ such that $|\phi(x) - \phi(y)| \ge \di_{\overline\Omega}(x,y)$. 
Then, being the left-hand side finite, $\Gamma_{xy} \neq \emptyset$ and 
we can lift the interior of any path $\gamma \in \Gamma_{xy}$ to a path $\wt \gamma : (0,1) \to \Omega_\di$ of the same length of $\gamma$, with $\wt \gamma((0,1)) \subset \Omega$. 
Choose paths $\gamma_\eps \in \Gamma_{x,y}$ 
with $\haus^1_\delta (\gamma_\eps) \downarrow \di_{\overline\Omega}(x,y)$ as $\eps \downarrow 0$. 
It is easy to check that $\wt \gamma_\eps(0^+) \doteq \wt x_\eps \in \wt{i}^{-1}(x)$ and 
$\wt \gamma_\eps(1^-) \doteq \wt y_\eps \in \wt{i}^{-1}(y)$. 
Since the fibers $\wt{i}^{-1}(x)$ and $\wt{i}^{-1}(y)$ are compact, up to subsequences 
$\wt{x}_{\e_k} \to \wt{x} \in \wt{i}^{-1} (x)$ and $\wt{y}_{\e_k} \to \wt{y} \in \wt{i}^{-1} (y)$. 
By $x \neq y$, we have 
$ 0< \di(\wt{x},\wt{y}) = \lim_{k \to \infty} \di (\wt{x}_{\e_k} , \wt{y}_{\e_k}  ) \le \di_{\overline\Omega}(x,y)$. 
However, from the last property in \eqref{def_Spabzeta} for $\wt{\phi}_j$, we get the following contradiction:
	\[
		\begin{aligned}
			\di (\wt{x}, \wt{y}) 
			\leq \di_{\overline{\Omega}} (x,y) \le \left| \phi(x) - \phi(y) \right| 
			= \left| \wt{\phi} (\wt{x}) - \wt{\phi} (\wt{y}) \right| 
			&= \lim_{j \to \infty} \left| \wt{\phi}_j(\wt{x}) - \wt{\phi}_j(\wt{y}) \right| 
			\\
			&\leq \zeta \left( \di (\wt{x}, \wt{y}) \right) \di (\wt{x}, \wt{y})
			< \di (\wt{x}, \wt{y}).
		\end{aligned}
	\]

\subsection{Three useful lemmas}

We shall introduce three lemmas that will be repeatedly used. First, we search for an efficient way to approximate $\psi \in \cX(\Omega)$ by a sequence in $\cX(\Omega)$ having suitably prescribed boundary values.

\begin{lemma}[Approximation]\label{prop_approx}
Let $\Omega \subset \R^m$ be a domain, let $u, \psi : \Omega \to \R$ and for $\e>0$ define
		\[
		\psi^u_{\eps} \doteq \max\{u, \psi - \eps\} + \min\{u, \psi + \eps\} - u = \left\{ \begin{array}{ll}
		u & \text{if } \, |\psi - u| < \eps, \\[0.2cm]
		\psi + \eps & \text{if } \, u \ge \psi + \eps, \\[0.2cm]
		\psi -\eps & \text{if } \, u \le \psi -\eps.
		\end{array}\right.
		\]
Consider a sequence $\{\eps_j\} \subset \R^+$, $\eps_j \to 0$ and functions $u_j : \Omega \to \R$, and define $\psi_j \doteq \psi^{u_j}_{\eps_j}$.		
		
\begin{enumerate}
	\item[{\rm (i)}]
	Assume that $m \geq 2$, $\Omega$ is bounded, $\phi_j,\phi \in C(\partial \Omega)$, $u_j \in \cX_{\phi_j}(\Omega)$ and $u,\psi \in \cX_\phi(\Omega)$:
		\begin{itemize}
		\item[{\rm (a)}] 
		if $\|\phi_j-\phi\|_{C (\partial \Omega)} < \eps_j$, 
		then $\psi_j \in  \cX_{\phi_j}(\Omega)$ and $\psi_j \equiv u_j$ on $\Omega \setminus \Omega_j$ for some set $\Omega_j \Subset \Omega$. 
		Moreover, if for some $\Omega'\Subset \Omega$ it holds $\psi \equiv u$ and $|u_j - u|< \eps_j$ on $\Omega \setminus \Omega'$, 	then $\psi_j \equiv u_j$ on $\Omega \setminus \Omega'$;
		\item[{\rm (b)}] if $u_j \to u$ in $\cX(\Omega)$ as $j\to \infty$, then $\psi_j \to \psi$ in $W^{1,q}(\Omega) \cap C(\ov{\Omega})$ for each $q \in [1,\infty)$.
		\end{itemize}
	\item[{\rm (ii)}] 
	If $m \geq 3$, $\Omega = \Rm$ and $u,u_j,\psi \in \cX_0(\R^m)$ satisfy $u_j \to u$ in $\cX(\R^m)$, 
	then $\{\psi_j\} \subset \cX_0(\R^m)$ and {\rm (a)} holds. Furthermore, {\rm (b)} holds with $q \in [2^\ast,\infty)$, 
	and $\| D\psi_j - D\psi \|_q \to 0$ for all $q \in [2,\infty)$. 
	\end{enumerate}

\end{lemma}	
\begin{proof}
(i) 
By Proposition \ref{teo_compaembe}, $u,u_j,\psi \in C(\ov{\Omega})$. By construction, 
	\begin{equation}\label{conv-psij}
		\psi_j \in C(\ov{\Omega}), \qquad \|\psi_j - \psi\|_\infty \le \eps_j \to 0,  \qquad 
		\Omega_j \doteq \{|u_j-\psi| \ge \eps_j\} \Subset \Omega.
	\end{equation}
	Next, the identity
	\begin{equation}\label{Dpsij}
		D\psi_j = \begin{dcases}
		D \psi & \quad \text{a.e. on } \, |\psi - u_j| \ge \eps_j, \\[0.2cm]
		Du_j & \quad \text{a.e. on } \, |\psi - u_j| < \eps_j  
		\end{dcases}
	\end{equation}
implies that $|D \psi_j| \le 1$ a.e. on $\Omega$. 
Since $\psi_j = u_j$ on $\Omega \setminus \Omega_j$ and $u_j \in \cX_{\phi_j}(\Omega)$, 
we infer that $\psi_j \in \cX_{\phi_j}(\Omega)$. 
Thus, (a) holds. 
As for (b), if $u_j \to u$ in $\cX(\Omega)$, then we get $ u_j \to u$ in $C(\ov{\Omega})$. Thus, fix $\{\delta_j\}$ such that $\delta_j \to 0$ and $\|u_j - u\|_\infty < \delta_j$. 
Taking a subsequence $\{j_k\}$, we have $Du_{j_k} (x) \to Du(x)$ a.e. $\Omega$. 
Then as $k \to \infty$, a.e. $\Omega$,
	\begin{equation}\label{gradconverg}
	\begin{aligned}
		|D\psi_{j_k} - D\psi| 
		=
		|Du_{j_k}-D\psi| \cdot \mathbb{1}_{\{|\psi-u_{j_k}|<\eps_{j_k}\}} 
		&\le 
		|Du_{j_k}-D\psi| \cdot \mathbb{1}_{\{|\psi-u|< \eps_{j_k} + \delta_{j_k}\}} \\
		& \to |Du-D\psi| \cdot \mathbb{1}_{\{|\psi-u| = 0\}} = 0,
	\end{aligned}
	\end{equation}
where we used Stampacchia's theorem (see \cite[Theorem 4.4]{evansgariepy}). Since the limit is unique, 
$D\psi_j \to D\psi$ a.e. on $\Omega$. 
Thus, the dominated convergence theorem with $\| D\psi_j \|_\infty \leq 1$ yields 
$\| D\psi_j - D\psi \|_q \to 0$ for each $q \in [1,\infty)$. 
From \eqref{conv-psij}, (b) also holds.

	(ii) 
	When $\Omega = \Rm$, from \eqref{Dpsij} it is easily seen that $\psi_j \in \cX_0(\Rm)$. 
In addition, by Proposition \ref{prop-cX-emb}, $\| u_j - u \|_{\infty} \to 0$ and 
$\cX_0(\Rm) \hookrightarrow C_0(\Rm)$. 
Hence, we may apply the same argument as above to prove (a) in this case. 
As for (b), setting $f_k \doteq |Du_{j_k} - D\psi|$, $g_k \doteq \mathbb{1}_{\{|\psi-u|< \eps_{j_k} + \delta_{j_k}\}}$ and $f = |Du-D\psi|$, we deduce from \eqref{gradconverg} that
	\[
	\|D\psi_{j_k} - D\psi\|_2 \le \|f_kg_k\|_2 \le \|(f_k-f)g_k\|_2 + \|f g_k\|_2 \le \|(f_k-f)\|_2 + \|f g_k\|_2 \to 0
	\]
as $k \to \infty$, where we used $u_{j_k} \to u$ in $\cX(\R^m)$, $fg_k \to 0$ a.e. $\Rm$ and 
the dominated convergence theorem. 
The bound $\|D \psi_{j_k}- D\psi\|_\infty \le 2$ then implies $\| D\psi_{j_k} - D \psi \|_q \to 0$ for all $q \in [2,\infty)$.	
Since the limit is unique, $\| D\psi_j - D \psi \|_q \to 0$ for all $q \in [2,\infty)$. 
From $\| \psi_j - \psi \|_\infty \to 0$ and Sobolev's inequality, it follows that 
$\| \psi_j - \psi \|_{W^{1,q}} \to 0$ for all $q \in [2^\ast, \infty)$ and (b) also holds. 
\end{proof}

In the above lemma, assume that $\Omega$ is bounded. If $\phi_j \to \phi$ in $C(\partial \Omega)$ and $u_j \rightharpoonup u$ weakly in $\cX(\Omega)$, then the strong convergence $\psi_j \to \psi$ in $\cX(\Omega)$ may fail. 
However, for future use, we need to guarantee the existence of $\psi_j \in \cX_{\phi_j}(\Omega)$ and $\psi \in \cX_\phi(\Omega)$ such that $\psi_j \to \psi$ strongly in $\cX(\Omega)$. We propose a construction that depends on the following compactness result for Lorentzian balls.

	\begin{lemma}\label{lem_simplecompact}
		Let $\Omega \subset \R^m$ be a domain, let $\mathscr{G} \subset W^{1,\infty}(\Omega)$ be compact in $C(K)$ 
		for each compact set $K \subset \Omega $, and assume that $\|Du\|_\infty \le 1$ on $\Omega$ for each $u \in \mathscr{G}$. 
		Fix an open subset $\Omega' \Subset \Omega$ and $\wt{\e}>0$. Then, the following are equivalent:
		\begin{itemize}
			\item[{\rm (a)}] 
			For each $\Omega''\Subset \Omega'$ with $\di_\delta(\Omega'',\partial \Omega') \ge \wt{\e}$, 
			every $u \in \mathscr{G}$ does not have a light segment $\overline{xy} \subset \overline{\Omega'} \backslash \Omega''$ 
			with $x \in \partial \Omega''$, $y \in \partial \Omega'$.
			\item[{\rm (b)}] 
			There exists $R= R(\mathscr{G},\wt{\e},\Omega') >0$ such that $L_R^u(\Omega'') \Subset \Omega'$ for each $u \in \mathscr{G}$ 
			and each subset $\Omega'' \Subset \Omega'$ satisfying $\di_\delta(\Omega'', \partial \Omega') \ge \wt{\e}$, 
			where $L_R^u$ is the Lorentzian ball of radius $R$ associated to the graph of $u$.
		\end{itemize}
		Furthermore, the following are equivalent:
		\begin{itemize}
			\item[{\rm (a')}] 
			Every $u \in \mathscr{G}$ does not have light segments in $\Omega'$.
			\item[{\rm (b')}] 
			For each $\e >0$, there exists $R= R(\mathscr{G},\e,\Omega') >0$ such that for each pair of open subsets $\Omega_1 \Subset \Omega_2 \subset \Omega'$ with $\di_\delta(\Omega_1,\partial \Omega_2) \ge \e$, it holds $L_R^u(\Omega_1) \Subset \Omega_2$ for each $u \in \mathscr{G}$.
		\end{itemize}
	\end{lemma}
	
	\begin{proof}
		(b) $\Rightarrow$ (a) and (b') $\Rightarrow$ (a') are obvious. 
		The proofs of (a) $\Rightarrow$ (b) and (a') $\Rightarrow$ (b') are analogous, so we only prove (a') $\Rightarrow$ (b'). 
		Assume by contradiction the existence of $\e>0$, $\Omega_1^{(j)} \Subset \Omega_2^{(j)}$ 
		with $\di_{\delta} (\Omega_1^{(j)},\partial \Omega_2^{(j)}) \geq \e$, $u_j \in \mathscr{G}$, 
		points $z_j \in \partial \Omega_1^{(j)}$ and $p_j \in \partial \Omega_2^{(j)}$ such that 
		\begin{equation}\label{eq_bonita_noli}
			\ov{z_jp_j} \subset \ov{\Omega^{(j)}_2} \subset \ov{\Omega'}, \quad 
			\haus_\delta^1 \left( \ov{z_jp_j} \right) \geq \e, \quad 
			\left| z_j - p_j \right| - \left| u_j(z_j) - u_j(p_j) \right| \le \frac{1}{j}. 
		\end{equation}
		Since $\mathscr{G}$ is compact in $C(\ov{\Omega'})$, up to subsequences, 
		$u_j \to u \in \mathscr{G}$ in $C(\ov{\Omega'})$, $z_j \to z \in \ov{\Omega'} $ and $p_j \to p \in \ov{\Omega'}$. 
		Passing to the limit in \eqref{eq_bonita_noli}, $u$ has a light segment $\overline{zp}$ of length $\ge \e$.
		Noticing that $B_\e(z_j) \subset \Omega'$ for each $j$, we get $B_\e(z) \subset \Omega'$ and 
		thus part of $\ov{zp}$ lies in $\Omega'$, a contradiction. 
	\end{proof}

\begin{lemma}\label{lem_strongconv}
Let $\Omega \subset \R^m$ be a bounded domain, $\phi_j, \phi \in \Spa(\partial \Omega)$ satisfy $\phi_j \to \phi$ in $C(\partial \Omega)$ as $j \to \infty$. Then, there exist $\psi_j \in \cX_{\phi_j}(\Omega)$ and $\psi \in \cX_\phi(\Omega)$ such that $\psi_j \to \psi$ strongly in $\cX(\Omega)$. 
\end{lemma} 

\begin{proof}
By Theorem \ref{teo_bartniksimon_2}, we can find $\psi_j$ (respectively, $\psi$) of class $C^\infty(\Omega)$ solving \eqref{borninfeld} with $\rho = 0$ and boundary value $\phi_j$ (resp. $\phi$). 
The comparison principle and $\phi_j \to \phi$ in $C(\partial \Omega)$ imply 
$\psi_j \to \psi$ in $C(\ov{\Omega})$. 
%
%
For each compact set $K \subset \Omega$, we apply Lemma \ref{lem_simplecompact} with $\mathcal{G} := \{ \psi_{j} \}_j \cup \{ \psi \}$ to obtain $R>0$ such that $L^{\psi_{j}}_R (K) \Subset \Omega$ for all $j \geq 1$. The monotonicity formula \cite[Lemma 2.1]{bartniksimon} 
and the argument in \cite[Proof of Theorem 4.1]{bartniksimon} imply that there exists $\theta >0$ such that 
	\[
		\sup_{j \geq 1} \sup_{K} |D \psi_{j} | \leq 1 - \theta
	\]
and $\{ \psi_{j} \}$ is bounded in $W^{2,2}_{\rm loc} (\Omega)$. Hence, up to a subsequence, $D \psi_{j_k} \to D \psi$ a.e. $\Omega$. 
The dominated convergence theorem then guarantees $\|D\psi_{j_k} -D\psi\|_p \to 0$ for each $p \in [1,\infty)$, whence $\psi_{j_k} \to \psi$ strongly in $\cX(\Omega)$. 
Since $\psi$ is independent of the choices of subsequences, we obtain $\psi_j \to \psi$ strongly in $\cX(\Omega)$. 
%
\end{proof}

We conclude this subsection with a simple observation, which however implies an important property of light segments.
\begin{proposition}\label{prop_nocross}
	Let $m \geq 2$, $\Omega \subset \R^m$ be any nonempty domain, 
	and let $u \in W^{1,\infty}(\Omega)$ with $|D u| \le 1$ almost everywhere. Then, light segments in the graph of $u$ cannot cross each other, that is, no point of $\Omega$ can lie in the interior of two distinct light segments.
\end{proposition}
\begin{proof}
	The conclusion directly follows from \cite[Lemma 3.5]{wang_notes}, according to which, if $\overline{xy} \subset \Omega$ is a light segment for $u$, then $u$ is differentiable at any interior point of $\overline{xy}$ with gradient being the unit vector along $\overline{xy}$ 
	pointing towards the increasing direction of $u$.
\end{proof}

\subsection{Convergence of minimizers}\label{sec:conv-min}

Our proof of the solvability of \eqref{borninfeld} depends on an approximation procedure, 
smoothing $\rho$ by convolution. Thus, it entails a convergence result for minimizers.

\begin{proposition}\label{1001}
Let $\{\rho_j\} \subset \cX(\Omega)^*$, and consider the following assumptions:
	\begin{enumerate}
	\item[{\rm (i)}] 
	$\Omega \subset \Rm$ is a bounded domain with $m \geq 2$ and 
	$\{\phi_j\} \subset \Spa(\partial \Omega) $ satisfies $\phi_j \to \phi \in \Spa (\partial \Omega)$ strongly in $C(\partial \Omega)$. 
	Assume that $\rho_j = \mu_j + f_j$, where $\mu_j \in \cM(\Omega)$, $f_j \in \cX(\Omega)^\ast$, and that 
		\begin{equation}\label{conv-muk-fk}
		\mu_j \rightharpoonup \mu \quad \text{weakly in } \, \cM(\Omega), 
		\qquad f_j \to f \quad \text{strongly in } \, \cX(\Omega)^\ast.
		\end{equation}
	\item[{\rm (ii)}] 
	$\Omega = \Rm$ with $m \geq 3$, $\rho_j = \mu_j + f_j$ where $\mu_j$ and $f_j$ satisfy \eqref{conv-muk-fk}. Assume also that, for each $\e>0$, there exists $R_\e > 0$ such that
	\begin{equation}\label{tight}
			\left| \mu_j \right| \left( \R^m \backslash B_{R_\e} \right) < \e \quad \text{for each $j \geq 1$}.
	\end{equation}
	\end{enumerate}
Under either {\rm (i)} or {\rm (ii)}, by setting $\rho \doteq \mu + f$, up to a subsequence, 
$u_{\rho_j} \to u_\rho$ strongly in $W^{1,q}(\Omega) \cap C(\ov{\Omega})$, respectively, for every $q \in [1,\infty)$ if $\Omega$ is a bounded domain, and for every $q \in [2^\ast,\infty)$ if $\Omega = \R^m$. 
Furthermore, $\| Du_{\rho_j} - Du_\rho \|_q \to 0$ for every $q \in [2,\infty)$ when $\Omega = \Rm$. 
In particular,
	\[
	u_{\rho_j} \to u_\rho \quad \text{in } \, \cX(\Omega), \qquad	\la \rho_j, u_{\rho_j} \ra \to \la \rho, u_\rho \ra \qquad \text{as } \, j \to \infty. 
	\]
\end{proposition}

\begin{proof}
For the ease of notation, we write $u_j$ instead of $u_{\rho_j}$. 
We first suppose that $\Omega$ is bounded. Notice that by $\phi \in \Spa(\partial \Omega)$, 
$\cX_\phi \neq \emptyset$ and the minimizer $u_\rho$ exists. On the other hand, due to Proposition \ref{teo_compaembe} and $u_{j} \in \cX_{\phi_j}(\Omega)$, 
$\{u_j\}$ is relatively compact in $C(\overline\Omega)$ and hence it is bounded in $W^{1,q}(\Omega)$ for any $q \in [1,\infty]$. 
Up to a subsequence, $u_{j} \rightharpoonup u$ weakly in $W^{1,q}(\Omega)$ for each fixed $q \in (1,\infty)$, and strongly in $C(\ov{\Omega})$, 
which yields $u \in \cX_\phi (\Omega) $.

	From \eqref{1} we get 
	\begin{equation}\label{ineq:vu_k}
		\begin{aligned}
			\int_{\Omega} \left( 1 - \sqrt{1 - |Du|^2} \right) \di x
			= 
			\sum_{k=1}^\infty b_k\|D u \|_{2k}^{2k} 
			&\leq \sum_{k=1}^\infty b_k \liminf_{j \to\infty} 
			\| Du_{j} \|_{2k}^{2k} 
			\\
			&\leq \lim_{n \to \infty} \liminf_{j \to \infty} \sum_{k=1}^n b_k \| Du_{j} \|_{2k}^{2k}
			\\
			&\leq \liminf_{j\to\infty} \int_{ \Omega } 
			\left( 1 - \sqrt{1 - |Du_{j}|^2 } \right) \rd x.
		\end{aligned}
	\end{equation}
Since $u_{j} \to u$ weakly in $\cX(\Omega)$ and strongly in $C(\ov{\Omega})$, our assumptions on $\{\mu_j\}$ and $\{ f_j\}$ give 
	\begin{equation}\label{conv-rhok-uk}
		\lim_{j \to \infty} \la \rho_j , u_{j} \ra = \lim_{j \to \infty} \left( \la \mu_j, u_j \ra + \la f_j , u_j \ra \right) = \la \mu, u \ra + \la f , u \ra = \la \rho , u \ra.
	\end{equation}
Hence, by \eqref{ineq:vu_k} and the minimality of $u_\rho$ and $u_j$, we obtain 
	\begin{equation}\label{ineq-urho-up1}
		\begin{aligned}
			I_\rho (u_\rho)
			\leq I_\rho(u)
			\leq \liminf_{j \to\infty}
			I_{\rho_{j}}(u_j).
		\end{aligned}
	\end{equation}

By Lemma \ref{lem_strongconv}, choose $\psi_j \in \cX_{\phi_j}(\Omega)$ and $\psi \in \cX_\phi(\Omega)$ such that $\psi_j \to \psi$ in $\cX(\Omega)$ as $j \to \infty$. Having fixed a sequence $\eps_j \to 0$ with $\eps_j > |\phi_j-\phi|$ on $\partial \Omega$, we consider the approximation $\bar u_j \doteq ( u_\rho )^{ \psi_j }_{\e_j} \in \cX_{\phi_j}(\Omega)$ of $u_\rho$ given by Lemma \ref{prop_approx}, 
which therefore satisfies $\bar u_j \to u_\rho$ in $\cX(\Omega)$. 
Weak convergence $\rho_j \rightharpoonup \rho$ and continuity of $I_\rho$ imply that $I_{\rho_j}(\bar u_j) \to I_\rho(u_\rho)$. 
%
%
Hence, by the minimality of $u_j$ we obtain 
	\begin{equation}\label{ineq-urho-up2}
		\limsup_{j \to \infty} I_{\rho_j} (u_j) \leq \limsup_{j \to \infty} I_{\rho_j} (\bar u_j) = I_{\rho} (u_\rho). 
	\end{equation}
From \eqref{ineq-urho-up1}, we infer that 
$I_\rho (u_\rho) = I_\rho(u) = \lim_{j \to \infty} I_{\rho_j} (u_j)$, so $u = u_\rho$ by the uniqueness of minimizer.

	Finally, we shall show $u_j \to u$ strongly in $W^{1,q} (\Omega)$ for any $q \in [1,\infty)$. 
To this end, notice that $I_{\rho_j} (u_j) \to I_\rho(u_\rho)$ and \eqref{conv-rhok-uk} with $u = u_\rho$ imply 
	\[
	\int_{\Omega} \left( 1 - \sqrt{1-|Du_{j}|^2} \right) \di x
	\to \int_{\Omega} \left( 1 - \sqrt{1 - |Du_\rho|^2 } \right) \di x.
	\]
If there exists $k_0>0$ such that
	\[
	\e_0 \doteq \liminf_{j \to\infty} \| Du_{j} \|_{2k_0}^{2k_0} -
	\| Du_\rho \|_{ 2k_0 }^{2k_0} > 0,
	\]
then by \eqref{1} we can choose $h_0>k_0$ so large that
	\[
	\int_{\Omega} \left( 1 - \sqrt{1-|Du_\rho|^2} \right) \di x
	- \sum_{k=1}^{h_0} b_k \| Du_\rho \|_{2k}^{2k}
	< \frac{b_{k_0}\e_0}{2},
	\]
and therefore deduce the following contradiction:
	\[
	\begin{aligned}
	\int_{\Omega} \left( 1 - \sqrt{1-|D u_\rho|^2} \right)\di x
	&< \frac{b_{k_0}\e_0}{2} + \sum_{k=1}^{h_0} b_k \|Du_\rho \|_{2k}^{2k}
	\\
	&\leq \liminf_{j\to\infty} \sum_{k=1}^{h_0} b_k
	\| Du_{j} \|_{2k}^{2k} - \frac{b_{k_0} \e_0 }{2}
	\\
	&\leq \liminf_{j\to\infty} \int_{\Omega} \left( 1 - \sqrt{1-|Du_{j}|^2}\right) \di x
	- \frac{b_{k_0} \e_0 }{2}
	\\
	&= \int_{\Omega} \left( 1 - \sqrt{1-|Du_\rho|^2} \right)\di x - \frac{b_{k_0} \e_0 }{2}.
	\end{aligned}
	\]
Thus, $\| Du_{j} \|_{2k} \to \| Du_\rho \|_{2k}$ for each $k \geq 1$. 
The uniform convexity of $L^{2k} (\Omega)$ and $\| u_{j} - u_\rho \|_{\infty} \to 0$ imply that 
$Du_{j} \to Du_\rho$ in $L^{2k}(\Omega)$, hence $u_{j} \to u_\rho$ in $W^{1,2k}(\Omega)$ for any $k \geq 1$. 
By H\"older's inequality, $u_{j} \to u_\rho$ strongly in $W^{1,q} (\Omega)$ for each $q \in [1,\infty)$ 
and we complete the proof for the case $\Omega$ is a bounded domain.

	When $\Omega = \Rm$ with $m \geq 3$, first observe that by our assumptions $\{\rho_j\}$ is uniformly bounded in $\cX(\Omega)^*$. 
Hence, from $I_{\rho_j}(u_{j}) \le I_{\rho_j}(0) = 0$ and the coercivity estimate \eqref{coercive} for $v = u_{j}$, 
we deduce that $\{u_{j}\}$ is uniformly bounded in $\cX(\R^m)$. 
By Proposition \ref{prop-cX-emb} and $\| Du_{j} \|_\infty \leq 1$, 
$\{u_{j}\}$ is bounded in $W^{1,q}(\R^m)$ for each $q \in [2^\ast,\infty)$, hence in $L^\infty(\R^m)$. 
Up to a subsequence, $u_{j} \rightharpoonup u$ weakly in $W^{1,q}(\Rm)$ for each $q \in [2^\ast, \infty)$, $u_{j} \to u$ in $C_\loc(\Rm)$, and $u_{j} \rightharpoonup u$ weakly in $\cX(\R^m)$ by the reflexivity of $\cX(\R^m)$. Since each $u_{j}$ is $1$-Lipschitz, so is $u$ and $u \in \cX_0(\Rm)$. Coupling condition \eqref{tight} for $\{\mu_j\}$ with the convergence $u_{j} \to u$ in $C_\loc(\Rm)$ and the uniform boundedness of $\{u_{j}\}$, we deduce that $\langle \mu_j , u_{j} \rangle \to \la \mu,u  \ra$, hence \eqref{conv-rhok-uk} holds.  
Then, arguing as above with $\bar u_j \doteq u_\rho$, 
we may verify $u = u_{\rho}$ and $Du_{j} \to Du_\rho$ strongly in 
$L^q(\Rm)$ for each $q \in [2,\infty)$. 
Hence, $u_{j} \to u_\rho$ strongly in $W^{1,q} (\Rm)$ for every $q \in [2^\ast, \infty)$, in particular $u_j \to u_\rho$ in $\cX(\R^m)$. Concluding, $u_j \to u_\rho$ in $C(\R^m)$ follows from Proposition \ref{prop-cX-emb}. 
\end{proof}

	\begin{remark}\label{rem-use-spacelike}
		In the proof of Proposition \ref{1001} (i), condition $\phi_j, \phi \in \Spa (\partial \Omega)$ is only used to construct a good competitor $\bar u_j$ to obtain \eqref{ineq-urho-up2}. 
		When $\phi_j \equiv \phi $ for each $j$, we may simply choose $u_j = u_\rho$ and weaken the condition $\phi \in \Spa (\partial \Omega)$ to $\cX_{\phi} (\Omega) \neq \emptyset$. 
	\end{remark}

\subsection{Local integrability of $w$ and the Euler-Lagrange inequality}

In this subsection, assuming $\phi \in \Spa(\partial \Omega)$ if $\Omega$ is bounded, 
we show that the minimizer $u_\rho$ is not too degenerate and solves an Euler-Lagrange inequality. 
The argument in the proof is inspired by \cite[Proposition 2.6]{bpd}, in particular, case (ii) in the following is essentially contained therein. 

\begin{proposition}\label{lem_basicL2}
Let $\Omega \subset \R^m$ be a domain.	
	\begin{enumerate}
		\item[{\rm (i)}]
		Assume that $m \geq 2$ and that $\Omega$ is bounded. For any given compact subset $\FF \subset \Spa(\partial \Omega)$, and any  $\e, \mathcal{I}_1 > 0$, there exists a constant 
		$\mathcal{C} = \mathcal{C}(\Omega,\FF,m,p_1,\mathcal{I}_1, \diam_{\delta}(\Omega),\e)$ 
		such that if 
		\[
			\phi \in \FF, 
			\qquad \rho \in \cX(\Omega)^* \ \ \text{with } \, \|\rho\|_{\cX^*} \le \mathcal{I}_1, 
		\]
		then for each open subset $\Omega' \Subset \Omega$ with $\di_\delta(\Omega',\partial \Omega) \ge \eps$ the minimizer $u_\rho$ satisfies
		\begin{equation}\label{basic_L2bound}
			\int_{\Omega'} \frac{\di x}{\sqrt{1- |Du_\rho|^2}}  \le \mathcal{C}. 
		\end{equation}
		In particular, $|Du_\rho|< 1$ a.e. on $\Omega$. Moreover, for each $\psi \in \cX_\phi(\Omega)$, 
		\begin{align}
			& \frac{ Du_\rho \cdot (Du_\rho - D\psi)}{\sqrt{1-|Du_\rho|^2}} \in L^1(\Omega) \label{eq_L1test}, \\
			& \sqrt{ 1 - |D\psi|^2 } - \sqrt{1 - |Du_\rho|^2 }\le \frac{ Du_\rho \cdot (Du_\rho - D\psi)}
				{\sqrt{1-|Du_\rho|^2}} \qquad \text{a.e. on } \, \Omega \label{eq_reversedCS}
		\end{align}
		and 		
		\begin{equation}\label{eq:2}
			\begin{aligned}
				\int_{\Omega}\Big( \sqrt{ 1 - |D\psi|^2 } - \sqrt{1 - |Du_\rho|^2 }\Big) \rd x 
				\leq \int_{\Omega} \frac{ Du_\rho \cdot (Du_\rho - D\psi)}
				{\sqrt{1-|Du_\rho|^2}} \rd x \leq \la \rho , u_\rho - \psi \ra.
			\end{aligned}
		\end{equation}
		
		\item[{\rm (ii)}]
		Assume that $m \geq 3$ and that $\Omega = \Rm$. For any given $\mathcal{I}_1>0$ and $\Omega' \Subset \Rm$, 
		there exists a constant $\mathcal{C}' = \mathcal{C}' ( m, p_1,\mathcal{I}_1, |\Omega'|_\delta) > 0$ such that 
		if $\| \rho \|_{\cX^\ast} \leq \mathcal{I}_1$, then \eqref{basic_L2bound} holds with $\mathcal{C}'$. 
		Furthermore, \eqref{eq_L1test}--\eqref{eq:2} hold for each $\psi \in \cX_0(\Rm)$. 
	\end{enumerate}
\end{proposition}	

\begin{remark}
Notice that choosing $\Omega = \R^m$ and $\psi = 0$ in \eqref{eq_L1test} we infer the integrability condition in \eqref{eq_ham_Rm} mentioned in the Introduction.
\end{remark}

\begin{proof}[Proof of Proposition \ref{lem_basicL2}]
(i) 
We first prove \eqref{basic_L2bound}. Fix $\Omega' \Subset \Omega$ with $\rd_\delta (\Omega', \partial \Omega) \geq \e$. 
Given $\psi \in \cX_\phi(\Omega)$, observe that $u_t \doteq (1-t ) u_\rho + t  \psi \in \cX_\phi(\Omega)$ for every $t  \in (0,1]$. 
Thus, $I_\rho (u_\rho) \leq I_\rho (u_t)$, and  rearranging we get
	\begin{equation}\label{eq:4-0}
		\frac{1}{t} \int_{\Omega} 
		\Big( \sqrt{ 1 - |Du_t|^2} 
		- \sqrt{1 - |Du_\rho|^2 }\Big) \rd x 
		\leq \la \rho , u_\rho - \psi \ra
		\quad \forall \, t  \in (0,1].
	\end{equation}

	Next, the concavity of the map $p \mapsto \sqrt{1-|p|^2}$ on $\ov{B_1(0)}$ implies that
	\[
\sqrt{1 - |Du_t|^2} \geq (1-t ) \sqrt{1 - \left|Du_\rho\right|^2} + t \sqrt{1 - |D \psi|^2} \qquad 
		\text{a.e. on $\Omega$}, \ \forall \ t \in (0,1],
	\]
which yields 
	\begin{equation}\label{eq_concavity}
\disp \sqrt{1-|D\psi|^2} - \sqrt{1-|Du_\rho|^2} \le \disp \frac{1}{t} \Big\{ \sqrt{1-|Du_t|^2} - \sqrt{1-|Du_\rho|^2} \Big\} \qquad \text{a.e. on } \, \Omega. 
	\end{equation}
Let $\mathscr{G} \subset \cX(\Omega)$ be the set of minimizers of $I_0$ (i.e. with $\rho = 0$) whose boundary value lies in $\FF$. 
For given $\phi \in \FF$ we denote by $\bar \phi \in \mathscr{G}$ the corresponding minimizer. 
The compactness of $\FF$ and Propositions \ref{teo_compaembe} and \ref{1001} guarantee that $\mathscr{G}$ is compact in $C(\overline\Omega)$. 
By Theorem \ref{teo_bartniksimon_2}, every $ u \in \mathscr{G}$ does not have light segments in $\Omega$, 
thus applying the first part of Lemma \ref{lem_simplecompact}  for $\Omega_\e \Subset \Omega_{\e/2}$ 
we obtain $R = R(\Omega,\FF,\e) > 0$ such that 
$L^{u}_{R} ( \Omega_\e ) \Subset \Omega_{\e/2}$ for each $u \in \mathscr{G}$. 
From the monotonicity formula \cite[Lemma 2.1]{bartniksimon}, 
we infer the existence of $\theta=\theta(\Omega, \FF, \e)$ such that 
\begin{equation}\label{eq:3}
	\sup_{\Omega'}|D\bar \phi| \le 1- 4\theta.
	\end{equation}
We take $\psi = \bar \phi$, and note that on the set of full measure $V \subset \Omega'$ of points where $u_\rho$ is differentiable it holds $|Du_t|< 1$ for every $t \in (0,1]$. We set 
	\[
		K \doteq \Big\{ x \in \Omega  \ :\  1-\theta < |Du_\rho(x)| \Big\},
	\]
and split the domain of integration $\Omega$ in \eqref{eq:4-0} into the sets $\Omega \setminus \Omega'$, $V \cap K$ and $V \cap K^c$. 
We use \eqref{eq_concavity} on $\Omega \setminus \Omega'$ and the identity 
	\begin{equation}\label{eq_useful_pi}
	\begin{array}{l}
	\disp \frac{1}{t} \Big\{ \sqrt{1-|Du_t|^2} - \sqrt{1-|Du_\rho|^2} \Big\} \\[0.5cm]
	\qquad \qquad \disp = \frac{ 2 Du_\rho \cdot (Du_\rho -D \psi) - t |Du_\rho - D\psi|^2 }
			{\sqrt{1 - |Du_t|^2 }  + \sqrt{1 - |Du_\rho|^2} } \qquad \text{a.e. on } \, \Omega \cap \big\{ |D\psi| + |Du_\rho| < 2 \big\}
	\end{array}
	\end{equation}
to deduce that 
	\begin{equation}\label{eq:8}
		\begin{aligned}
			\la \rho , u_\rho - \bar\phi \ra 
			& \geq \int_{\Omega \setminus \Omega'} 
			\Big(\sqrt{ 1 - |D \bar\phi|^2 } - \sqrt{ 1 - | D u_\rho|^2 } \Big)  \rd x 
			\\
			& \qquad 
			+ \int_{V \cap K} 
			\frac{ 2Du_\rho \cdot (Du_\rho - D\bar \phi) - t |Du_\rho - D \bar\phi|^2}
			{ \sqrt{ 1 - |Du_t|^2} + \sqrt{ 1 - |Du_\rho|^2 } } \rd x
			\\
			& \qquad 
			+ \int_{V \cap K^c }  
			\frac{ 2Du_\rho \cdot (Du_\rho - D\bar \phi) - t |Du_\rho - D \bar\phi|^2}
			{ \sqrt{ 1 - |Du_t|^2} + \sqrt{ 1 - |Du_\rho|^2 } } \rd x.
		\end{aligned}
	\end{equation}

	Recalling \eqref{eq:3}, we restrict to $t$ small enough so that $4t < \theta^2$. 
By the definition of $K$, the next inequality holds on $\Omega' \cap K$:
	\begin{equation}\label{eq:9}
			2 Du_\rho \cdot (Du_\rho - D\bar \phi) - t|Du - D\bar \phi|^2 
			\ge 
			2 \left[ (1-\theta)^2 - (1-4\theta) \right] - 4t > 4\theta > 0.
	\end{equation}
Remark also that the last term in the right-hand side of \eqref{eq:8} is bounded uniformly with respect to $t  \in (0,1)$.  
Thus, letting $t \to 0$ in \eqref{eq:8} and using \eqref{eq:9}, Fatou's lemma and the dominated convergence theorem, we infer
	\begin{equation}\label{eq:10}
		\begin{aligned}
			\la \rho , u_\rho - \bar\phi \ra 
			&\geq  
			\int_{\Omega \setminus \Omega'} 
			\Big(\sqrt{ 1 - |D \bar\phi|^2 } - \sqrt{ 1 - | D u_\rho|^2 } \Big)  \rd x 
			\\
			&\qquad 
			+ \int_{V \cap K}  
			\frac{2\theta}
			{\sqrt{1-|Du_\rho|^2}} \rd x  + \int_{V \cap K^c } 
			\frac{ Du_\rho \cdot (Du_\rho - D\bar \phi) }
			{\sqrt{1-|Du_\rho|^2}} \rd x.
		\end{aligned}
	\end{equation}
From
	\begin{equation}\label{eq_fuoriomegap}
		\left| \int_{\Omega \setminus \Omega'} \sqrt{ 1 - |D \bar\phi|^2 } - \sqrt{ 1 - | D u_\rho|^2 } \  \rd x  \right| 
		\leq |\Omega \setminus \Omega'|_\delta
	\end{equation}
and the following straightforward estimate on $\Omega' \cap K^c$: 
	\begin{equation*}\label{eq:11}
		\int_{\Omega' \cap K^c} 
		\left|\frac{Du_\rho \cdot (Du_\rho - D\bar \phi)}{\sqrt{1-|Du_\rho|^2}} \right| \rd x
		\leq \int_{\Omega' \cap K^c } \frac{2 \di x}{\sqrt{ 2 \theta - \theta^2 }} 
		\le \frac{2|\Omega'|_\delta }{\sqrt{2\theta - \theta^2}},
 	\end{equation*}
it follows from \eqref{eq:10} and $|\Omega' \setminus V|=0$ that 
	\[
		\int_{\Omega' \cap K} \frac{2\theta}{\sqrt{1 - |Du_\rho|^2}} \rd x 
		\leq |\Omega \setminus \Omega'|_\delta + \la \rho , u_\rho - \bar\phi \ra + \frac{2  |\Omega'|_\delta }{\sqrt{2\theta - \theta^2}}. 
	\]
Therefore, 
	\begin{equation}\label{eq_final_ine}
		\begin{aligned}
			\int_{\Omega'} \frac{\rd x}{\sqrt{1-|Du_\rho|^2}} 
			&=
			\int_{\Omega' \cap K} \frac{ \rd x}{\sqrt{1-|Du_\rho|^2}}
			+ \int_{\Omega' \cap K^c} \frac{ \rd x}{\sqrt{1-|Du_\rho|^2}} 
			\\
			&\leq
			\frac{1}{2\theta} 
			\left( |\Omega \setminus \Omega'|_\delta + \|\rho\|_{\cX^*} \| u_\rho - \bar{\phi}\|_{\cX} + \frac{2 |\Omega'|_\delta}{\sqrt{2\theta -\theta^2}}
			\right) + \frac{|\Omega'|_\delta}{\sqrt{2\theta -\theta^2}}.
		\end{aligned}
	\end{equation}
For $\psi \in \cX_{\phi} (\Omega)$, \eqref{eq_Linftybound} and simple estimates for the $W^{1,p_1}$ norm give 
	\[
	\|u_\rho - \bar \phi\|_{\cX} 
	\le 4\left( \sup_{\phi \in \FF} \|\phi\|_{C(\partial \Omega)} + \diam_\delta(\Omega) + |\Omega|_\delta^{\frac{1}{p_1}} \right).
	\]
Hence, \eqref{basic_L2bound} holds by \eqref{eq_final_ine}. 
Notice that, by \eqref{basic_L2bound} and the arbitrariness of $\Omega'$, $|Du_\rho| < 1 $ a.e. on $\Omega$.

	Next, we shall prove \eqref{eq_L1test}--\eqref{eq:2}. 
Let $\psi \in \cX_\phi(\Omega)$ and consider as above $u_t \doteq (1-t) u_\rho + t \psi \in \cX_\phi(\Omega)$ for $t \in (0,1)$. 
By combining $|Du_\rho|<1$ a.e. $\Omega$, \eqref{eq_useful_pi} and \eqref{eq_concavity}, for each $t \in (0,1)$,
	\begin{equation}\label{eq:est-below}
		\disp \sqrt{1-|D\psi|^2} - \sqrt{1-|Du_\rho|^2} \le \frac{ 2 Du_\rho \cdot (Du_\rho -D \psi) - t |Du_\rho - D\psi|^2 }
		{\sqrt{1 - |Du_t|^2 }  + \sqrt{1 - |Du_\rho|^2} } \qquad \text{a.e. on } \, \Omega.
	\end{equation}
Thus letting $t \to 0$ on the set $\{|Du_\rho|< 1\}$, we deduce \eqref{eq_reversedCS}.

	On the other hand, from \eqref{eq:4-0} and \eqref{eq_useful_pi}, it follows that 
	\[
		\int_\Omega \frac{ 2 Du_\rho \cdot (Du_\rho -D \psi) - t |Du_\rho - D\psi|^2 }
		{\sqrt{1 - |Du_t|^2 }  + \sqrt{1 - |Du_\rho|^2} } \di x \leq \la \rho , u_\rho - \psi \ra.
	\]
Using a variant of Fatou's lemma as $t \to 0$ and \eqref{eq:est-below}, we therefore deduce 
	\[
\int_{\Omega} \Big(\sqrt{ 1 - |D\psi|^2 } - \sqrt{1 - |Du_\rho|^2 } \Big) \rd x 
			\leq \int_{\Omega} \frac{ Du_\rho \cdot (Du_\rho - D\psi)}
			{\sqrt{1-|Du_\rho|^2}} \rd x \leq \la \rho , u_\rho - \psi \ra,
	\]
which proves \eqref{eq:2}. Taking \eqref{eq_reversedCS} into account, both the negative and the positive part of 
	\[
	\frac{Du_\rho \cdot (Du_\rho- D\psi)}{\sqrt{1-|Du_\rho|^2}}
	\]
are integrable, and \eqref{eq_L1test} holds.

	(ii) We first observe that \eqref{coercive}, $I_\rho(u_\rho) \le I_\rho(0) = 0$ and $\|\rho\|_{\cX^*} \le \mathcal{I}_1$ imply that 
$\|u_\rho\|_\cX \le C_1(m, \mathcal{I}_1)$. 
One can therefore perform the same computations in \eqref{eq:4-0}--\eqref{eq:10} with 
$\Omega = \R^m$, $\bar \phi = 0$, $\theta = 1/8$ and replacing \eqref{eq_fuoriomegap} with
	\[
	0 \le \int_{\R^m\backslash \Omega'} \Big( 1- \sqrt{1-|Du_\rho|^2} \Big) \di x \le I_\rho(u_\rho) + \la \rho, u_\rho \ra \le \mathcal{I}_1 C_1.
	\] 
Inequality \eqref{eq_final_ine} becomes 
	\[
	\int_{\Omega'} \frac{\rd x}{\sqrt{1-|Du_\rho|^2}} \leq
			4\left( 2\mathcal{I}_1 C_1 + C_2|\Omega'|_\delta\right) + C_2|\Omega'|_\delta,
	\]
for some absolute constant $C_2$. The rest of the proof follows verbatim, taking into account that  
$1 - \sqrt{1-|p|^2} \leq |p|^2$ on $\ov{B_1(0)}$ and thus 
$\sqrt{1-|D\psi|^2} - \sqrt{1-|Du_\rho|^2} = (1-\sqrt{1-|Du_\rho|^2} ) - ( 1 - \sqrt{1-|D\psi|^2} ) \in L^1(\Rm)$. This completes the proof.
\end{proof}

\begin{remark}\label{rem_reversedCS}
Inequality \eqref{eq_reversedCS} has a nice geometric interpretation, holding more generally for pairs of Lipschitz functions $u,\psi$ with $|Du|< 1$, $|D\psi| \le 1$ a.e. on $\Omega$. Indeed, if we consider the normal vectors ${\bf n}_u' \doteq  Du + \partial_0$, ${\bf n}_\psi' = D\psi + \partial_0$ (respectively, timelike and causal a.e. on $\Omega$), the reversed Cauchy-Schwarz inequality $- {\bf n}_u' \cdot {\bf n}_\psi' \ge |{\bf n}_u'|_{\mathbb{L}}|{\bf n}_\psi'|_{\mathbb{L}}$ is equivalent to  
	\[
		\frac{{\bf n}_u'}{|{\bf n}_u'|_{\mathbb{L}}} \cdot ({\bf n}_u'-{\bf n}_\psi') \ge |{\bf n}_\psi'|_{\mathbb{L}}- |{\bf n}_u'|_{\mathbb{L}},
	\]
that can be rewritten as \eqref{eq_reversedCS} with $u$ replacing $u_\rho$. 
\end{remark}

\subsection{Global minimizers VS solutions to $(\mathcal{BI})$}

In this subsection, we describe in detail the interplay between solutions of \eqref{borninfeld} and global minimizers of $I_\rho$, stating some useful equivalent characterizations of the solvability of \eqref{borninfeld} that, perhaps surprisingly, hold without assuming any regularity of $\partial \Omega$.

\begin{definition}
We say that $u \in \cX_\phi(\Omega)$ is a \emph{local minimizer for $I_\rho$} 
if $I_\rho(u) \le I_\rho(\psi)$ for every $\psi \in \cX_\phi(\Omega)$ with $\{u \neq \psi\} \Subset \Omega$. 
Similarly, for $\Omega = \R^m$, we say that $u \in \cX_0(\R^m)$ is a local minimizer for $I_\rho$ 
if $I_\rho(u) \le I_\rho(\psi)$ for every $\psi \in \cX_0(\R^m)$ with $\{u \neq \psi\} \Subset \R^m$.
\end{definition}	


\begin{proposition}[Minimizers VS solutions to \eqref{borninfeld}]\label{prop_localglobal}
Let $m \geq 2$, $\Omega$ be a bounded domain, $\phi \in \mathcal{S}(\partial \Omega)$ and $u$ a local minimizer.  
Then, $u = u_\rho$. Furthermore, the following are equivalent:
\begin{itemize}
	\item[{\rm (i)}] $u$ is a weak solution to \eqref{borninfeld}, that is, 
		\begin{equation}\label{equi_BI_1}
		\frac{1}{\sqrt{1-|Du|^2}} \in L^1_\loc(\Omega), \qquad \int_\Omega \frac{D u \cdot D\eta}{\sqrt{1-|Du|^2}} \di x =  \langle \rho, \eta \rangle \quad \forall \, \eta \in \lip_c(\Omega);
		\end{equation}		
	\item[{\rm (ii)}] $u = u_\rho$ and 
		\[
		\int_\Omega \frac{D u \cdot (Du - D \psi)}{\sqrt{1-|Du|^2}} \di x = \langle \rho, u- \psi \rangle \qquad \forall \, \psi \in \cX_\phi(\Omega) \ \text{ strictly spacelike};		
		\]
	\item[{\rm (iii)}] $u = u_\rho$ and 
		\begin{equation}\label{equi_BI_3}
		\int_\Omega \frac{D u \cdot (Du - D \psi)}{\sqrt{1-|Du|^2}} \di x = \langle \rho, u- \psi \rangle \qquad \forall \, \psi \in \cX_\phi(\Omega) \ \text{ with } \, \{\psi \neq u \} \Subset \Omega;		
		\end{equation}				
	\item[{\rm (iv)}] $u = u_\rho$ and
		\[
			\int_\Omega \frac{D u \cdot (Du - D \psi)}{\sqrt{1-|Du|^2}} \di x = \langle \rho, u- \psi \rangle \qquad \forall \, \psi \in \cX_\phi(\Omega).
		\]
\end{itemize}
In particular, if $u$ is a classical solution to \eqref{borninfeld}, then $u$ satisfies any of {\rm (i)}--{\rm (iv)}.

	The same assertions hold true for $m \geq 3$ and $\Omega = \Rm$. 
\end{proposition}
\begin{proof}
Since the case $\Omega = \Rm$ may be proved similarly, we only deal with bounded domains. 
Let $\Omega$ be a bounded domain and $u$ a local minimizer. 
For $\psi \in \cX_\phi(\Omega)$ and $\eps_j \downarrow 0$, 
consider the approximation $\psi_j\doteq \psi_{\eps_j}^u$ constructed in Lemma \ref{prop_approx}, 
that satisfies $\{\psi_j \neq u\} \Subset \Omega$. We first notice $I_\rho(u) \le I_\rho(\psi_j)$. 
Since $I_\rho$ is continuous on $\cX_{\phi}(\Omega)$ as observed in Subsection \ref{subsec_functional}, Lemma \ref{prop_approx} implies $I_\rho(\psi_j) \to I_\rho(\psi)$ and thus 
$I_{\rho} (u) \leq I_\rho(\psi)$. Whence, $u = u_\rho$. 
Also, if $u$ is a classical solution to \eqref{borninfeld}, then an integration by parts shows that \eqref{equi_BI_1} holds.

	We next prove that (iv) $\Rightarrow$ (ii) $\Rightarrow$ (i) $\Rightarrow$ (iii) $\Rightarrow$ (iv).\\[0.2cm]
(iv) $\Rightarrow$ (ii) is obvious. \\[0.2cm]
(ii) $\Rightarrow$ (i).

	Since $u= u_\rho$, the integrability $(1-|Du|^2)^{-1/2} \in L^1_\loc(\Omega)$ follows by Proposition \ref{lem_basicL2}. 
By density and the dominated convergence theorem, it is enough to prove (i) for $\eta \in C^1_c(\Omega)$. 
Fix an open set $\Omega'$ satisfying $\{ \eta \neq 0\} \Subset \Omega' \Subset \Omega$, 
and choose a strictly spacelike extension $\bar \phi$ of $\phi$, for instance the solution to \eqref{borninfeld} for $\rho =0$ as in Theorem \ref{teo_bartniksimon_2}. 
Since $\sup_{\Omega'}|D \bar \phi| < 1$, for $|t|$ small enough, the function $\psi \doteq \bar \phi + t \eta \in \cX_{\phi}(\Omega)$ is strictly spacelike and thus  
	\[
	\int_\Omega \frac{D u \cdot (Du - D \bar \phi - tD\eta)}{\sqrt{1-|Du|^2}} \di x = \langle \rho, u- \bar \phi - t\eta \rangle.
	\]	
Differentiating at $t=0$ gives \eqref{equi_BI_1}.\\[0.2cm]		
(i) $\Rightarrow$ (iii).

	Identity \eqref{equi_BI_3} follows immediately from \eqref{equi_BI_1} since $u-\psi \in \lip_c(\Omega)$. 
To show that \eqref{equi_BI_3} implies $u = u_\rho$, first observe that $|Du|<1$ a.e on $\Omega$, in view of the first property in \eqref{equi_BI_1}. 
Let $\psi \in \cX_{\phi} (\Omega)$ with $\{ \psi \neq u \} \Subset \Omega$. 
Apply Remark \ref{rem_reversedCS} and \eqref{equi_BI_3} to deduce 
	\[
	\int_\Omega \left( \sqrt{1-|D\psi|^2} - \sqrt{1-|Du|^2}\right)\di x 
	\le 
	\int_\Omega \frac{Du \cdot (Du-D\psi)}{\sqrt{1-|Du|^2}}\di x = \langle \rho, u-\psi \rangle,
	\]
which can be rewritten as $I_\rho(u) \le I_\rho(\psi)$. Hence, $u$ is a local minimizer and thus it coincides with $u_\rho$.\\[0.2cm]
(iii) $\Rightarrow$ (iv).

	We recall \eqref{eq:2}, argue by contradiction and 
suppose that there exist $\psi \in \cX_\phi(\Omega)$ and $\delta >0$ such that
	\begin{equation}\label{contrad_BI}
	\int_\Omega \frac{Du \cdot (Du - D\psi)}{\sqrt{1-|Du|^2}}\di x \leq \langle \rho, u-\psi \rangle - \delta.	
	\end{equation}
Select $\Omega' \Subset \Omega$ satisfying 
	\begin{equation}\label{assu_omegaprimo}
	\int_{\Omega \backslash \Omega'} \left|\frac{Du \cdot (Du - D\psi)}{\sqrt{1-|Du|^2}} \right|\di x < \frac{\delta}{4},
	\end{equation}
which is possible by \eqref{eq_L1test}. Fix a sequence $\eps_j \downarrow 0$ and 
consider the approximation $\psi_j$ for $\psi$ constructed in Lemma \ref{prop_approx} 
by choosing $u_j = u$ for each $j$. 
By construction, $\psi_j \equiv u$ on $\Omega \setminus \Omega_j$ for some $\Omega_j \Subset \Omega$, 
and, without loss of generality, we can assume that $\Omega' \subset \Omega_j$ as well as 
$D\psi_j \to D\psi$ a.e. $\Omega$. 
From $\psi_j \to \psi$ strongly in $\cX(\Omega)$, we get
	\begin{equation}\label{conv-rho-u-psij}
		\langle \rho, u- \psi_j \rangle \to \langle \rho, u-\psi \rangle \qquad \text{as } \, j \to \infty.
	\end{equation}
Also, by \eqref{basic_L2bound} in Proposition \ref{lem_basicL2} and the dominated convergence theorem,
	\begin{equation}\label{eq_lebe}
	\int_{\Omega'} \frac{Du \cdot (Du - D\psi_j)}{\sqrt{1-|Du|^2}}\di x 
	\to 
	\int_{\Omega'} \frac{Du \cdot (Du - D\psi)}{\sqrt{1-|Du|^2}}\di x \qquad \text{as } \, j \to \infty.
	\end{equation}
By the definition of $\psi_j$, 
	\begin{equation}\label{def_psieps}
	Du-D\psi_j = (Du-D\psi) \cdot \mathbb{1}_{V_j}, \qquad \text{where} \qquad V_j \doteq \big\{ |u-\psi| \ge \eps_j \big\},
	\end{equation}
hence from \eqref{contrad_BI} and \eqref{conv-rho-u-psij}, we infer 
	\[
	\begin{aligned}
	\langle \rho, u-\psi_j \rangle - \delta & \geq  
	\disp \int_\Omega \frac{Du \cdot (Du - D\psi)}{\sqrt{1-|Du|^2}}\di x  - o_j(1) \\
	& =  \disp \int_{\Omega \backslash \Omega'} \frac{Du \cdot (Du - D\psi)}{\sqrt{1-|Du|^2}}\di x 
	+ \int_{\Omega'} \frac{Du \cdot (Du - D\psi_j)}{\sqrt{1-|Du|^2}}\di x - o_j(1) \qquad \text{by \eqref{eq_lebe}}\\
	& \ge  \disp - \frac{\delta}{4} + \int_{\Omega'} \frac{Du \cdot (Du - D\psi_j)}{\sqrt{1-|Du|^2}}\di x - o_j(1) 
	\qquad \text{by \eqref{assu_omegaprimo}}\\
	& = \disp - \frac{\delta}{4} + \int_{\Omega_j} \frac{Du \cdot (Du - D\psi_j)}{\sqrt{1-|Du|^2}}\di x 
	- \int_{\Omega_j \backslash \Omega'} \frac{Du \cdot (Du - D\psi_j)}{\sqrt{1-|Du|^2}}\di x - o_j(1) \\
	& =  \disp - \frac{\delta}{4} + \langle \rho, u-\psi_j \rangle 
	- \int_{(\Omega_j \backslash \Omega') \cap V_j} \frac{Du \cdot (Du - D\psi)}{\sqrt{1-|Du|^2}}\di x - o_j(1) \\
	&\quad   \disp \text{by \eqref{equi_BI_3} and \eqref{def_psieps}} \\
	& \ge  \disp - \frac{\delta}{2} + \langle \rho, u-\psi_j \rangle  - o_j(1) \qquad \text{by \eqref{assu_omegaprimo}}, \\
	\end{aligned}
	\]
a contradiction if $j$ is large enough. 
\end{proof}	

\begin{remark}
For $\Omega = \R^m$, a different proof of (iii) $\Rightarrow$ (iv) was given in \cite[Theorem 6.4]{bpd}.
\end{remark}

\section{Main tools}
\label{sec-main-tools}

In this section, we prove the tool results needed to obtain our existence/non-existence theorems detailed in the Introduction. The main achievements herein are Theorem \ref{teo_removable} on removable singularities, Theorem \ref{teo_nolight} about the nonsolvability of \eqref{borninfeld}, the $L^2$-estimates for the second fundamental form in Proposition \ref{prop_inte_estimate} and Corollary \ref{cor_secondfund}, and Theorem \ref{teo_higherint} guaranteeing the higher integrability of $w_\rho$. 
To reach our goals, we first need to regularize $\rho$ and $u_\rho$, a device which will also be necessary in Section \ref{prf-thms}.

\subsection{Setup for our strategy} \label{subsec_strategy}

According to Remark \ref{rem_duals}, defining $p = q'$ it holds 
	\[
	\cM(\Omega) + W^{-1,p}(\Omega) \subset \cX(\Omega)^* \qquad \text{for each } \left\{ \begin{array}{ll}
	p \in [p_1',\infty) & \text{if $\Omega$ is bounded,} \\[0.2cm]
	p \in [p_1',2_*] & \text{if $\Omega = \R^m$.}
	\end{array}\right. 
	\]
We shall hereafter restrict to 
	\begin{equation*}
	\rho \in \cM(\Omega) + L^{p}(\Omega) \qquad \text{for } \, p \in (1,2_*],  
	\end{equation*}
where $L^{p}(\Omega) \subset W^{-1,p}(\Omega)$ is the set of pairs $(v,0)$ as in Remark \ref{rem_duals}. 
\begin{quote}
\emph{Since $2_* = 1$ when $m=2$, hereafter the space $L^p(\Omega)$ is tacitly assumed to be empty when $p \in (1,2_*]$ and $m=2$.} 
\end{quote}
%
Notice that $\cM(\Omega) + L^{p}(\Omega) \hookrightarrow \cX(\Omega)^*$ provided that $p_1$ is sufficiently large. For instance, we may (and henceforth do) choose 
	\begin{equation}\label{def_p_1}
	p_1 = 3 \quad \text{if } \, m=2, \qquad p_1 = \max\{2^*,m\} + p' \quad \text{if } \, m \ge 3. 
	\end{equation}
By a standard mollifying argument (see \cite[Chapter 2]{Ponce-mono}) and Young's inequality, 
for given 
	\[
	\rho = \mu + f \in \cM(\Omega) + L^{p}(\Omega)
	\]
we can find sequences of functions $g_j,f_j \in C^\infty(\overline\Omega)$ such that, setting $\mu_j \doteq  g_j \di x$ 
and recalling $p=q'$,
	\[
	\begin{array}{c}
	\|\mu_j\|_{\cM(\Omega)} \le \|\mu\|_{\cM(\Omega)}, \qquad \|f_j\|_{L^{p}(\Omega)} \le \|f\|_{L^{p}(\Omega)} \\[0.2cm]
	\text{$\mu_j \rightharpoonup \mu$ weakly in $\cM(\Omega)$, $\qquad f_j \to f$ strongly in $L^{p}(\Omega)$ (hence, in $\cX(\Omega)^*$).}
	\end{array}
	\]
Define $\rho_j \doteq \mu_j + f_j$. When $\Omega = \Rm$, the construction via convolution also guarantees, for each $\e>0$, the existence of $R_\e>0$ such that \eqref{tight} holds for $\{\mu_j\}$. Moreover, up to replacing $\rho,f$ by $\rho \mathbb{1}_{B_j}$ and $f \mathbb{1}_{B_j}$ and using a diagonal argument, we can assume that $g_j, f_j \in C^\infty_c(\R^m)$.

	Fix $\phi \in \mathcal{S}(\partial \Omega)$ if $\Omega$ is bounded, and denote the minimizer of $I_{\rho_j}$ by $u_j$. 
Because of Theorem \ref{teo_bartniksimon_2} or \cite[Theorem 1.5 and Remark 3.4]{bpd}, respectively 
if $\Omega$ is bounded or if $\Omega = \R^m$, $u_j$ is a smooth solution to \eqref{borninfeld} 
with Lorentzian mean curvature $H_j \doteq - ( g_j + f_j ) $ (thus, $u_j$ minimizes $I_{\rho_j}$ with $\rho_j = -H_j \di x$). 
Write $w_j \doteq (1-|Du_j|^2)^{-1/2}$. Proposition \ref{1001} yields $u_j \to u_\rho$ strongly in $\cX(\Omega)$ and $\langle \rho_j, u_j \rangle \to \langle \rho,u_\rho \rangle$.
%
%
Therefore, using Proposition \ref{prop_localglobal}, to show that $u_\rho$ weakly solves \eqref{borninfeld} it is enough to prove that
	\begin{equation}\label{limit-eq}
		\lim_{j \to \infty} 
		\int_\Omega w_j Du_j \cdot D\eta \, \di x = \int_\Omega w_\rho Du_\rho \cdot D\eta \, \di x 
		\qquad \forall \, \eta \in \Lip_c(\Omega).
	\end{equation}
Since $\| Du_j \|_\infty \leq 1$ and we may assume $Du_j \to Du_\rho$ a.e. on $\Omega$, identity \eqref{limit-eq} follows from Vitali's convergence theorem (see \cite[Theorem 3.1.9]{Willem-3}) provided that $\{w_j\}$ is locally uniformly integrable in the following sense.
\begin{definition} 
Let $\Omega \subset \R^m$ be an open subset. We say that a subset $\mathcal{W} \subset L^1_\loc(\Omega)$ is 
\emph{locally uniformly integrable on $\Omega$} if, for each $\Omega' \Subset \Omega$ and $\eps > 0$, there exists $\delta = \delta(\eps, \Omega')$ such that 
	\[
	A \subset \Omega' \ \text{measurable}, \ |A| < \delta \qquad \Longrightarrow \qquad \int_{A} |w|\di x < \eps \quad \forall \, w \in \mathcal{W}.
	\]
\end{definition}

By de la Vall\'ee-Poussin's Theorem (see, for instance, \cite[Theorem 3.1.10]{Willem-3}), 
$\mathcal{W}$ is locally uniformly integrable if and only if there exists a compact  exhaustion $\{\Omega_k\}_{k=1}^\infty$ of $\Omega$, that is, $\Omega_k \Subset \Omega$, $\Omega_k \uparrow \Omega$, and increasing convex functions $f_k :\R^+_0 \to \R^+_0$ such that 
	\[
	\lim_{t \to \infty} \frac{f_k(t)}{t} = +\infty, \qquad \sup_{w \in \mathcal{W}} \int_{\Omega_k} f_k(|w|) \di x < \infty \quad \forall \, k.
	\] 
The purpose of the next subsections is to obtain a local uniform integrability for $\{w_j\}$. 
We begin by studying the behavior of $u_\rho$ in regions where $\rho$ is singular.

\subsection{Removable and unremovable singularities}

To our knowledge, 
the only removable singularity theorem for the prescribed Lorentzian mean curvature equation is the one in \cite{Mi92_sing}. 
The theorem considers maximal graphs $u$ that are smooth and strictly spacelike in a domain $\Omega' \backslash E$, where $E \Subset \Omega'$ is compact. Under the assumption that the $p$-capacity of $E$ is zero for some $p \in (1,m]$, and that 
\begin{equation}\label{eq_integr}
\int_{\Omega'\backslash E} w^{\frac{p}{p-1}} \di x < \infty, 
\end{equation}
then $u$ can be smoothly extended to a spacelike maximal solution on $\Omega'$. 
In particular, by the known relation between Hausdorff measure and capacity (cf. \cite{evansgariepy}), 
compact subsets $E$ with $\haus^{m-p}_\delta(E) = 0$ are removable for maximal graphs satisfying \eqref{eq_integr}. 
However, the proof seems not easy to extend to more general measures $\rho \neq 0$, 
and currently we are unable to prove an a-priori estimate yielding \eqref{eq_integr}. 
Therefore, we take a different approach. 
Our contribution is the following result, which applies to any measure and only needs a local uniform integrability 
for the sequence of energy densities $\{w_j\}$.

\begin{theorem}[\textbf{Removable singularity}]\label{teo_removable}
Assume $\Omega \subset \R^m$ is either a bounded domain with $m \geq 2$ or $\R^m$ with $m \geq 3$. Let
	\[
	\rho \in \cM(\Omega) + L^p(\Omega), \qquad p \in (1,2_\ast],  
	\] 
and, if $\Omega$ is bounded, let $\phi \in \mathcal{S}(\partial \Omega)$. Choose $\{p_1,\rho_j,u_j,w_j\}$ as in Subsection \ref{subsec_strategy}. 
Suppose that $E \Subset \Omega$ is a compact set with $\haus^1_\delta(E) = 0$. Then, for every open subset $\Omega' \subset \Omega$,
	\[
	\begin{array}{c}
	\text{$\{w_j\}$ is locally uniformly} \\
	\text{integrable on $\Omega'\backslash E$}
	\end{array}
	\quad \Longrightarrow \quad \begin{array}{c}
	\text{$\{w_j\}$ is locally uniformly integrable on $\Omega'$, and} \\[0.2cm]
	\disp \int_{\Omega'} \frac{Du_\rho \cdot D\eta}{\sqrt{1-|Du_\rho|^2}} = \langle \rho, \eta \rangle \quad \forall \, \eta \in \lip_c(\Omega').	
	\end{array}
	\] 
In particular, if $\{w_j\}$ is locally uniformly integrable on $\Omega \backslash E$, then $u_\rho$ weakly solves \eqref{borninfeld}.
\end{theorem}

\begin{remark}\label{rem_sharp_remov}
The above requirements on $E$ cannot be weakened to $\haus_\delta^1(E) < \infty$. Indeed, consider the example in Corollary \ref{prop_nosolve}, and set $E = \overline{xy}$. Since $u=u_\rho$ has no light segments in $\Omega \backslash \overline{xy}$, the energies $\{w_j\}$ are locally uniformly integrable there. This can be shown by combining Lemma \ref{lem_simplecompact} with \cite[Lemma 2.1]{bartniksimon}, proceeding as in \cite[Proof of Theorem 4.1]{bartniksimon}. However, $u_\rho$ does not solve \eqref{borninfeld}, so $E$ is not removable. 
As a related example, one can see the nice \cite[Example 2]{klyachin_mixed}. 
\end{remark}

The result is a consequence of the next lemma, which estimates the growth of $w$ on balls centered at a given point.

\begin{lemma}\label{lem_growthinte}
Let $\Omega \subset \R^m$ be an open set, $H \in C^\infty(\Omega)$ and $u \in C^\infty(\Omega)$ be a 
strictly spacelike solution of
	\[
	-\diver \left( \frac{Du}{\sqrt{1-|Du|^2}} \right) = \rho \doteq -H \di x \qquad \text{on } \, \Omega.
	\]
For any given $y \in \Omega$, define
	\[
	J_y(s) \doteq \int_{B_s(y)} \frac{\di x}{\sqrt{1-|Du|^2}}, \qquad 0 < s < \di_\delta(y, \partial \Omega). 
	\]
Then, for each $0 < s  < t < \di_\delta(y, \partial \Omega)$, it holds
	\begin{equation}\label{ineq_I_simple_2}
	J_y(s) \le s \left[ \frac{J_y(t)}{t} + |\rho|(B_t(y)) \right].
	\end{equation}
	\end{lemma}

\begin{proof}
Let $\varphi \in \lip_c(\Omega)$. Up to a translation, we may assume $u(y) = 0$. Let $M$ be the graph of $u$. Recalling \eqref{eq:1}, we first test $\Delta_M u = H w$ against $u \varphi$ and integrate by parts to deduce
	\[
	\int \varphi \|\nabla u\|^2 \, \rd x_g = - \int u \varphi Hw \, \rd x_g  
	- \int \la u\nabla u, \nabla \varphi \ra \rd x_g.
	\] 
We set $o = y$ in \eqref{def_ello} and write $\ell (x) = \ell_y(x)$. Multiplying the equation $\Delta_M \ell^2 = 2m +  H \bar{D} l^2 \cdot \mathbf{n}$ in \eqref{eq_elle2} by $\varphi$ and integrating by parts we get 
	\[
	2m \int \varphi \, \rd x_g = 
	- 2 \int \ell \la \nabla \ell, \nabla \varphi \ra \rd x_g  - \int \varphi H \bar D l^2 \cdot {\bf n} \, \rd x_g.
	\]
Noting that $\ell^2(x) = r^2(x) - u^2(x)$ and $u(y) = 0$, and using the identities
	\[
		\ell \nabla \ell = r \nabla r - u \nabla u, \quad 
		w^2 = 1 + \| \nabla u\|^2, \quad 
		\bar{D} l^2 \cdot \mathbf{n} = 2w \left[ r \left( Du, Dr \right) - u \right],
	\]
we infer 
	\begin{equation}\label{eq_w_iter}
	\begin{array}{l}
\disp	m \int \varphi w^2 \rd x_g = m \int \varphi \, \rd x_g + m \int \varphi \| \nabla u \|^2 \rd x_g 
			\\[0.4cm]
\disp 	\qquad = - \int \ell \la \nabla \ell, \nabla \varphi \ra \rd x_g - \int \varphi H w \left[ r (Du, Dr) - u \right] \rd x_g
			\\[0.4cm]
\disp 	\qquad -m \int u \varphi H w \, \rd x_g - m \int \la u \nabla u , \nabla \varphi \ra \rd x_g
			\\[0.4cm]
\disp	\qquad = - \int \la r \nabla r + (m-1) u \nabla u , \nabla \varphi \ra \rd x_g 
			- \int \varphi H w \left[ r (Du,Dr) + (m-1) u \right] \rd x_g.
	\end{array}
	\end{equation}
First, since $\| \nabla \varphi \| \leq w |D\varphi|$, $|(Du,Dr)| \leq 1$ and $|u| \le r$ due to $\| Du \|_\infty \leq 1$, 
we get
	\[
	\begin{aligned}
	\langle r \nabla r + (m-1)u \nabla u, \nabla \varphi \rangle & \le  \|r \nabla r + (m-1)u \nabla u\|\|\nabla \varphi\| \\
	& \le  \disp m r \max\{ \|\nabla r\|, \|\nabla u\| \}\|\nabla \varphi\| \leq  m r |D\varphi| w^2. 
	\end{aligned}
	\]
Setting
	\[
		T_\rho(\varphi) \doteq - \frac{1}{m} \int \varphi H w \big[ r(Du,Dr) + (m-1)u \big] \rd x_g,
	\]
we deduce from \eqref{eq_w_iter} the following inequality:
	\begin{equation}\label{eq_fundamental_1}	
		\int \varphi w^2 \, \rd x_g \le \int |D\varphi| r w^2  \, \rd x_g + T_\rho(\varphi) .
	\end{equation}

	Let $ 0 < s < t < \di_\delta (y, \partial \Omega)  $ and consider, for $\eps>0$ small enough, 
	\[
	\varphi(x) \doteq \left( \min \left\{ 1, \frac{s + \eps - r(x)}{\eps} \right\} \right)_+ \in \Lip_c ( B_t(y) ) 
	\subset \Lip_c (\Omega).
	\]
From $|u| \le r$, $|(Du, Dr)| \le 1$ on the support of $\varphi$, $|\varphi| \le 1$ and $\di x_g = w^{-1}\di x$, 
and using the coarea formula, we get  
	\[
	|T_\rho(\varphi)| \le \int_{B_{s+\eps}(y)}r|H| w \di x_g 
	= \int_0^{s+\eps} \sigma \left[ \int_{\partial B_\sigma(y)} |H| \di \haus^{m-1}_\delta \right] \di \sigma.
	\]
Letting $\e \to 0$ and observing that 
	\[
		\int |D \varphi | r w^2 \rd x_g 
		= \int |D \varphi | r w \, \rd x \to s \int_{ \partial B_s (y) }  w \, \rd \haus^{m-1}_\delta  
	\]
for a.e. $s$, from \eqref{eq_fundamental_1}, we obtain 
\[
	\int_{B_s(y)} w \, \rd x \le s \int_{\partial B_s(y)} w \, \rd \haus^{m-1}_\delta 
	+ \int_0^{s} \left[ \sigma \int_{\partial B_\sigma(y)} |H| \di \haus^{m-1}_\delta \right] \di \sigma \qquad \text{for a.e.} \, s \in [0,t].
\]
By the coarea formula, the above inequality can also be rewritten as 
	\[
		- \frac{\rd}{\rd s}  \frac{J_y(s)}{s} \leq \frac{1}{s^2} \int_0^s \sigma f_y(\sigma) \di \sigma \qquad \text{for a.e. $s \in (0,t]$},
	\]
where 
	\[
	f_y(\sigma) = \int_{\partial B_\sigma(y)} |H| \di \haus^{m-1}_\delta.
	\]
Integrating on $[s,t]$ and using Tonelli's Theorem, we deduce
	\[
	\begin{array}{lcl}
	\disp - \frac{J_y(t)}{t} + \frac{J_y(s)}{s} & \le & \disp \int_s^t \frac{1}{\tau^2} \left\{ \int_0^\tau \sigma f_y(\sigma) \di \sigma \right\} \di \tau \\[0.4cm] 
	& = & \disp \int_0^t \sigma f_y(\sigma) \left\{ \int_{\max\{s,\sigma\}}^t \frac{\di \tau}{\tau^2} \right\} \di \sigma \\[0.4cm]
	& \le & \disp \int_0^t \sigma f_y(\sigma) \left[ - \frac{1}{\tau}\right]_\sigma^t \di \sigma \le \disp \int_0^t \sigma f_y(\sigma) \frac{1}{\sigma} \di \sigma \\[0.4cm]
	& = & \disp \int_0^t f_y(\sigma)\di \sigma = \int_{B_t(y)} |H|  \di x = |\rho|\big(B_t(y)\big),
	\end{array}
	\]
which proves \eqref{ineq_I_simple_2}. 
%
\end{proof}

	Using Lemma \ref{lem_growthinte} and a covering argument, we shall prove Theorem \ref{teo_removable}:

\begin{proof}[Proof of Theorem \ref{teo_removable}]
Write $\rho = \mu + f$ with $\mu \in \cM(\Omega)$ and $f \in L^p(\Omega)$. Referring to Subsection \ref{subsec_strategy}, for $m=2$ the term $f$ does not appear, and our choice of $p_1$ imply that $\rho \in \cX(\Omega)^*$. Let $\mu_j,f_j$ be as therein, thus $\mu_j \rightharpoonup \mu$ weakly in $\cM(\Omega)$ and $f_j \to f$ strongly in $L^p(\Omega)$. Choose $0 < R_0 \le \di_\delta(E, \partial \Omega)/20$. The relative compactness of $B_{10R_0}(E)$ implies that $\rho_j = \mu_j + f_j \di x \rightharpoonup \rho$ weakly in $\cM(B_{10R_0}(E))$, so in particular there exists a constant $C_{\cM}$ such that  
	\begin{equation}\label{bdd-tv-rhoj}
		\left\| \rho_j \right\|_{\cM (B_{10R_0}(E)) }   \leq C_{\cM} \quad \text{for each $j \geq 1$}.
	\end{equation}
Write $\rho_j = -H_j \di x$. By Proposition \ref{lem_basicL2}, there exists a constant $\mathcal{C}(R_0)$, depending on $\phi$, $R_0$, $\|\rho\|_{\cX^*}$ such that
	\begin{equation}\label{eq_uniboundw}
	\int_{B_{4R_0}(E)} w_j \, \di x \le \mathcal{C}(R_0). 
	\end{equation}
For $x \in B_{R_0}(E)$ and $s \in (0,R_0]$, set
	\[
	J_{x,j}(s) \doteq \int_{B_s(x)} w_j \di x.
	\]
Note that \eqref{eq_uniboundw} implies $J_{x,j}(R_0) \le \mathcal{C}(R_0)$ for all $j \geq 1$ and $x \in B_{R_0} (E)$, 
hence Lemma \ref{lem_growthinte} and \eqref{bdd-tv-rhoj}, \eqref{eq_uniboundw} ensure that for all $x \in B_{R_0} (E)$, $j \geq 1$ and $s \in (0,R_0)$,
	\begin{equation*}
	J_{x,j}(s) \le s \left[ \frac{\mathcal{C}(R_0)}{R_0} + |\rho_j|(B_{R_0}(x)) \right] \le \mathcal{C}_1 s,  
	\end{equation*}
for some $\mathcal{C}_1(R_0,\mathcal{C}(R_0), \mathcal{C}_\cM)$. By our assumption $\haus^1_\delta(E) = 0$ and since $E$ is compact, for given $\tau>0$ we can cover $E$ with finitely many balls $\{B_k\}_{k=1}^N$, $B_k = B_{r_k}(x_k)$ satisfying $r_k < R_0$ and $\sum_{k} r_k \le \tau$. We can also assume that $x_k \in B_{R_0}(E)$ for each $k$. 
Therefore, for each fixed $\e > 0$ we can take $\tau > 0$ small enough to satisfy
	\begin{equation*}
	\int_{\bigcup_{k =1}^N B_k} w_j \di x 
	\le 
	\sum_{k=1}^N J_{x_k,j}(r_k) \le \mathcal{C}_1 \sum_{k=1}^N r_k \le \mathcal{C}_1\tau < \frac{\eps}{2}.
	\end{equation*}
Let $\Omega'' \Subset \Omega'$ be a relatively compact subset. By defining $U \doteq \bigcup_{k=1}^N B_k$, our assumption yields that $\{w_j\}$ is uniformly integrable on $\Omega'' \backslash U$. 
Thus, there exists $\delta>0$ such that $A \subset \Omega'' \backslash U$ and $|A|< \delta$ imply $\int_A w_j \di x < \e/2$. 
Then, for each subset $A \subset \Omega''$ with $|A| < \delta$, 
	\[
	\int_A w_j \di x \le \int_{A \cap U} w_j \di x + \int_{A \backslash U} w_j \di x < \frac{\eps}{2} + \frac{\eps}{2} = \eps, 
	\]
which means that $\{w_j\}$ is uniformly integrable on $\Omega''$. In particular, \eqref{limit-eq} holds for every fixed $\eta \in \lip_c(\Omega')$ by Vitali's Theorem. 
%
%
%
\end{proof}

We next consider singularities which cannot be removed. 
While the examples in Section \ref{sec_counterexamples} show that solutions to \eqref{borninfeld} may possess light segments when $\rho \in L^q(\Omega)$ and $q < m-1$, 
we shall now prove that such solutions exhibit, in a sense, a ``borderline'' behavior.

\begin{theorem}\label{teo_nolight}
Let $\Omega \subset \Rm$ be either a bounded domain with $m \geq 2$ and $\phi \in \mathcal{S}(\partial \Omega)$, or $\Omega = \R^m$ with $m \geq 3$. 
Let $\rho \in \cX(\Omega)^*$, and assume that the minimizer $u_\rho$ has a light segment $\overline{xy} \subset\Omega$ with $u_\rho(y) - u_\rho(x) = |y-x|$. Then, for each $\alpha > 0$, $u_\rho$ also minimizes the functional $I_{\rho_\alpha}$ with 
	\[
	\rho_\alpha = \rho + \alpha( \delta_y - \delta_x),
	\]
but it does not solve \eqref{borninfeld} weakly for $\rho_\alpha$.
\end{theorem}

\begin{proof}
For simplicity, in this proof we remove the subscript $\rho$ and denote by $I \doteq I_\rho$ and $u \doteq u_\rho$. 
We also write $I_\alpha \doteq I_{\rho_\alpha}$ and denote its minimizer by $u_\alpha$. 
We argue by contradiction and assume that $u_\alpha \neq u$ for some $\alpha > 0$. 
By uniqueness of the minimizer, we infer
	\[
		I(u)  
		=  I_\alpha(u) + \alpha\big[u(y) - u(x)\big] 
		>  \disp I_\alpha(u_\alpha) + \alpha \big[u(y) - u(x)\big],
	\]
	which implies
	\[
		u(y) - u(x) < \frac{I(u)-I_\alpha(u_\alpha)}{\alpha}.
	\]
Similarly, 	
	\[
		I_\alpha(u_\alpha)  =  I(u_\alpha) - \alpha \big[u_\alpha(y) - u_\alpha(x)\big] 
		 >  \disp I(u) - \alpha \big[u_\alpha(y) - u_\alpha(x)\big],	
	\]
thus, 
	\[
	u_\alpha(y) - u_\alpha(x) > \frac{I(u)- I_\alpha(u_\alpha)}{\alpha}.
	\]
Therefore, $u_\alpha(y) - u_\alpha(x) > u(y)-u(x) = |y-x|$, contradicting the fact that $u_\alpha \in \cX_\phi(\Omega)$.

	We have therefore proved that $u= u_\alpha$ for each $\alpha>0$. 
By Theorem \ref{teo_bartniksimon_2}, pick a strictly spacelike extension $\bar \phi$ of $\phi$, so that, in particular, $|y-x| - \bar\phi(y) + \bar\phi(x) > 0$. Since $u$ minimizes $I$, we see from Proposition \ref{lem_basicL2} that 
	\[
	\begin{aligned}
	\int_\Omega \frac{Du \cdot (Du-D\bar \phi)}{\sqrt{1-|Du|^2}}\di x 
	\le \la \rho, u- \bar \phi \ra 
	&= \la \rho_\alpha, u- \bar \phi \ra - \alpha \la (\delta_y-\delta_x), u-\bar \phi \ra 
	\\
	& = \la \rho_\alpha, u- \bar \phi \ra - \alpha \left[ |y-x| - \bar\phi(y) + \bar\phi(x) \right] \\
	& < \la \rho_\alpha, u- \bar \phi \ra.
	\end{aligned}
	\]
Therefore, due to Proposition \ref{prop_localglobal}, $u$ does not solve \eqref{borninfeld} for $\rho_\alpha$.
\end{proof}

\begin{remark}\label{rem_impo_obse}
When $\Omega$ is bounded, the condition $\phi \in \Spa(\partial \Omega)$ only serves to guarantee that $u_\rho$ does not solve \eqref{borninfeld} with charge $\rho_\alpha$. In other words, the fact that $u_\rho$ minimizes $I_{\rho_\alpha}$ only requires $\cX_\phi(\Omega) \neq \emptyset$. 
\end{remark}


\subsection{Local second fundamental form estimate} \label{subsec_locSF}

We first observe that $W^{2,q}_\loc$ estimates, for $q \ge 1$, are not to be expected for general $\rho$. An easy counterexample can be produced building on the explicit expression of $u_\rho$ for $\rho = \omega_{m-1} \delta_0$ in the unit ball with zero boundary values:
	\begin{equation}\label{e:rad-example}
		u_{\rho}(x) = \int_{|x|}^1 \frac{\di t}{\sqrt{t^{2m-2} + 1}} \qquad \text{on } \, B_1(0) \subset \R^m.
	\end{equation}
%
%
%
%
%
%
%
%
%
Fix $R \in (0,1)$ and let $s \in (0, \|u\|_\infty)$, be the constant value of $u$ on $\partial B_R(0)$. 
Then, the function $u_s \doteq \min\{u,s\}$ solves 
\[
\left\{ \begin{aligned}
\diver \left( \frac{Du_s}{\sqrt{1-|Du_s|^2}} \right) &= 
-R^{1-m} \haus^{m-1}_{\delta} \measrest \partial B_R(0) & &\text{on } \, B_1(0), \\
u_s&=0 & & \text{on } \, \partial B_1(0).
\end{aligned}\right.
\]
Clearly, $u_s \not \in W_{\loc}^{2,q}$ for any $q \ge 1.$ Note however that, by explicit computation, $u \in W^{2,q}(B_1(0))$ for each $q \in [1,m)$.

	It is reasonable to guess that $u_\rho \in W^{2,2}_\loc(\Omega)$ provided that $\rho \in L^2(\Omega)$. 
Indeed, a stronger estimate holds. First, observe that integrating \eqref{norm_second} on a domain $\Omega'$ we get
	\begin{equation}\label{inte_secondfund}
			\int_{M'} \|\SF\|^2 \di x_g = 
			\int_{\Omega'} w 
			\left\{ | D^2u|^2 
			+ 2 w^2 \left| D^2 u \left( Du, \cdot \right) \right|^2   
			+ w^4 \left[D^2u(Du,Du)\right]^2\right\} \di x,
	\end{equation}
where $M'$ denotes the graph of $u=u_\rho$ over $\Omega'$. In this subsection, we prove local second fundamental form estimates for the graph of $u_\rho$ in regions $\Omega'$ where $\rho \in L^2$. 
We first treat the case of smooth $\rho = - H \rd x$ with $H \in C^\infty(\overline{\Omega})$, and then argue by approximation. 
Let $u$ be the smooth solution to \eqref{borninfeld}, with subscript $\rho$ suppressed. The search for estimates on $\|\SF\|^2$ is natural in view of the Jacobi equation
\begin{equation}\label{eq_jacobi_standard}
		\Delta_M w = -\la \nabla H, \partial_0^\parallel \ra + \|\SF\|^2 w,
\end{equation}
see \cite[p.519]{amr}, which will be our starting point. First, we need to decompose $\|\SF\|^2$ into convenient pieces.
Recall that, by \eqref{eq:1}, $\|\SF\|^2 = w^{-2} \|\nabla^2 u\|^2$. We rewrite $\| \nabla^2 u \|^2$ as follows:

	\begin{lemma}
		Assume $\rd u(x) \neq 0$ at $x \in M$ and set $\nu \doteq \nabla u / \| \nabla u \|$ in a neighborhood of $x$. 
		Denote by $A$ the traceless second fundamental form of the level set $\{u =u(x)\}$ in the direction $-\nu$ 
		and write $u_{\nu \nu} \doteq \nabla^2u (\nu,\nu)$. Then 
		\begin{equation}\label{norm-u-hessian}
			\begin{aligned}
				\|\nabla^2 u\|^2 = & 
				\|\nabla u\|^2 \|A\|^2 + \frac{1}{m-1} \left( H^2 w^2 - 2H w \, u_{\nu\nu}\right) \\
				&  + \frac{m}{m-1} \big\| \nabla \|\nabla u\|\big\|^2 + \frac{m-2}{m-1} \big\| \nabla^\top  \|\nabla u\|\big\|^2,
			\end{aligned}
		\end{equation}
	where $\nabla^\top$ stands for the component of $\nabla$ tangent to the level set $\{u= u(x)\}$. 
	\end{lemma}

	\begin{proof}
Consider an orthonormal frame $\{\nu, e_\alpha\}$, $2 \le \alpha \le m$ on $M$. 
We denote by $u_{ij}$ the components of $\nabla^2 u$ in the above frame. Then, 
	\[
		\langle \nabla \|\nabla u\|, e_\alpha \rangle 
		= u_{\alpha \nu}, \qquad \langle \nabla \|\nabla u\|, \nu \rangle = u_{\nu\nu},
	\]
thus
\begin{equation}\label{eq_kato}
\|\nabla^2 u\|^2 = \sum_{\alpha,\beta=2}^m u_{\alpha\beta}^2 + 2\|\nabla^\top  \|\nabla u\| \|^2 + u_{\nu\nu}^2.
\end{equation}

	Next, it follows from the definition of $A$ that
	\[
		\|\nabla u\| A_{\alpha\beta} = u_{\alpha \beta} - \frac{\sum_{\gamma=2}^m u_{\gamma\gamma}}{m-1} \delta_{\alpha\beta}.
	\]
Splitting the norm of the matrix $[u_{\alpha\beta}]$ into its trace and traceless parts, and recalling \eqref{eq:1}, we get 
	\[
		\begin{aligned}
			\sum_{\alpha,\beta=2}^m u_{\alpha\beta}^2 
			& = 
			\|\nabla u\|^2\|A\|^2 + \frac{1}{m-1} \left(\sum_{\alpha=2}^m u_{\alpha\alpha}\right)^2 
			= \|\nabla u\|^2\|A\|^2 + \frac{(\Delta_M u - u_{\nu\nu})^2}{m-1} \\
			& = \|\nabla u\|^2\|A\|^2 + \frac{1}{m-1} \left( H^2 w^2 -2 Hw u_{\nu\nu} + u_{\nu\nu}^2 \right).
		\end{aligned}
	\]
Inserting this into \eqref{eq_kato} and noting that $\| \nabla \|\nabla u\|\|^2 = \| \nabla^\top  \|\nabla u\|\|^2 + u_{\nu\nu}^2$, 
we obtain \eqref{norm-u-hessian}. 
	\end{proof}

\begin{remark}
When $H=0$, we obtain the classical refined Kato inequality for harmonic functions
	\[
	\|\nabla^2 u\|^2 \ge \frac{m}{m-1} \big\| \nabla \|\nabla u\| \big \|^2.
	\]
\end{remark}

	It is convenient to rewrite the equations in terms of the hyperbolic angle
	\[
		\beta \doteq \mathrm{arcch} \, w = \log \left( w + \sqrt{w^2 - 1}\right).
	\]
Note that $w \mapsto \beta$ is a diffeomorphism on $\{\di u \neq 0\}$. The identities
	\[
		w = \ch \beta, \qquad \|\nabla u\| = \sqrt{w^2-1} = \sh \beta, \qquad u_{\nu\nu} = \langle \nabla \|\nabla u\|, 
		\nu \rangle = \ch \beta \langle \nabla \beta, \nu \rangle,
	\]
\eqref{norm-u-hessian} and the fact that $\SF = w^{-1} \nabla^2 u = 0$ a.e. on the set $\{\di u = 0\}$ due to Stampacchia's theorem 
allow us to rewrite $\|\SF\|^2 =w^{-2} \| \nabla^2 u \|^2$ as
		\begin{equation}\label{ide_SF_cong}
		\begin{aligned}
			\|\SF\|^2 
			& = 
			\left[ \frac{\sh^2 \beta}{\ch^2 \beta} \|A\|^2 + \frac{H^2}{m-1} - \frac{2H \langle \nabla \beta,\nu \rangle}{m-1}  
			 + \frac{m\|\nabla \beta\|^2}{m-1} + \frac{m-2}{m-1} \|\nabla^\top  \beta\|^2 \right] 
			\cdot \mathbb{1}_{\{\di u \neq 0\}}
		\end{aligned}
		\end{equation}
a.e. on $\Omega$. We therefore deduce that, for some constant $C = C(m)>0$, 
	\begin{equation}\label{equival_point}
	\|\SF\|^2 \le C(m) \left[\frac{\sh^2 \beta}{\ch^2 \beta} \|A\|^2 + \|\nabla \beta\|^2 + H^2\right] \cdot \mathbb{1}_{\{\di u \neq 0\}}
	\end{equation}
and that, for every $M' \Subset M$,
	\[
		\int_{M'} \|\SF\|^2 \rd x_g \le \mathcal{C} 
		\quad \Longleftrightarrow \quad 
		\int_{M' \cap \{\di u \neq 0\}} \left[\frac{\sh^2 \beta}{\ch^2 \beta} \|A\|^2 + \|\nabla \beta\|^2 + H^2\right] \rd x_g 
		\le \mathcal{C}',
	\]
where $\mathcal{C}$ and $\mathcal{C}'$ might be different, but with the same qualitative dependence on the data of our problem \eqref{borninfeld}. 
We point out that in \eqref{ide_SF_cong} the coefficient of $\| \nabla \beta \|^2$ is strictly larger than $1$ and 
this plays a crucial role in our analysis. 
Indeed, it guarantees that the coefficient of $\| \nabla \beta \|^2$ in inequality \eqref{est-Y-bel1} below is positive for sufficiently small $\e>0$, allowing us to bound the integral of $\| \nabla \beta \|^2$ from above.

We next rewrite the Jacobi equation in a way that is more suited to our purposes. We begin with the following 

\begin{lemma}\label{lem_jacobi}
Define
\begin{equation}\label{def_Y}
Y \doteq \frac{\nabla w - H \nabla u}{w} \qquad \text{on } \, M.
\end{equation}
Then,
\begin{equation}\label{jacobibetter}
\diver_M Y = \disp \|\SF\|^2 -H^2 - \la Y, \frac{\nabla w}{w} \ra.
\end{equation}
\end{lemma}

\begin{proof}
We first claim that
\begin{equation}\label{jacobi}			
\Delta_M w = \Big(\| \SF \|^2 - H^2\Big) w +\diver_M \big( H \nabla u \big) 	\qquad {\rm on} \ M.
\end{equation}
Indeed, the identity follows from \eqref{eq_jacobi_standard} once we recall that $\partial_0^\parallel = -\nabla u$ and that 
\[
\la \nabla H, \nabla u \ra = \diver_M(H \nabla u) - H \Delta_M u = \diver_M(H \nabla u) - H^2 w.
\]
From \eqref{jacobi} we therefore obtain 
	\[
		\Delta_M \log w = \|\SF\|^2 - H^2 - \frac{\|\nabla w\|^2}{w^2} + \diver_M \left(\frac{H\nabla u}{w}\right) 
		+ H \la \frac{\nabla u}{w}, \frac{\nabla w}{w}\ra,
	\]
which is \eqref{jacobibetter} up to rearranging terms.
\end{proof}

	By \eqref{ide_SF_cong}, $\nabla u = \sh \beta \nu$ and $\nabla w / w = \sh \beta \nabla \beta / \ch \beta$, 
 we rewrite the vector field $Y$ as
		\begin{equation}\label{Y}
		Y = \frac{\sh \beta}{\ch \beta}\big( \nabla \beta - H \nu\big)
		\end{equation}
and $\diver_M Y$ as 
	\[
				\begin{aligned}
					\diver_M Y  
					= & 
					\left[ \frac{\sh^2 \beta}{\ch^{2} \beta} \|A\|^2 - \frac{m-2}{m-1}H^2 
					- \frac{2}{m-1} H \langle \nabla \beta, \nu \rangle \right. \\
					& \quad  \left. + \frac{m}{m-1} \|\nabla \beta\|^2 
					+ \frac{m-2}{m-1} \|\nabla^\top  \beta\|^2 - \frac{\sh \beta}{\ch \beta} \la Y, \nabla \beta  \ra \right] 
					\cdot \mathbb{1}_{\{\di u \neq 0\}}
				\end{aligned}
	\]
a.e. on $\Omega$. By \eqref{Y} with $0 \leq \sh \beta / \ch \beta \leq 1$ and Cauchy-Schwarz's and Young's inequalities, we have 
	\[
	\begin{aligned}
		\left| \frac{\sh \beta}{\ch \beta} \la Y, \nabla \beta \ra \right| 
		& \le \left\| \nabla \beta - H \nu \right\| \|\nabla \beta \| 
		\le \| \nabla \beta \|^2 + |H| \|\nabla \beta\| \le (1+ \eps) \|\nabla \beta\|^2 + \frac{4}{\e}H^2,
		\\
		\left| H \la \nabla \beta, \nu \ra  \right| & \le |H| \| \nabla \beta \| \le \frac{1}{2\e} |H|^2 + \frac{\e}{2} \| \nabla \beta \|^2.
	\end{aligned}
	\]
Thus there exist constants $C_m, C_{m,\eps} $ such that 
	\begin{equation}\label{est-Y-bel1}
		\begin{aligned}
			\diver_M Y 
			\geq 
			\left[ 
			\frac{\sh^2 \beta}{\ch^2 \beta} \|A \|^2 - C_{m,\e} H^2 
			+ \left\{ \frac{1}{m-1} - \frac{C_m \e}{2} \right\} \| \nabla \beta \|^2 
			\right] \cdot \mathbb{1}_{\{\di u \neq 0\}}
		\end{aligned}
	\end{equation}
a.e. on $\Omega$. We notice from the smoothness of $Y$, $H$ and from estimate \eqref{est-Y-bel1} that the function $\| \nabla \beta \|^2 \mathbb{1}_{ \{ \di u \neq 0 \} } $ is integrable on the graph of $u$.

	\begin{proposition}\label{prop_inte_estimate}	
	There exists a constant $C = C_m>0$ such that, for every $\varphi \in \lip_c(\Omega)$, 
	\begin{equation}\label{inte_deltag_0}
		\int_M \varphi^2 \|\SF\|^2 \di x_g \le C_m \left( \int_M \|\nabla \varphi\|^2 \di x_g + \int_M \varphi^2 H^2 \di x_g \right).
	\end{equation}	
	\end{proposition}	
	
\begin{proof}
We test \eqref{est-Y-bel1} with the function $\varphi^2$ to obtain 
	\begin{equation}\label{int-est-1}
		\begin{aligned}
			& \int_{ \{\di u \neq 0\} } \left[ \frac{\sh^2 \beta}{\ch^2 \beta} \| A \|^2 + 
			\left\{ \frac{1}{m-1} - \frac{C_m \e}{2}  \right\} \| \nabla \beta \|^2 \right] 
			\varphi^2
			\rd x_g
			\\
			\leq \ & \int \varphi^2  \diver_M Y \, \rd x_g 
			+  C_{m,\e} \int H^2 \varphi^2 \, \rd x_g 
			\\
			= \ & -2 \int  \varphi \la \nabla \varphi , Y \ra \rd x_g + C_{m, \e} \int H^2 \varphi^2 \, \rd x_g.
		\end{aligned}
	\end{equation}
Since, from its very definition, $Y = 0$ on $\{\di u = 0\}$, and since $0 \leq \sh \beta / \ch \beta \leq 1$, 
using Cauchy-Schwarz's and Young's inequalities we see from \eqref{Y} that
	\[
		\begin{aligned}
			\left| \varphi \la \nabla \varphi , Y \ra \right| 
			&\leq 
			\left\{ \left| \varphi \la \nabla \varphi , \nabla \beta \ra \right| 
			+ \left| \varphi H \la \nabla \varphi , \nu \ra \right| \right\} \mathbb{1}_{ \{ \rd u \neq 0 \}  }
			\\
			& \leq \frac{1}{2\e} \| \nabla \varphi \|^2 + \frac{\e}{2} \varphi^2 \| \nabla \beta \|^2 \mathbb{1}_{ \{ \rd u \neq 0 \} }
			+ \frac{1}{2} \varphi^2 H^2 + \frac{1}{2} \| \nabla \varphi \|^2.
		\end{aligned}
	\]
Recalling that $\| \nabla \beta \|^2 \mathbb{1}_{ \{ \di u \neq 0\} }$ is integrable, 
it follows from \eqref{int-est-1} that 
	\[
		\begin{aligned}
			&\int_{ \{ \di u \neq 0 \} } 
			\left[ \frac{\sh^2 \beta}{\ch^2 \beta} \| A \|^2 + 
			\left\{  \frac{1}{m-1} - \frac{C_m\e}{2} - \e \right\} \| \nabla \beta \|^2 \right] \varphi^2
			\rd x_g 
			\\
			\leq \ & C_{m,\e} \int H^2 \varphi^2 \, \rd x_g + C_\e \int \| \nabla \varphi \|^2 \, \rd x_g.
		\end{aligned}
	\]
Choosing a small $\e>0$ and taking \eqref{equival_point} into account, we readily deduce \eqref{inte_deltag_0} 
and complete the proof. 
\end{proof}

Using \eqref{inte_secondfund}, \eqref{inte_deltag_0} and the approximation in Subsection \ref{subsec_strategy}, 
we prove the following result. We recall that, for $m=2$, the space $L^p(\Omega)$ below is meant to be empty. 
%

\begin{corollary}\label{cor_secondfund}
Let $\Omega \subset \Rm$ be a domain. Assume that either 
	\begin{itemize}
	\item[-] $m\ge 2$, $\Omega$ is bounded, $\FF \subset \Spa(\partial \Omega)$ is a compact subset, and $\phi \in \FF$; 
	\item[-] $m \ge 3$, $\Omega = \Rm$. 
	\end{itemize}
Fix $\mathcal{I}_1, \mathcal{I}_2 \in \R^+$, $\Omega' \Subset \Omega$ and, for $\e > 0$, define $\Omega'_\e \doteq \left\{ x \in \Omega' : \di_\delta(x, \partial \Omega') > \e \right\}$. Let $p \in (1,2_*]$. Then, there exists a constant
	\begin{equation}\label{eq_C_secf}
	\mathcal{C} = \left\{ \begin{array}{ll}
	\mathcal{C}(\Omega,\FF,m,\diam_\delta (\Omega), p,\mathcal{I}_1, \mathcal{I}_2, \e, \di_\delta(\Omega', \partial \Omega) ) & \quad \text{if $\Omega$ is bounded,} \\[0.2cm]
	 \mathcal{C} (m,p,\mathcal{I}_1, \mathcal{I}_2, \e, |\Omega'|_\delta) & \quad \text{if $\Omega = \R^m$}
	 \end{array}\right.
	 \end{equation}
such that for each $\rho \in \cM(\Omega) + L^p(\Omega)$ satisfying 
	\[
	\|\rho\|_{\cM(\Omega)+L^p(\Omega)} \le \mathcal{I}_1, \qquad \| \rho \|_{L^2(\Omega')} \leq \mathcal{I}_2,
	\]
it holds
		\begin{equation}\label{bdd:u_rho}
			\begin{aligned}
				\int_{\Omega'_\e} 
				&\left\{ w_{\rho} \left| D^2 u_\rho \right|^2 
				+ w_\rho^3 \left| D^2u_\rho \left( Du_\rho, \cdot  \right) \right|^2 
				+ w_\rho^5 \left[ D^2 u_\rho \left( Du_\rho, Du_\rho \right) \right]^2 
				\right\} \rd x 
				\leq \mathcal{C}.
			\end{aligned}
		\end{equation}
In particular, 
		\begin{equation}\label{bound_Dlogw}
		\begin{aligned}
			\int_{\Omega'_\e} 
			\frac{1}{w_\rho} \left\{\left|D \log w_\rho \right|^2 + \left| Dw_\rho \cdot Du_\rho \right|^2 \right\} \di x \leq \mathcal{C}, 
			\\
			\int_{\Omega'_\e} \left\{\left|D \log w_\rho \right| + \left| Dw_\rho \cdot Du_\rho \right| \right\} \di x \leq \mathcal{C}.
		\end{aligned}
		\end{equation}
	\end{corollary}

	\begin{proof}
We choose $p_1$ as in \eqref{def_p_1} to guarantee that $\rho \in \cX(\Omega)^*$, and referring to Subsection \ref{subsec_strategy}, we approximate $\rho$ through convolution obtaining $\{ \rho_j \}$ with $\rho_j = - H_j \rd x$ and $H_j \in C^\infty(\overline{\Omega})$ (resp. $H_j \in C^\infty_c(\R^m)$). 
Let $u_j$ be the smooth solution to \eqref{borninfeld} with source $\rho_j$, and write $w_j \doteq (1 - |Du_j|^2)^{-1/2}$. For the ease of notation, in the next arguments we write $u,w$ instead of $u_\rho,w_\rho$. Proposition \ref{1001} yields $u_j \to u$ strongly in $W^{1,q}(\Omega)$, for each $q \in [1,\infty)$ if $\Omega$ is bounded 
and each $q \in [2^\ast , \infty)$ if $\Omega = \Rm$. We fix $\varphi \in C^1_c( \Omega' )$ so that $\varphi \equiv 1$ on $\Omega'_\e$ and 
$ | D \varphi (x) | \leq 2/\e$ for each $x \in \Omega$. 	
From 
	\[
		\| \nabla \varphi \|^2 = |D \varphi|^2 + w_j^2 \left( Du_j \cdot D\varphi \right)^2 
		\leq \left( 1 + w_j^2|Du_j|^2 \right) \left| D \varphi \right|^2 = w_j^2 \left| D \varphi \right|^2, 
	\]
\eqref{inte_secondfund} and Proposition \ref{prop_inte_estimate} with $u_j$, it follows that 
	\[
		\begin{aligned}
			\int_{\Omega} \varphi^2 
			&w_j \left\{ 
			\left| D^2u_j \right|^2 + 2 w_j^2 \left| D^2u_j \left( Du_j, \cdot \right) \right|^2 
			+ w_j^4 \left[ D^2u_j \left( Du_j, Du_j \right) \right]^2  
			 \right\} \rd x 
			 \\
			 &\leq \mathcal{C}_m \int_{\Omega} 
			 \left\{ w_j \left| D \varphi \right|^2 + \varphi^2 \rho_j^2 w_j^{-1}   \right\} \rd x.
		\end{aligned}
	\]
Combining this estimate with $w_j \geq 1$, the properties of $\varphi$ and Proposition \ref{lem_basicL2}, 
we find a constant $\mathcal{C}$ as in \eqref{eq_C_secf} such that 
	\begin{equation}
	\label{bdd:u_j}
		\sup_{j \geq 1}\int_{\Omega'_\e} 
		w_j \left\{ 
		\left| D^2u_j \right|^2 + 2 w_j^2 \left| D^2u \left( Du_j, \cdot \right) \right|^2 
		+ w_j^4 \left[ D^2u_j \left( Du_j, Du_j \right) \right]^2  
		\right\} \rd x \leq \mathcal{C}. 
	\end{equation}
In particular, $\{u_j\}$ is bounded in $W^{2,2}(\Omega'_\e)$ and we may suppose that $u_j \rightharpoonup u$ weakly in $W^{2,2}(\Omega'_\e)$. From the $W^{1,q}$ convergence we may also suppose that 
$u_j(x) \to u(x)$, $Du_j (x) \to Du(x)$ and $w_j(x) \to w(x)$ for a.e. $x \in \Omega'_\e$.

	Fix $N>1$ and set 
	\[
		w_{N,j} (x) \doteq \min \{ w_j(x) , N \}, \quad w_{N} (x) \doteq \min \{ w(x) , N \}. 
	\]
By \eqref{bdd:u_j}, we have 
	\begin{equation}\label{bdd:Nu_j}
		\sup_{j \geq 1, N > 1}
		\int_{\Omega'_\e} 
		w_{N,j} \left\{ 
		\left| D^2u_j \right|^2 + 2 w_{N,j}^2 \left| D^2u_j \left( Du_j, \cdot \right) \right|^2 
		+ w_{N,j}^4 \left[ D^2u_j \left( Du_j, Du_j \right) \right]^2  
		\right\} \rd x \leq \mathcal{C}. 
	\end{equation}
From $w_j \to w$, $Du_j \to Du$ a.e. on $\Omega$, $w_{N,j} \leq N$ and $|Du_j| \leq 1$, 
it follows that for every $1 \leq i_1,i_2 \leq m$ and $q \in [1,\infty)$,
	\begin{multline*}
		\left\| w_{N,j} - w_{N} \right\|_{L^q(\Omega'_\e)}
		+ \left\| w_{N,j}^{3/2} (u_j)_{i_1} - w_{N}^{3/2} u_{i_1} \right\|_{L^q(\Omega'_\e)} 
		\\
		+ \left\| w_{N,j}^{5/2} (u_j)_{i_1} (u_j)_{i_2} - w_{N}^{5/2} u_{i_1} u_{i_2} 
		\right\|_{L^q(\Omega'_\e)} \to 0.
	\end{multline*}
Since $u_j \rightharpoonup u$ weakly in $W^{2,2} (\Omega_\e')$, for any $\psi \in L^\infty(\Omega'_\e)$, 
we see
	\[
		\begin{aligned}
			\int_{\Omega_\e'} 
			w_{N,j}^{1/2} (u_j)_{i_1,i_2} \psi \, \rd x 
			&\to 
			\int_{\Omega_\e'} w_{N}^{1/2} u_{i_1,i_2} \psi \, \rd x, 
			\\ 
			\int_{\Omega_\e'} 
			w_{N,j}^{3/2} (u_{j})_{i_1,i_2} (u_j)_{i_{3}} \psi \, \rd x 
			&\to 
			\int_{\Omega_\e'} 
			w_{N}^{3/2} u_{i_1,i_2} u_{i_{3}} \psi \, \rd x,
			\\
			\int_{\Omega_\e'}
			w_{N,j}^{5/2} (u_j)_{i_1,i_2} (u_j)_{i_3} (u_j)_{i_4} 
			\psi \, \rd x 
			&\to
			\int_{\Omega_\e'}
			w_{N}^{5/2} u_{i_1,i_2} u_{i_3} u_{i_4} 
			\psi \, \rd x .
		\end{aligned}	
	\]
Thus, the density of $L^\infty(\Omega_\e')$ in $L^2(\Omega_\e')$ yields 
	\[
		\begin{aligned}
			&w_{N,j}^{1/2} D^2u_j \rightharpoonup w_{N}^{1/2} D^2u, \quad  
			w_{N,j}^{3/2} D^2u_j \left( Du_j, \cdot \right) \rightharpoonup 
			w_{N}^{3/2} D^2 u \left( Du, \cdot \right), 
			\\
			&w_{N,j}^{5/2} D^2 u_j \left( Du_j, Du_j \right) \rightharpoonup 
			w_{N}^{5/2} D^2 u \left( Du, Du \right)
		\end{aligned}
	\]
weakly in $L^2(\Omega_\e')$. Hence, by \eqref{bdd:Nu_j} and the lower semicontinuity of the norm, we obtain 
	\[
		\sup_{N > 1} \int_{\Omega'_\e} 
		w_{N} \left\{ 
		\left| D^2u \right|^2 + 2 w_{N}^2 \left| D^2u \left( Du, \cdot \right) \right|^2 
		+ w_{N}^4 \left[ D^2u \left( Du, Du \right) \right]^2  
		\right\} \rd x \leq \mathcal{C}. 
	\]
By letting $N \to \infty$ and using the monotone convergence theorem, \eqref{bdd:u_rho} holds.

	The first in \eqref{bound_Dlogw} readily follows from 
	\[
	|D \log w|^2 = w^{4}\left| D^2u(Du, \cdot) \right|^2, \qquad D w \cdot Du = w^3 D^2 u( Du, Du)
	\]
a.e. on $\Omega$. On the other hand, the second in \eqref{bound_Dlogw} is derived from H\"older's inequality and Proposition \ref{lem_basicL2}:
	\[
	\begin{aligned}
		&\int_{\Omega'_\e} \left\{\left|D \log w \right| + \left| Dw \cdot Du \right| \right\} \di x \\
		\leq \ & 
		\left( \int_{\Omega'_\e} w \rd x \right)^{1/2} 
		\left( \int_{\Omega'_\e} \frac{1}{w} \left\{\left|D \log w \right|^2 + \left| Dw \cdot Du \right|^2 \right\} 
		\di x \right)^{1/2} \le \mathcal{C}.
	\end{aligned}
	\]
This concludes the proof. 
\end{proof}



\subsection{Higher regularity (log-improvements)}\label{sec_higher_regularity}

The main purpose of this subsection is to obtain a local higher integrability for the energy density of $u_\rho$ of the type
\[
\int_{\Omega''} w_\rho \log(1+w_\rho) \, \di x \le \mathcal{C},
\]
where $\Omega'' \Subset \Omega' \Subset \Omega$ and we assume $\rho \in L^2(\Omega')$. We shall split the discussion according to whether $m=2$ or not. Indeed, in dimension $2$ we can rely on the trace theorem and Corollary \ref{cor_secondfund} to ensure higher integrability with no additional assumption. As we shall see, the situation is more intricate in dimension $m \ge 3$. For $m=2$, we prove:

\begin{theorem}\label{teo_higher_m2}
Let $\Omega \subset \R^2$ be a bounded domain, let $\FF \subset \Spa(\partial \Omega)$ be compact and $\phi \in \FF$. Fix  $\Omega' \Subset \Omega$ and for $\e>0$, 
define $\Omega'_\e \doteq \{x \in \Omega' : \di_\delta(x, \partial \Omega') > \e\}$. Let $\rho \in \cM(\Omega)$ satisfy
	\[
	\|\rho\|_{\cM(\Omega)} \le \mathcal{I}_1, \qquad \|\rho\|_{L^2(\Omega')} \le \mathcal{I}_2
	\]
for some constants $\mathcal{I}_1,\mathcal{I}_2$. 
Then, there exists $\mathcal{C} = \mathcal{C}(\Omega,\FF, \diam_\delta(\Omega),\mathcal{I}_1,\mathcal{I}_2,\e,\di_\delta(\Omega',\partial \Omega)) $
such that the energy density $w_\rho = (1-|Du_\rho|^2)^{-1/2}$ satisfies 
	\begin{equation}\label{high-int-wrho}
	\int_{\Omega'_\e} w_\rho \log \left(1+w_\rho \right)\di x \le \mathcal{C}.
	\end{equation}
In particular, $u_\rho$ weakly solves \eqref{borninfeld} on $\Omega'$.
\end{theorem}

\begin{proof}
We fix $p_1$ as in \eqref{def_p_1} and, as in the proof of Corollary \ref{cor_secondfund}, we find $\rho_j \doteq - H_j \rd x$ satisfying 
$H_j \in C^\infty(\overline{\Omega})$ and  
	\[
		\sup_{j \geq 1} \|\rho_j\|_{\cM(\Omega)} \le \mathcal{I}_1, \qquad 
		\sup_{j \geq 1}\|\rho_j\|_{L^2(\Omega')} \le \mathcal{I}_2.
	\]
Denote by $u_j$ the minimizer of $I_{\rho_j}$ and by $w_j = (1-|Du_j|^2)^{-1/2}$. 
We recall that, for each Radon measure $\mu$ on $\R^m$, the following trace inequality holds for some constant $C = C(m)$, see  \cite[Corollary 1.1.2]{mazyashap}:
	\begin{equation}\label{eq_trace}
	\int \varphi \, \di \mu 
	\le 
	C \left[ \sup_{x \in \R^m, r>0} \frac{\mu(B_r(x))}{r^{m-1}}\right] \int |D\varphi| \, \di x 
	\qquad \forall \, \varphi \in C^\infty_c(\R^m).
	\end{equation}
By Proposition \ref{lem_basicL2}, 
		\[
			\int_{\Omega'} w_j   \rd x \le \mathcal{C}_1\big(\Omega,\FF, \diam_\delta(\Omega),\mathcal{I}_1,\di_\delta(\Omega',\partial \Omega)\big) ,
		\]
while, by Corollary \ref{cor_secondfund},
		\[
			\int_{\Omega'_{\e/2}} \left| D\log w_j \right| \di x 
			\le 
			\mathcal{C}_2\big(\Omega,\FF, \diam_\delta(\Omega),\mathcal{I}_1,\mathcal{I}_2,\e,\di_\delta(\Omega',\partial \Omega)\big) .
		\]	
Hereafter, $\mathcal{C}_j$ will denote a constant depending on the same data as $\mathcal{C}_2$. We consider the measure $\mu \doteq w_j \di x \measrest \Omega'_\e$ and set  
$\varphi \doteq \psi \log(1+w_j)$ for a  cut-off function $\psi$ satisfying $\psi \equiv 1$ on $\Omega'_{3\e/4}$ 
and ${\rm supp} \, \psi \subset \Omega'_{\e/2}$. By \eqref{ineq_I_simple_2}, 
for each $x \in \Omega'_{\e/4}$ and $r< \e/8$,
		\[
		\mu(B_r(x)) = \int_{B_r(x)\cap \Omega_\e'} w_j \di x \le r \left[ \frac{8}{\e} \int_{B_{\e/8}(x)} w\, \rd x 
		+ C(\mathcal{I}_1) \right] \le \mathcal{C}_3 r.
		\]
On the other hand, if $x \in \Omega'_{\e/4}$ and $r \ge \e/8$, then 
		\[
		\mu(B_r(x)) \le \int_{\Omega'} w_j \rd x \le \mathcal{C}_1 \le \mathcal{C}_4r.
		\]
When $x \not \in \Omega'_{\e /4}$ and $r < \e /8$, we clearly have $\mu(B_r(x)) = 0$. 
Hence, $\mu(B_r(x)) \le \mathcal{C}_5 r$ for each $x \in \R^2$, $r>0$. 
Our dimensional restriction, \eqref{eq_trace} and \eqref{bound_Dlogw} imply
		\[
		\begin{split}
			\int_{\Omega'_\e} w_j \log \left(1+w_j \right) \di x 
			& \le 
			\mathcal{C}_6 \int_{\R^2} \left| D \left( \psi \log \left(1+w_j \right) \right) \right| \di x 
			\\
			& \le \mathcal{C}_6 \int_{\Omega'_{\e/2}} \left[ \log \left(1+w_j\right) |D\psi| 
			+ \psi \left|D\log w_j \right| \right] \di x 
			\le \mathcal{C}_7.
		\end{split}
		\]
Now \eqref{high-int-wrho} follows by letting $j \to \infty$ and using Fatou's lemma. Finally, the fact that $u_\rho$ weakly solves \eqref{borninfeld} on $\Omega'$ follows from \eqref{high-int-wrho} and 
the discussion in Subsection \ref{subsec_strategy}. 
%
\end{proof}

	We remark that Theorem \ref{teo_higher_m2} cannot be extended to dimension $m \ge 4$. 
Otherwise, the entire proof of Theorem \ref{teo_BI_local_s2} in Subsection \ref{subsec_prf_2dim} would work for dimension $m \ge 4$, which contradicts the example in Remark \ref{rem_nosolve} (cf. Theorem \ref{teo_nolight}). 
In dimension $m=3$, proving that $\{w_j\}$ is locally uniformly integrable on a subdomain where $\rho$ is of class $L^2$ 
is an open problem, which seems challenging.

Nevertheless, when $\rho \in L^2(\Omega')$ we can prove a higher integrability for $w_\rho$ on $\Omega'_\eps$ which holds in any dimension, provided we further assume that the Lorentzian ball $L_R^\rho(\Omega'_\eps)$ defined in \eqref{def_LRo} satisfies
	\begin{equation}\label{eq_lorball}
	L_R^\rho(\Omega'_\eps) \Subset \Omega' \qquad \text{for some } \, R>0.
	\end{equation}
The requirement assures that Lorentzian balls of radius $R$ centered at every point $o \in \Omega'_\eps$ are relatively compact in $\Omega'$, and so it enables to localize by means of cut-off functions depending on the Lorentzian distance $\ell_o$ from $o$. Such cut-off functions are better behaved than those based on the Euclidean distance $r_o$, because $\Delta_M \ell_o$ can be controlled more efficiently than $\Delta_M r_o$.  More precisely, if $u \in \cX_\phi(\Omega)$ and $\phi \in \Spa(\partial \Omega)$, then from \eqref{eq_elle2} we get
	\begin{equation}\label{eq_deltal2}
	\|\nabla \ell_o^2\|^2 \le 4\ell^2_o + 16w^2 \left| x- o \right|, \qquad 
	|\Delta_M \ell_o^2| \le 2 m + 4wH \left| x - o \right|.
	\end{equation}
On the other hand, computing the gradient and Laplacian of $r_o$ and using \eqref{eq_hessianLapla}, we get 
	\[
		|\Delta_M r_o^2| \le C( 1 + w^2 + |H|w).
	\]
As we will discuss in Remark \ref{rem_point} below, the advantage of using $\ell_o$ instead of $r_o$ is exactly the absence of the addendum $w^2$ in the upper bound \eqref{eq_deltal2} for $|\Delta_M \ell_o^2|$. 
The necessity to differentiate the cut-off function to second order will be apparent in the proof of the next theorem, 
the second main result of this subsection. Under assumption \eqref{eq_lorball}, inequality \eqref{ine_thebest} below provides both a log-improvement for the energy density and allows to improve, by a factor of $\log^q w$, the conclusion \eqref{bdd:u_rho} in Corollary \ref{cor_secondfund}.

	To introduce the statement, we first recall that, by Proposition \ref{lem_basicL2}, 
given $\Omega' \Subset \Omega$ and $\mathcal{I}_1$ such that $\rho = -H \rd x$ and $\|\rho\|_{\cM(\Omega)} \le \mathcal{I}_1$, 
the identity $\di x_g = w^{-1}\di x$ implies
\[
\int_M |H| w \, \di x_g \le \mathcal{I}_1, \qquad \int_{M'} w^2 \, \di x_g \le \mathcal{C}, 
\]
where $M'$ is the graph over $\Omega'$ and $\mathcal{C}$ is a constant depending on the data of our problem.

\begin{theorem}\label{teo_higherint}
Let $\Omega \subset \R^m$ be either
\begin{itemize}
\item[-] a bounded domain, $m \geq 2$, $\FF \subset \Spa(\partial \Omega)$ is compact and $\phi \in \FF$, or 
\item[-] $\Omega = \R^m$ and $m \geq 3$. 
\end{itemize}
Let 
	\[
	H \in C^\infty(\overline\Omega) \ \ \text{ if $\Omega$ is bounded,} \qquad H \in C^\infty_c(\R^m) \ \ \text{ if } \, \Omega = \R^m,
	\] 
define the measure $\rho = - H \di x$, and let $u \in \cX_\phi(\Omega)$ be the minimizer of $I_\rho$. Assume that
\begin{equation}\label{assu_int_Linfty}
\|u\|_{L^\infty(\Omega)} \le \mathcal{I}_0, \qquad \|\rho\|_{\cM(\Omega)+L^p(\Omega)} \le \mathcal{I}_1,
\end{equation}
for some constants $\mathcal{I}_0,\mathcal{I}_1 > 0$ and $p \in (1,2_*]$. Suppose that there exist two open subsets $\Omega'' \Subset \Omega' \Subset \Omega $ such that
\begin{equation}\label{assu_higherint}
 \qquad \int_{\Omega'} H^2 \frac{(1 + \log w)^{q_0+2}}{w} \di x \le \mathcal{I}_{2,q_0},
\end{equation}
for some $q_0 \in \mathbb{N} \cup \{0\}$ and $\mathcal{I}_{2,q_0} \in \R^+$, and that for some $R>0$ it holds
	\[
		L^\rho_R(\Omega'') \Subset \Omega'.
	\]
Then, there exists a constant 
	\begin{equation}\label{eq_behaconst}
		\mathcal{C} = \left\{ \begin{aligned} 
			&\mathcal{C}(\Omega,\FF,m,\diam_\delta(\Omega),\mathcal{I}_0, \mathcal{I}_1,q_0,\mathcal{I}_{2,q_0}, \di_\delta(\Omega',\partial \Omega), R)
			& &\text{if $\Omega$ is bounded}, \\
			&\mathcal{C}(m, p, \mathcal{I}_0, \mathcal{I}_1,q_0,\mathcal{I}_{2,q_0}, |\Omega'|_\delta, R)
			& &\text{if $\Omega = \Rm$}
		\end{aligned}\right.
	\end{equation}
such that	
\begin{equation}\label{ine_thebest}
		\int_{\Omega''} \frac{(1+ \log w)^{q_0}}{w} \left\{ \|\SF\|^2 + w^2 \log w \right\} \di x \le \mathcal{C}.
	\end{equation}
\end{theorem}

\begin{proof}
By Theorem \ref{teo_bartniksimon_2} or \cite[Theorem 1.5 and Remark 3.4]{bpd}, we know that $u$ is smooth and strictly spacelike. In particular, $L^\rho_s(\Omega'') \Subset L^\rho_t(\Omega'')$ if $0 \leq s < t$. 
Define $p_1$ as in \eqref{def_p_1}. We proceed by induction on $q \in \{ 0, \ldots, q_0\}$. Set for convenience 	
	\[
	\bar R \doteq \frac{R}{q_0+1},
	\]
and define the sequence 
	\[
		\Omega'' \doteq \Omega_{q_0+1} \Subset \Omega_{q_0} \Subset \ldots \Subset \Omega_1 \Subset \Omega_0 
		\Subset \Omega', \qquad \Omega_q \doteq L^\rho_{ (q_0+1 - q ) \bar R }(\Omega'') \ \text{for } \, q \ge 0.
	\]
Let $M_q$ be the graph of $u$ over $\Omega_q$. By rephrasing \eqref{assu_higherint} in terms of the graph metric and the hyperbolic angle $\beta$, 
there exists a constant $\bar{\mathcal{I}}_{2,q_0}$ only depending on $\mathcal{I}_{2,q_0}$ such that
	\[
		\int_{M_0} H^2 (1+\beta)^{q_0+2}  \le \bar{\mathcal{I}}_{2,q_0},
	\]
where, hereafter in the proof, integration on subsets of the graph of $u$ will always be performed with respect to the graph measure $\di x_g$, that will be omitted as far as no confusion arises. Hence, 
	\begin{equation}\label{eq_buondH2}
		\int_{M_0} H^2 (1+\beta)^{q+2} 
		\le \bar{\mathcal{I}}_{2,q_0} \qquad \text{for each } \, q \in \{0, 1, \ldots, q_0\}.
	\end{equation}

	As a starting point, observe that Proposition \ref{lem_basicL2} and \eqref{assu_int_Linfty} imply 
the existence of 
	\[
		\bar{\mathcal{I}}_{1,0} 
		= \left\{\begin{aligned}
			& \bar{\mathcal{I}}_{1,0} \big(\Omega,\FF, m,\diam_\delta(\Omega), p, \mathcal{I}_0, \mathcal{I}_1, \di_\delta(\Omega',\partial\Omega)\big)
			& &\text{if $\Omega$ is bounded},
			\\
			& \bar{\mathcal{I}}_{1,0} \big( m, p, \mathcal{I}_0, \mathcal{I}_1, |\Omega'|_\delta  \big)
			& &\text{if $\Omega = \Rm$},
		\end{aligned}\right.
	\]
such that
		\begin{equation}\label{Ho}\tag{$\mathscr{A}_0$}
		\int_{M_0} |H| \ch \beta  + \int_{M_0} \ch^2 \beta  \le \bar{\mathcal{I}}_{1,0}.
		\end{equation}
We shall prove the following inductive step:
	\begin{align}
		& \text{if there exists }
		\nonumber
		\\
		&\quad 
		\cJ_{1,q} = \left\{\begin{aligned}
			& \cJ_1 \big(\Omega,\FF, m, \diam_\delta(\Omega), p, \mathcal{I}_0, \mathcal{I}_1, \di_\delta(\Omega',\partial \Omega), 
			q_0,q, R \big) 
			& &\text{if $\Omega$ is bounded},\\
			& \cJ_1 \big(  m, p, \mathcal{I}_0, \mathcal{I}_1, |\Omega'|_\delta, q_0,q,  R \big) 
			& &\text{if $\Omega=\Rm$},
		\end{aligned}\right.
		\nonumber
		\\
		&
		\text{such that} & \nonumber \\
		&\quad \int_{M_q} |H|(1+\beta)^q\ch \beta 
		+ \int_{M_q} (1+\beta)^q \ch^2 \beta  
		\le \cJ_{1,q}, 
		\tag{$\mathscr{A}_q$}
		\label{Aq}
		\\		
		& \text{then there exists} \nonumber \\
		&\quad 
			\cJ_{2,q} = \left\{\begin{aligned}
			& \cJ_2 \big(\Omega,\FF, m, \diam_\delta(\Omega), p, \mathcal{I}_0, \mathcal{I}_1, \di_\delta(\Omega',\partial\Omega), 
			q_0,q, \cJ_{1,q}, R \big) 
			& &\text{if $\Omega$ is bounded},\\
			& \cJ_2 \big(  m, p, \mathcal{I}_0, \mathcal{I}_1, |\Omega'|_\delta, q_0,q, \cJ_{1,q}, R \big) 
			& &\text{if $\Omega=\Rm$},
			\end{aligned}\right.
		\nonumber \\
		&\text{such that} \nonumber \\
		&\quad \int_{M_{q+1}} (1+\beta)^q \|\SF\|^2 
		+ \int_{M_{q+1}} (1+\beta)^{q+1} \ch^2 \beta 
		\le \cJ_{2,q}. \tag{$\mathscr{B}_q$}
		\label{Bq}
	\end{align}
In view of \eqref{equival_point} and \eqref{eq_buondH2}, to obtain \eqref{Bq} from \eqref{Aq} it is enough to show that
		\[
		\int_{M_{q+1}\cap \{\di u \neq 0\}} (1+\beta)^q\left[\frac{\sh^2 \beta}{\ch^{2} \beta} \|A\|^2 + \|\nabla \beta\|^2 + \beta \sh^2 \beta\right]  \le \cJ_{2,q},
		\]
with $\cJ_{2,q}$ possibly different, but depending on the same data. 
We first show that $\text{\eqref{Bq}} \Rightarrow (\mathscr{A}_{q+1})$ for each $0 \le q \le q_0-1$: 
by \eqref{eq_buondH2} and Young's inequality,
	\[
		\begin{aligned}
			\int_{M_{q+1}} |H|(1+\beta)^{q+1}\ch \beta 
			&\le 
			\int_{M_{q+1}} H^2 (1+\beta)^{q+2} 
			+ \int_{M_{q+1}} (1+\beta)^{q}\ch^2 \beta  
			\\
			&\le \bar{\mathcal{I}}_{2,q_0} + \cJ_{2,q},
		\end{aligned}
	\]
hence $(\mathscr{A}_{q+1})$ holds with $\cJ_{1,q+1} \doteq \bar{\mathcal{I}}_{2,q_0} + 2 \cJ_{2,q}$.

Since we verified \eqref{Ho}, 
if the implication $\text{\eqref{Aq}} \Rightarrow \text{\eqref{Bq}}$ is proved, 
then the induction hypothesis implies $(\mathscr{B}_{q_0})$, which is equivalent to \eqref{ine_thebest}. \\[0.2cm]

	With the above preparation, it suffices to prove that $\text{\eqref{Aq}} \Rightarrow \text{\eqref{Bq}}$. 
For small $t > 0$, we consider a smooth approximation $\beta_t \in C^\infty(\Omega)$ of $\beta$ defined by
		\[
		\ch \beta_t \doteq \sqrt{w^2 + t} \quad 
		\Leftrightarrow \quad \beta_t = \log \left( \sqrt{w^2 +t} + \sqrt{w^2+t-1} \right).
		\]
Note that
	\begin{equation}\label{boundgradient}
	\begin{aligned}
		&\beta \le \beta_t \le \beta + 1  \ \ \text{for small enough $t$}, 
		& & \nabla \beta_t = 0 \quad \text{a.e. on $\{ \rd u = 0 \}$},
		\\
		&\beta_t \downarrow \beta, \ \ \  \|\nabla \beta_t\| \uparrow \|\nabla \beta\|\cdot \mathbb{1}_{\{\di u \neq 0\}} \ \ \text{as} \, t \downarrow 0, 
		& &\la \nabla \beta_t , \nabla \beta \ra \mathbb{1}_{\{\rd u \neq 0\}} \ge 0.
	\end{aligned} 
	\end{equation}	
Define also  		
	\begin{equation}\label{def-baru}
		\bar{u} \doteq u - \| u \|_\infty \leq 0.
	\end{equation}
We consider the smooth vector field $Y + \beta_t \nabla e^{\bar u}$, where $Y$ is defined in \eqref{def_Y}, 
and compute its divergence. 
For $\eps \in (0,1)$ to be specified later, we use \eqref{est-Y-bel1} to deduce that 
for some positive constants $C_m$ and $C_{m,\eps}$ depending, respectively, on $m$ and on $(m,\eps)$,
\begin{equation}\label{eq_starting_higher}
		\begin{aligned}
			\diver_M \left( Y + \beta_t \nabla e^{\bar u} \right) 
			\geq &  
			\left[ 
			\frac{\sh^2 \beta}{\ch^2 \beta} \|A \|^2 - C_{m,\e} H^2 
			+ \left\{ \frac{1}{m-1} - C_m  \e \right\} \| \nabla \beta \|^2 
			\right] \cdot \mathbb{1}_{\{\di u \neq 0\}} \\
			& + e^{\bar u} \la \nabla \beta_t, \nabla u \ra + \beta_t e^{\bar u} H \ch \beta + \beta_t e^{\bar u} \sh^2 \beta.
		\end{aligned}
\end{equation}
The reason for adding the field $\beta_t \nabla e^{\bar u}$ is the appearance of the last term in the above inequality, $\beta_t e^{\bar u} \sh^2 \beta$, which is of the order of $w^2 \log w$ for large $w$. Its control will give us the log-improvement for the energy density we search for. 
Hereafter, $C_m, C_{m,\eps}$ as well as the constants $C_q,C_{q,\eps}$, may vary from line to line.

We integrate \eqref{eq_starting_higher} against the test function
	\begin{equation}\label{phi1}
		\psi = \varphi^2 (1+\beta_t)^q, 
		\quad \varphi \in \Lip_c (\Omega_q), \quad \varphi^2 \in W^{2,\infty}(\Omega_q). 
	\end{equation}
By 
		\[
		\nabla \psi = (1+\beta_t)^q \nabla \varphi^2 + q \varphi^2 (1+\beta_t)^{q-1} \nabla \beta_t,
		\]
we see that 
	\[
			\begin{aligned}
				&\int_{\{\di u \neq 0\}} \varphi^2 (1+\beta_t)^q \left[ 
				\frac{\sh^2 \beta}{\ch^2 \beta} \|A \|^2 - C_{m,\e} H^2 
				+ \left\{ \frac{1}{m-1} - C_m  \e \right\} \| \nabla \beta \|^2 
				\right] 
				\\
				& \quad +  \int_M \varphi^2 (1+\beta_t)^q e^{\bar u} \la \nabla \beta_t, \nabla u \ra 
				+  \int_M \varphi^2 (1+\beta_t)^{q}\beta_t e^{\bar u} H\ch \beta 
				\\
				& \quad + \int_M \varphi^2 (1+\beta_t)^q e^{\bar u} \beta_t \sh^2 \beta  
				\\
				\le \ & 
				- \int_M  (1+\beta_t)^q \la \nabla \varphi^2, Y +  \beta_t \nabla e^{\bar u} \ra 
				- q \int_M \varphi^2 (1+\beta_t)^{q-1} \la \nabla \beta_t, Y +  \beta_t \nabla e^{\bar u} \ra.
			\end{aligned}
	\]
Rearranging the terms and using Cauchy-Schwarz's inequality together with \eqref{boundgradient}, we obtain
		\[
		\begin{aligned}
			&\int_{\{\di u \neq 0\}} \varphi^2 (1+\beta_t)^q \left[ 
			\frac{\sh^2 \beta}{\ch^2 \beta} \|A \|^2 + \left\{ \frac{1}{m-1} - C_m  \e \right\} \| \nabla \beta \|^2 
			\right] 
			\\
			&\quad + \int_M \varphi^2 (1+\beta_t)^q e^{\bar u} \beta_t \sh^2 \beta  
			\\
			\le \ & 
			- \int_M  (1+\beta_t)^q \la \nabla \varphi^2, Y +  \beta_t \nabla e^{\bar u} \ra 
			- q \int_M \varphi^2 (1+\beta_t)^{q-1} \la \nabla \beta_t, Y +  \beta_t \nabla e^{\bar u} \ra 
			\\
			& \quad +  \int_{\{\di u \neq 0\}} \varphi^2 (1+\beta_t)^q e^{\bar u} \|\nabla \beta\| \sh \beta 
			+ \int_M \varphi^2 (1+\beta_t)^{q+1}e^{\bar u}|H|\ch \beta 
			\\
			& \quad + C_{m,\eps}  \int_M \varphi^2 (1+\beta_t)^q H^2.
		\end{aligned}
		\]
From $\bar{u} \leq 0$ (see \eqref{def-baru}) and 
	\[
	\varphi^2 (1+\beta_t)^q e^{\bar u} \|\nabla \beta\| \sh \beta \le \eps \varphi^2 (1+\beta_t)^q \|\nabla \beta\|^2 
	+ \eps^{-1} \varphi^2 (1+\beta_t)^q \sh^2 \beta, 
	\]		
we infer
		\begin{equation}\label{start_1}
		\begin{aligned}
			&\int_{\{\di u \neq 0\}} \varphi^2 (1+\beta_t)^q \left[ 
			\frac{\sh^2 \beta}{\ch^2 \beta} \|A \|^2 + \left\{ \frac{1}{m-1} - C_m  \e \right\} \| \nabla \beta \|^2 
			\right] 
			\\
			& \quad + \int_M \varphi^2 (1+\beta_t)^q e^{\bar u} \beta_t \sh^2 \beta
			\\
			\le \ & 
			- \int_M  (1+\beta_t)^q \la \nabla \varphi^2, Y +  \beta_t \nabla e^{\bar u} \ra 
			- q \int_M \varphi^2 (1+\beta_t)^{q-1} \la \nabla \beta_t, Y +  \beta_t \nabla e^{\bar u} \ra 
			\\
			& \quad + \eps^{-1} \int_M \varphi^2 (1+\beta_t)^q \sh^2\beta + \int_M \varphi^2 (1+\beta_t)^{q+1}|H|\ch \beta 
			\\
			& \quad + C_{m,\eps}  \int_M \varphi^2 (1+\beta_t)^q H^2.
		\end{aligned}
		\end{equation}

	Because of \eqref{Aq}, \eqref{eq_buondH2} and the first in \eqref{boundgradient},
		\begin{equation}\label{eq_gHchg}
		\begin{aligned}
		\int_M \varphi^2 (1+\beta_t)^q \sh^2 \beta & \le C_q \|\varphi\|_\infty^2 \cJ_{1,q}, 
		\\		
		\int_M \varphi^2 (1+\beta_t)^{q+1} |H|\ch \beta
		& \le
		\frac{\|\varphi\|_\infty^2}{2} \left\{\int_{M_q} (1+\beta_t)^{q+2}H^2
		+ \int_{M_q} (1+\beta_t)^q \ch^2 \beta \right\} \\
		 & \le 
		 C_q \|\varphi\|^2_\infty \left[ \bar{\mathcal{I}}_{2,q_0} + \cJ_{1,q} \right].
		\end{aligned}
		\end{equation}
Notice that due to \eqref{Y},
	\[
		\| \nabla \varphi \|^2 \leq w^2 |D\varphi|^2= \ch^2\beta |D\varphi|^2, \qquad 
		\|Y\|^2 \cdot \mathbb{1}_{\{\rd u \neq 0\} } \le 2 \left[ \|\nabla \beta\|^2 + H^2 \right] \cdot \mathbb{1}_{\{\rd u \neq 0\}}.
	\]
Using $Y = 0$ a.e. on $\{ \rd u = 0\}$, Young's inequality and assumption \eqref{Aq}, we infer
		\[
			\begin{aligned}
			& - \int_M (1+\beta_t)^q \langle \nabla \varphi^2, Y \rangle  
			\\
			 \le \ & 
			\eps \int_{\{ \rd u \neq 0 \}} \varphi^2 (1+\beta_t)^q \left[ \|\nabla \beta\|^2 + H^2 \right]  
			+ \frac{4}{\eps} \int_{ \{ \rd u \neq 0\} } (1+\beta_t)^q \|\nabla \varphi\|^2  
			\\
			\le \ & \eps \int_{\{ \rd u \neq 0 \}} \varphi^2 (1+\beta_t)^q \left[ \|\nabla \beta\|^2 + H^2 \right]  
			+ 4 \e^{-1} \| D \varphi \|_{\infty}^2 \int_{M_q} (1+\beta_t)^q\ch^2 \beta  
			\\
			\le \  & \eps \int_{ \{ \rd u \neq 0\} } \varphi^2 (1+\beta_t)^q \left[ \|\nabla \beta\|^2 + H^2 \right] 
			+ C_{q,\eps} \| D \varphi \|_{\infty}^2  \cJ_{1,q}.
		\end{aligned}
		\]
Moreover, from \eqref{Y}, $\bar u \le 0$, \eqref{boundgradient} and $Y + \beta_t \nabla e^{\bar{u}} = 0$ a.e. on $\{ \rd u = 0 \}$ 
it follows that 
		\[
		\begin{aligned}
		&- q \int_M \varphi^2  (1+\beta_t)^{q-1} \la \nabla \beta_t, Y +  \beta_t \nabla e^{\bar u} \ra 
		\\
		\le \ & 
		- q \int_{\{ \rd u \neq 0 \}} \varphi^2  (1+\beta_t)^{q-1} 
		\la \nabla \beta_t, - \frac{\sh \beta}{\ch \beta} H \nu +  \beta_t \nabla e^{\bar u} \ra 
		\\
		\le \ & 
		q \int_{ \{ \rd u \neq 0\} } \varphi^2 (1+\beta_t)^{q-1}\|\nabla \beta\| |H|  
		+ q \int_{\{ \rd u \neq 0 \}} \varphi^2 (1+\beta_t)^q \ch \beta \|\nabla \beta\| 
		\\
		\le \ & 
		2 \eps \int_{ \{ \rd u \neq 0\} } \varphi^2 (1+\beta_t)^q \|\nabla \beta\|^2 
		+ \frac{q^2}{\eps} \int_M \varphi^2 (1+\beta_t)^{q-2}H^2 
		+ \frac{q^2}{\eps} \int_M \varphi^2 (1+\beta_t)^q \ch^2 \beta  
		\\
		\le \ & 
		2 \eps \int_{ \{ \rd u \neq 0 \}} \varphi^2 (1+\beta_t)^q \|\nabla \beta\|^2 
		+ \e^{-1}  C_{q} \| \varphi\|^2_\infty
		\left[ \bar{\mathcal{I}}_{2,q_0} + \cJ_{1,q}\right].
		\end{aligned}
		\]
Plugging these inequalities into \eqref{start_1}, we get
		\begin{equation}\label{start_3}
		\begin{aligned}
			&\int_{\{\di u \neq 0\}} \varphi^2 (1+\beta_t)^q 
			\left[ 
			\frac{\sh^2 \beta}{\ch^2 \beta} \|A \|^2 + \left\{ \frac{1}{m-1} - C_m  \e \right\} \| \nabla \beta \|^2 
			\right]
			\\
			& \quad + \int_M \varphi^2 (1+\beta_t)^q e^{\bar u} \beta_t \sh^2 \beta  
			\\
			\le \ & 
			- \int_M  (1+\beta_t)^q \la \nabla \varphi^2, \beta_t \nabla e^{\bar u} \ra
			+ C_{m,q,\eps} \| \varphi \|_{W^{1,\infty}}^2
			\left[ \bar{\mathcal{I}}_{2,q_0} + \mathcal{J}_{1,q}\right].
		\end{aligned}
		\end{equation}

	We next examine the term
	\[
	K \doteq - \int_M (1+\beta_t)^q \la \nabla \varphi^2, \beta_t \nabla e^{\bar u} \ra.
	\]
For $U \Subset \Omega_q$, we choose $\varphi$ satisfying \eqref{phi1} and 
	\begin{equation}\label{phi2}
		\varphi = 0 \quad  \text{on $\partial U$}.
	\end{equation}
Hereafter, we will denote by $\mathcal{C}_j$ a constant depending on the same quantities as \eqref{eq_behaconst}.
Since $\nabla \beta_t = 0$ a.e. on $\{ \rd u = 0\}$, we compute
	\begin{equation}\label{equ_A}
		\begin{aligned}
			K &= - \int_M (1+\beta_t)^q \beta_t \la \nabla \varphi^2 , \nabla (e^{\bar u}-1) \ra 
			\\
			&= - \int_{M} \la \nabla \varphi^2, \nabla \left[(1+\beta_t)^q \beta_t (e^{\bar u}-1) \right]\ra 
			+ \int_{ \{ \rd u \neq 0\} } (e^{\bar u}-1) \la \nabla\varphi^2, \nabla \left[ (1+\beta_t)^q \beta_t\right] \ra.
		\end{aligned}
\end{equation}
The last integral can be easily estimated by using \eqref{assu_int_Linfty}, \eqref{boundgradient} and the definition of $\bar u$:
\begin{equation}\label{eq_intefacile}
\begin{aligned}
	&\left| \int_{\{ \rd u \neq 0\}} 
	(e^{\bar u}-1) \la \nabla\varphi^2, \nabla \left[ (1+\beta_t)^q \beta_t\right] \ra \right| 
	\\
	\le \ & \eps \int_{ \{ \rd u \neq 0\} } \varphi^2 (1+\beta_t)^q\|\nabla \beta\|^2 
	+ 4 \e^{-1} (1+q)^2\|e^{\bar u}-1\|^2_{L^\infty(\Omega_q)} \int_M (1+\beta_t)^q \|\nabla \varphi\|^2 
	\\
	\le \ & \eps \int_{ \{ \rd u \neq 0\} } \varphi^2 (1+\beta_t)^q\|\nabla \beta\|^2 
	+ \e^{-1} \mathcal{C}_1 \| D \varphi \|_\infty^2 \cJ_{1,q}.
	\end{aligned}
\end{equation}
On the other hand, since $\varphi^2 \in W^{2,\infty}(\Omega_q)$ with $\mathrm{supp}\, \varphi \Subset \Omega_q$, we get
	\begin{equation}\label{equ_B'}
		\begin{aligned}
			- \int_{M} \la \nabla \varphi^2, \nabla \left[ (1+\beta_t)^q\beta_t (e^{\bar u}-1) \right] \ra 
			&= \int_{M} (1+\beta_t)^q \beta_t (e^{\bar u}-1) \Delta_M \varphi^2
			\\
			&= \int_{M} (1+\beta_t)^q \beta_t \left( 1 - e^{\bar{u}} \right) \left( - \Delta_M \varphi^2 \right).
		\end{aligned}
	\end{equation}
We set $U = L_{\bar R}(o)$ where $o \in \Omega_{q+1}$. 
Then $U \Subset \Omega_q$ and 
since $u$ is smooth with $\| Du \|_\infty < 1$, 
$\partial L_{\bar R}(o)$ is smooth. We also set 
	\[
		\varphi (x) \doteq \left( \bar R^2 - \ell_o^2(x) \right)_+.
	\]
It is easily seen that \eqref{phi1} and \eqref{phi2} are satisfied. Moreover, 
by \eqref{eq_elle2} and 
	\begin{equation}\label{eq_deltal4}
		\begin{aligned}
			-\Delta_M \ell^4_o 
			= - 2 \| \nabla \ell_o^2 \|^2 - 2\ell^2_o \Delta_M \ell_o^2 \le - 2\ell^2_o \Delta_M \ell_o^2,
		\end{aligned}
	\end{equation}
it follows that on $U$, 
	\begin{equation}\label{Lap-phi^2}
		\begin{aligned}
			-\Delta_M \varphi^2 & = -\Delta_M \left( \bar R^4 - 2 \bar R^2 \ell_o^2 + \ell_o^4 \right) 
			\leq 2 \left( \bar R^2 - \ell_o^2 \right) \Delta_M \ell^2_o \\
			&\leq 4 \bar R^2 \left( m + 2 \left| H \right| \ch \beta \left| x- o \right| \right)
			\\
			& \leq \mathcal{C}_2 \left( 1 + |H| \ch \beta \right).
		\end{aligned}
	\end{equation}
Remark also that 
	\[
		\| \varphi \|_{W^{1,\infty}} \leq \mathcal{C}_3. 
	\]
From \eqref{Aq}, \eqref{equ_B'}, \eqref{Lap-phi^2},  $0 \leq 1 - e^{\bar{u}} \leq 1$, $\beta \le \ch^2 \beta$, \eqref{eq_intefacile} and \eqref{eq_gHchg}, we deduce
	\begin{equation}\label{equ_quasifinal}
	\begin{aligned}
	K &\le 
	\mathcal{C}_2 \int_{M_q} \left(1+\beta_t\right)^q \beta_t \left( 1+ |H| \ch \beta \right)
	\\
	& \qquad + \mathcal{C}_1 \e^{-1} \| D\varphi \|_\infty^2 \cJ_{1,q}  
	+ \eps \int_{\{ \rd u \neq 0 \}} \varphi^2 (1+\beta_t)^q \|\nabla \beta\|^2
	\\
	&\le \mathcal{C}_3 \e^{-1} \left[\bar{\mathcal{I}}_{2,q_0} +  \cJ_{1,q}\right]
	+ \eps \int_{\{ \rd u \neq 0\}} \varphi^2 (1+\beta_t)^q \|\nabla \beta\|^2.
	\end{aligned}
	\end{equation}

	Since $\varphi \geq \bar R^2/2$ on $L_{\bar R/2}(o)$, it follows from \eqref{start_3} and \eqref{equ_quasifinal} that 
	\[
			\begin{aligned}
			\int_{L_{\bar R/2}(o)} 
			(1+\beta_t)^q 
			&\left[ \frac{\sh^2 \beta}{\ch^{2} \beta} \|A\|^2
			+ \left\{ \frac{1}{m-1}- C_m \eps \right\} \| \nabla \beta \|^2 \right] \cdot
			\mathbb{1}_{\{\rd u \neq 0\}}
			\\
			& \qquad  + \int_{L_{\bar R/2}(o)} e^{\bar u} (1+\beta_t)^q\beta_t \sh^2 \beta 
			\le \mathcal{C}_4 C_{m,q,\e} \left[ \cJ_{1,q} + \bar{\cI}_{2,q_0}\right].
			\end{aligned}
	\]
Choosing $\eps = \big[2C_m(m-1)\big]^{-1}$, noting that $e^{\bar{u}} \geq e^{-2\mathcal{I}_0}$ and letting $t \to 0$, we deduce
		\begin{equation}\label{eq_finish}
		\int_{L_{\bar R/2}(o)} (1+\beta)^q \left[\frac{\sh^2 \beta}{\ch^{2} \beta} \|A\|^2 + \|\nabla \beta\|^2 + \beta \sh^2 \beta\right] \cdot \mathbb{1}_{\{\di u \neq 0\}}
		\le \mathcal{C}_5.
		\end{equation}
Consider a maximal set of disjoint Euclidean balls $\{B_{\bar R/4}(o_1), \ldots, B_{\bar R/4}(o_s)\}$ with $o_i \in \Omega_{q+1}$. Since $B_{\bar R/4}(o_i) \subset L_{\bar R/4} (o_i) \Subset \Omega_q \Subset \Omega'$, we get
	\[
	s \le \left\lceil \frac{|\Omega'|_\delta}{\omega_m (\bar R/4)^m}\right\rceil 
	\doteq \tau(m,R, q_0, |\Omega'|_\delta).
	\]
Using that $\{B_{\bar R/2}(o_j)\}$ covers $\Omega_{q+1}$ and $B_{\bar R/2}(o_j) \subset L_{\bar R/2}(o_j) \Subset \Omega_q$, summing up \eqref{eq_finish} we conclude
		\[
		\int_{M_{q+1}} (1+\beta)^q\left[\frac{\sh^2 \beta}{\ch^{2} \beta} \|A\|^2 + \|\nabla \beta\|^2 + \beta \sh^2 \beta\right] \cdot \mathbb{1}_{\{\di u \neq 0\}}
		\le \mathcal{C}_5\tau,
		\]
which proves $(\mathscr{B}_{q})$. 
\end{proof}

\begin{remark}\label{rem_point}
We comment on the choice of $\varphi$ in the above proof. For a general cut-off function $\varphi$, in view of \eqref{eq_hessianLapla}, one could just obtain the bound
	\[
		\left| \Delta_M \varphi^2 \right| 
		\le 
		m \| D^2 \varphi^2 \|_\infty (1+ \ch^2 \beta) + \| D \varphi^2 \|_\infty |H| \ch \beta,
	\]
which inserted into \eqref{equ_B'} would make necessary to estimate a term of the type
	\begin{equation}\label{eq_thebadone}
	\int_U  (1+\beta_t)^q \beta_t \ch^2 \beta.
	\end{equation}
Such a term cannot be absorbed into the last addendum on the left-hand side of \eqref{start_3}. This is the main reason why we use the extrinsic Lorentzian distance. Furthermore, the translation performed in the first line of \eqref{equ_A} and the choice of $\bar{u}$ in \eqref{def-baru}  
are crucial to make sure that 
the coefficient which multiplies $-\Delta_M \varphi^2$ in \eqref{equ_B'} is non-negative. 
Hence, an upper estimate for $-\Delta_M \varphi^2$ is sufficient and we can get rid of the term 
$\|\nabla \ell_o\|$ in \eqref{eq_deltal4}, that would have lead, again, to the appearance of an integral of the type \eqref{eq_thebadone}.
\end{remark}


\section{Proofs of the main theorems}
\label{prf-thms}

%
%

\subsection{Proof of Theorem \ref{teo_bocofo}}

Item $(i)$ directly follows from Proposition \ref{prop_nocross}. 
To prove $(ii)$ we consider a maximally extended light segment $\overline{xy}$, 
which we know to be bounded because $u_\rho \to 0$ at infinity by Proposition \ref{prop-cX-emb}. 
Due to maximality and to the anti-peeling theorem \cite[Theorem 3.2]{bartniksimon} applied to $\Omega \Subset \Rm \setminus \{x_1,\dots, x_k\}$, $\ov{xy}$ must start and end at points in $\{ x_1,\dots, x_k \}$, say $\ov{xy} = \ov{x_ix_j}$. It remains to prove $a_ia_j < 0$. 
To this end, we argue by contradiction and suppose $a_i a_j > 0$. 
By changing the indices $i$ and $j$ and considering $-u_\rho$ if necessary, 
we may suppose $a_j > 0$ and $u_\rho (x_i) - u_\rho (x_j) = |x_i - x_j|$. 
According to Theorem \ref{teo_nolight} and Remark \ref{rem_impo_obse}, for any $\alpha > 0$ the function $u_\rho$ is also a minimizer for $I_{\rho_\alpha}$ with $\rho_\alpha \doteq \rho + \alpha ( \delta_{x_i} - \delta_{x_j} ) $. 
Now we choose $\alpha = a_j$ and deduce that 
$u_\rho$ minimizes $I_{\rho_{a_j} }$. Since $\rho_{a_j}$ is of type \eqref{pointcharges_intro} and has no charge at $x_j$, the anti-peeling theorem implies that the light segment $\ov{x_ix_j}$ can be extended beyond $x_j$, contradicting the maximality of $\ov{x_ix_j}$.\par
We prove $(iii)$. Consider the approximation $\{\rho_j,H_j,u_j,w_j\}$ in Subsection \ref{subsec_strategy} and fix $\Omega' \Subset \R^m \backslash \{x_1,\ldots, x_k\}$ with smooth boundary. Then 
	\begin{equation}\label{bdd-Hj}
		\sup_{j \geq 1} \| H_j \|_{L^\infty(\Omega')} < \infty.
	\end{equation}
By Proposition \ref{1001}, $u_j \to u_\rho$ in $L^\infty(\Rm)$ and $\mathscr{G} \doteq \{u_\rho\} \cup \{ u_j : j \in \N \}$ is 
compact in $C(\Rm)$. Thus, for given $\Omega'' \Subset \Omega'$, 
by Lemma \ref{lem_simplecompact} and the assumption that $u_\rho$ has no light-segments, there exists $R>0$ independent of $j$ such that 
the Lorentzian ball $L^{\rho_j}_R(\Omega'') \Subset \Omega'$ for all $j \geq1$. 
By \eqref{bdd-Hj}, we can apply Theorem \ref{teo_higherint} to deduce 
	\[
		\sup_{j \geq 1} \left\| w_j \log \left( 1 + w_j \right) \right\|_{L^1(\Omega'')} < \infty.
	\]
Thus, the sequence $\{w_j\}$ is locally uniformly integrable on $\Omega'$. By the arbitrariness of $\Omega'$, 
$\{w_j\}$ is locally uniformly integrable on $\Omega \backslash \{x_1,\ldots,x_k\}$; 
hence, Theorem \ref{teo_removable} with $E= \{ x_i \}_{i=1}^k$ implies	
	\begin{equation}\label{eq_BI_bcf}
	\int_{\R^m} w_\rho Du_\rho \cdot D\eta \, \di x = \la \rho, \eta \ra = \sum_{i=1}^k a_i \eta(x_i)  \qquad \forall \, \eta \in \lip_c(\R^m).
	\end{equation}
Therefore, $u_\rho$ weakly solves \eqref{borninfeld}. 
Since $u_\rho$ does not possess light segments, for each ball $\widetilde{B} \Subset \R^m \backslash \{x_1,\ldots,x_k\}$ we have 
	\[
		|u_\rho(x) - u_\rho(y) | < |x-y| = d_{\widetilde{B}}(x,y) \quad 
		\text{$\forall \, x,y \in \partial \widetilde{B}$ with $x \neq y$}.
	\]
Moreover, $u_\rho$ minimizes $I_0$ on $\widetilde{B}$ with respect to its own boundary values, thus $u_\rho$ coincides with the solution given in Bartnik-Simon's \cite{bartniksimon}, which is smooth and strictly spacelike. Hence, $u_\rho$ is smooth and strictly spacelike away from $x_1,\ldots, x_k$. We next prove that $u_\rho$ has an isolated singularity at each $x_i$, in the sense of Ecker \cite{Ecker}. Fix $B \doteq B_r(x_i)$ with $x_j \not \in \overline{B}$ for $j \neq i$, and choose $\eta \in \Lip_c(B)$ with 
$\eta = - a_i$ in a neighborhood of $x_i$. Suppose by contradiction that $u_\rho$ minimizes $I_0$ in $B$, that is, 
	\begin{equation}\label{min-I0}
		I_0(u_\rho) = \inf \left\{ I_0(v)  \ :\ v \in \cX_{u_\rho} ( B ) \right\}, \quad 
		I_0(v) \doteq \int_B \Big(1 - \sqrt{1 - |Dv|^2}\Big) \, \rd x. 
	\end{equation}
Since $D\eta = 0$ around $x_i$ and $u_\rho$ is strictly spacelike away from $x_1,\ldots, x_k$, for $t>0$ small enough it holds $u_\rho + t\eta \in \cX_{u_\rho}(B)$. 
Using Proposition \ref{lem_basicL2} and comparing to \eqref{eq_BI_bcf}, we get 
	\[
		0 \ge \int_B w_\rho Du_\rho \cdot \left( Du_\rho - D(u_\rho+t\eta) \right)\, \di x = - t \int_B w_\rho Du_\rho \cdot D \eta \, \di x = t|a_i|^2 > 0,  
	\]
which is a contradiction.

	To conclude, \cite[Theorem 1.5]{Ecker} ensures that $u_\rho$ is asymptotic to a light cone $C$ near $x_i$, 
and we can therefore apply the argument in \cite[Theorem 3.5]{bcf} to deduce 
that $C$ is upward or downward pointing respectively when $a_i<0$ or $a_i>0$. 
\qed

\subsection{Proof of Theorem \ref{teo_BI_local_s2}}
\label{subsec_prf_2dim}

Let $\Sigma \Subset \Omega$ and $\rho \in \cM(\Omega)$ satisfy the assumptions in Theorem \ref{teo_BI_local_s2}. 
Fix $\FF,\mathcal{I}_1, \mathcal{I}_2,\Omega'$ and $\e$ as in (ii):
	\begin{equation}\label{eq_assu_mainthm2}
	\phi \in \FF, 
	\qquad \|\rho\|_{\cM(\Omega)} \le \mathcal{I}_1, \quad \|\rho\|_{L^2(\Omega')} \le \mathcal{I}_2.
	\end{equation}
We also choose $p_1 =3$ for $\cX (\Omega)$ (any $p_1>2$ works). 
We split the proof into several steps. \\[0.2cm]
\noindent \textbf{Step 1:} \textsl{for each $\phi,\rho$ satisfying \eqref{eq_assu_mainthm2}, and for each $\e>0$, there exists 
	\[
	\mathcal{C}_1\big(\Omega,\FF,\diam_\delta(\Omega), \mathcal{I}_1, \mathcal{I}_2,\e, \di_\delta(\Omega', \partial \Omega)\big) 
	\]
such that
	\[
		\disp \int_{\Omega'_{\e}} w_\rho \log \left(1+w_\rho\right) \di x \le \mathcal{C}_1, \quad 
		\Omega'_\e \doteq \left\{ x \in \Omega' \ : \ \di_{\delta} (x,\partial \Omega') > \e \right\}.
	\]
}

	\begin{proof}[Proof of Step 1] 
This directly follows from Theorem \ref{teo_higher_m2} and \eqref{eq_assu_mainthm2}. 
	\end{proof}
The higher integrability allows to prove the next no-light-segment property. \\[0.2cm]
\noindent \textbf{Step 2:} 
\textsl{The minimizer $u_\rho$ does not have light segments in $\Omega'$. 
}
%
	\begin{proof}[Proof of Step 2]
%
Assume by contradiction that $\ov{xy} \subset \Omega'$ is a light segment for $u_\rho$. Up to renaming, $u_{\rho}(y) - u_{\rho}(x) = |y-x|$. 
Define
	\[
	\wt{\rho} \doteq \rho + \delta_y - \delta_x.
	\]
By Theorem \ref{teo_nolight}, $u_{\rho}$ also minimizes $I_{\wt{\rho}}$: $u_{\rho} = u_{\wt{\rho}}$. To reach our desired contradiction, we tweak the argument in Theorem \ref{teo_nolight} 
used to show that $u_{\rho}$ does not solve \eqref{borninfeld}. Let $\{\varphi_j\}$ be a mollifier and define $\rho_j = \varphi_j * \rho$ and $\wt{\rho}_{j} = \varphi_j * \wt{\rho}$. 
Call $u_j, \wt{u}_j \in \cX_\phi(\Omega)$, respectively, the minimizers of $I_{\rho_j}$ and $I_{\wt{\rho}_{j}}$, 
and denote by $w_j$ and $\wt{w}_j$, respectively, their energy densities. In view of Proposition \ref{1001} and $u_{\rho} = u_{\wt{\rho}}$, as $j \to \infty$, 
we have $u_j \to u_{\rho}$ and $\wt{u}_j \to u_{\rho}$ in $C(\overline\Omega)$. 
Notice that, by the properties of convolutions (see \cite[Proof of Proposition 2.7]{Ponce-mono}), 
	\[
	\|\rho_j\|_{\cM(\Omega)} \le \|\rho\|_{\cM(\Omega)} \le \mathcal{I}_1, 
	\qquad \| \wt{\rho}_{j}\|_{\cM(\Omega)} \le \|\wt{\rho}\|_{\cM(\Omega)} \le \mathcal{I}_1 + 2
	\]	
and for each $\Omega'' \Subset \Omega'\backslash \{x,y\}$, $j$ large enough and $\e$ small enough, 
	\[
	\| \rho_j\|_{L^2(\Omega_{\e/4}'')} + \|\wt{\rho}_{j}\|_{L^2(\Omega_{\e/4}'')} \le \|\rho\|_{L^2(\Omega'')} + \|\wt{\rho}\|_{L^2(\Omega'')} \le 2 \mathcal{I}_2 + 2.
	\]
Hence, we can apply Theorem \ref{teo_higher_m2} on $\Omega'' \Subset \Omega' \setminus \set{x,y}$ 
to both $u_j$ and to $\wt{u}_j$ to deduce that 
$\{w_j\}$ and $\{\wt{w}_{j}\}$ are locally uniformly integrable on $\Omega' \backslash \{x,y\}$. 
Then, Theorem \ref{teo_removable} with $E=\{x,y\}$ guarantees that 
	\[
		\int w_{\rho} Du_{\rho} \cdot D\eta \, \rd x 
		= \la \rho, \eta \ra, \qquad \int w_{\rho} Du_{\rho} \cdot D\eta \, \rd x = \la \wt{\rho}, \eta \ra 
		\quad \text{$\forall \, \eta \in \Lip_c (\Omega')$}.
	\] 
However, choosing $\eta$ such that $\eta(y)\neq \eta(x)$, we deduce 
	\[
		\la \wt{\rho}, \eta \ra = \la \rho, \eta \ra + \eta(y) - \eta(x) \neq \la \rho, \eta \ra, 
	\]
giving the desired contradiction. \qed
%
%

\bigskip

Hereafter, we denote with $\{\rho_j,u_j,w_j\}$ the approximation described in Subsection \ref{subsec_strategy}. With the aid of Step 2 and $\rho \in L^2(\Omega')$, 
an application of Lemma \ref{lem_simplecompact}, Corollary \ref{cor_secondfund} and Theorem  \ref{teo_higherint} gives the next improved higher integrability and second fundamental form estimates for $u_\rho$, which conclude the proof of Theorem \ref{teo_BI_local_s2} (ii). \\[0.2cm]

\noindent \textbf{Step 3:} 
\textsl{Higher integrability, Theorem \ref{teo_BI_local_s2} (ii): 
for each $\e>0$, $q_0 >0$, there exists a constant 
	\[
		\mathcal{C} = \mathcal{C}(\Omega,\FF,\diam_\delta(\Omega), \mathcal{I}_1, \mathcal{I}_2, \e, \Omega',q_0) > 0
	\] 
such that for each $\rho$ and $\rho$ satisfying \eqref{eq_assu_mainthm2},
	\[
	\begin{aligned}
		&\int_{\Omega'_{\e}} (1+ \log w_\rho)^{q_0} \left\{ w_\rho |D^2u_\rho|^2 + w_\rho^3 \left| D^2 u_\rho 
		\left( D u_\rho, \cdot \right) \right|^2 
		+ w_\rho^5 \left[D^2u_\rho(Du_\rho,Du_\rho)\right]^2\right\}\di x 
		\\
		& \qquad + \int_{\Omega'_\e} w_\rho (1+ \log w_\rho)^{q_0+1}\di x \le \mathcal{C}.
	\end{aligned}
	\]
}
\begin{proof}[Proof of Step 3]
Let $\mathscr{G} \subset \cX(\Omega)$ be the set of minimizers $u_\rho$ whose boundary value $\phi$ and 
source $\rho$ satisfy \eqref{eq_assu_mainthm2}. 
Because of the compactness of $\scrF$ and of Propositions \ref{teo_compaembe} and \ref{1001}, 
taking into account the lower semicontinuity of $\|\cdot\|_{L^2(\Omega')}$ and $\| \cdot \|_{\cM(\Omega)}$ under weak convergence, 
we deduce that $\mathscr{G}$ is compact in $C(\overline\Omega)$. 
Applying the second part of Lemma \ref{lem_simplecompact}, for $\e>0$ we infer the existence of 
	\[
	R = R \big(\Omega,\FF,\diam_\delta(\Omega), \mathcal{I}_1, \mathcal{I}_2,\e,\Omega'\big).
	\]
such that $L^{\rho_j}_R(\Omega'_\e) \Subset L_R^{\rho_j}(\Omega')$ for each $u \in \mathscr{G}$. Theorem \ref{teo_higherint} with $\Omega'' = \Omega'_\e$ ensures that \eqref{ine_thebest} holds for $u_j$ uniformly in $j$. The corresponding inequality for the pointwise limit $u_\rho$, which is a rewriting of our desired estimate, then follows by the same method as that in Corollary \ref{cor_secondfund}. 
	\end{proof}

\noindent \textbf{Step 4:} 
\textsl{Weak solvability and no light segments, Theorem \ref{teo_BI_local_s2} (i).}

	\begin{proof}[Proof of Step 4]
Applying Step 1 to the mollified sources $\rho_j$, we deduce that $\{w_j\}$ are locally uniformly integrable in $\Omega \setminus \Sigma$. Using $\haus^1_\delta(\Sigma) = 0$, Theorem \ref{teo_removable} implies that the limit $u_\rho$ is a weak solution to \eqref{borninfeld} on $\Omega$. On the other hand, by Step 2, $u_\rho$ does not have light segments in any set $\Omega'' \Subset \Omega \backslash \Sigma$, hence in $\Omega \backslash \Sigma$. 
Since $\haus^1_\delta(\Sigma) = 0$, there are no light segments on the entire $\Omega$.
\end{proof}

\noindent \textbf{Step 5:} 
\textsl{Regularity for $\rho \in L^\infty$, Theorem \ref{teo_BI_local_s2} (iii).}

\begin{proof}[Proof of Step 5]
Let $\rho \in L^\infty(\Omega')$, and fix a domain $\Omega'' \Subset \Omega'$. Due to Step 2, every point $x \in \Omega''$ has positive Lorentzian distance from $\partial \Omega'$, with a uniform bound depending on the data of our problem. We can therefore use the local gradient estimate in \cite[Lemma 2.1]{bartniksimon} as in \cite[Proof of Theorem 4.1]{bartniksimon} to deduce an $L^\infty$-estimate for $w_\rho$ and a $W^{2,2}$-estimate for $u_\rho$ in $\Omega''$. From Theorem \ref{teo_BI_local_s2} (i) and (ii), $u_\rho \in W^{2,2}_\loc(\Omega')$ is a strong solution to 
	\[
		- \sum_{i=1}^m \partial_i \left( a_i(Du_\rho) \right) = \rho \quad \text{in $\Omega''$, where }  
		\ a_i(p) \doteq \left( 1 - |p|^2 \right)^{-1/2} p_i : B_1(0) \to \R.
	\]
By differentiating formally the equation in $x_k$, we see that $(u_\rho)_k \in W^{1,2}(\Omega'')$ is a weak solution to 
	\[
		 - \sum_{i=1}^m \partial_i \sum_{n=1}^m \frac{\partial a_i}{\partial p_n} (Du_\rho) (u_\rho)_{nk} 
		 = \sum_{i=1}^m \partial_i \left( \rho \delta_{ki} \right) \quad \text{in $\Omega''$}.
	\]
Since $(\partial a_i/ \partial p_n)$ is bounded and uniformly elliptic on $\Omega''$ 
due to the $L^\infty$-bound of $w_\rho$, 
applying \cite[Theorem 8.22 or Corollary 8.24]{GiTr01}, we see that $(u_\rho)_k \in C^\alpha_\loc (\Omega'')$ for some $\alpha$, 
hence, $u_\rho \in C^{1,\alpha}_\loc(\Omega'')$. By bootstrapping, $u_\rho \in C^\infty(\Omega')$ whenever $\rho \in C^\infty(\Omega')$. 
\end{proof}

By Steps 1--5, we complete the proof of Theorem \ref{teo_BI_local_s2}. 
\end{proof}

\begin{remark}
Referring to the approximations $\{u_j\}$ of $u_\rho$ in Subsection \ref{subsec_strategy}, 
because of Theorem \ref{teo_higherint}, Lemma \ref{lem_simplecompact} and the argument in Step 2 above, 
we deduce that the uniform integrability of $\{w_j \log w_j\}$ on a subdomain $\Omega'$ where $\rho \in L^2$ is \emph{equivalent} to the nonexistence of light segments for $u_\rho$ on $\Omega'$. 
\end{remark}

\subsection{Proof of Theorem \ref{teo_BI_local_sm}}

The proof is similar to the one of Theorem \ref{teo_BI_local_s2}. 
We consider the approximation $\{\rho_j, H_j, u_j, w_j\}$ in Subsection \ref{subsec_strategy}. 
Fix $\Omega' \Subset\Omega \setminus (\Sigma \cup K^\rho_\phi)$ and a small $\e>0$. Then, 
	\[
	\|\rho_j\|_{L^2(\Omega'_\e)} \le \|\rho\|_{L^2(\Omega')} \qquad \text{for $j$ large enough.}
	\]
Let $\Omega'' \Subset \Omega'_\e$. From the definition of $K_\phi^\rho$ and Proposition \ref{1001}, the first part of Lemma \ref{lem_simplecompact} applied to $\mathscr{G} \doteq \{u_j\}_j \cup \{u\}$ guarantees the existence of $R$ such that $L_R^{\rho_j}(\Omega'') \Subset \Omega'$ for each $j$, and therefore, by Theorem \ref{teo_higherint} we deduce that, for each $q_0 \in \R^+$, 
	\[
			\sup_j \int_{\Omega''} \left\{ w_j \left( 1 + \log w_j \right) + \|\SF_j\|^2 w_j^{-1}\right\} \left( 1 + \log w_j \right)^{q_0}  \di x < \infty.
	\]
Hence, Theorem \ref{teo_BI_local_sm} (ii) holds by the same argument as the one in Corollary \ref{cor_secondfund}. In the case $\rho \in L^\infty(\Omega')$, from $L_R^{\rho_j}(\Omega'') \Subset \Omega'$ and $\|\rho_j\|_{L^\infty(\Omega'')} \le \|\rho\|_{L^\infty(\Omega')}$ for large enough $j$ we can proceed as in the proof of Step 5 in Theorem \ref{teo_BI_local_s2} to get $w_\rho \in L^\infty(\Omega'')$ and then $u_\rho \in C^{1,\alpha}_\loc(\Omega')$, which proves Theorem \ref{teo_BI_local_sm} (iii).

	Summarizing, in our assumptions $\{w_j\}$ is locally uniformly integrable on $\Omega \setminus (\Sigma \cup K^\rho_\phi)$. 
Theorem \ref{teo_removable} ensures that $u_\rho$ satisfies \eqref{borninfeld} on $\Omega \backslash K^\rho_\phi$. 
Moreover, if $K^\rho_\phi \cap (\partial \Omega \cup \Sigma) = \emptyset$, 
then we can  choose open sets $\Omega'', \Omega'$ 
such that $K^\rho_\phi \subset \Omega'' \Subset \Omega' \Subset \Omega \backslash \Sigma$. 
By the definition of $K^\rho_\phi$ and applying Lemma \ref{lem_simplecompact}, 
we get the existence of $R$ such that $L^{\rho_j}_R(\Omega'') \Subset \Omega'$ for each $j$, 
and therefore a uniform integrability of $\{w_j\}$ on $\Omega''$ by Theorem \ref{teo_higherint}. 
Hence, $\{w_j\}$ is locally uniformly integrable on the entire $\Omega \backslash \Sigma$, 
and $u_\rho$ solves \eqref{borninfeld} on $\Omega$ by Theorem \ref{teo_removable}. 
Thus, Theorem \ref{teo_BI_local_sm} (i) holds and this completes the proof. 
\qed

\subsection{Proof of Theorems \ref{teo_BI_global} and \ref{teo_BI_global_consing}}

We begin with the following proposition:

\begin{proposition}\label{prop_extrinsicglobal}
Let $m \ge 3$ and $\mathcal{I}>0$ be given. Then there exists  a constant $\mathcal{J} = \mathcal{J} (m,\mathcal{I},p_1) > 0$ such that 
for any $\rho \in \cX(\Rm)^\ast$ with $\| \rho \|_{\cX^\ast} \leq \mathcal{I}$, the minimizer $u_\rho$ satisfies  
	\begin{equation}\label{uniform-bdd-urho}
		\| u_\rho \|_{\infty} \leq \mathcal{J}.
	\end{equation}	
Moreover, $L_\e^\rho(\Omega'') \Subset \Omega'$ holds provided $\e > 0$ and 
$\Omega'' \subset \Omega' \subset \Rm$ satisfy
	\begin{equation}\label{assu_dist}
		\di_\delta (\Omega'', \R^m\backslash \Omega') \ge 2 \mathcal{J} + \eps.
	\end{equation}
\end{proposition}

\begin{proof}
Remark that the minimizer $u_\rho$ satisfies $I_\rho(u_\rho) \leq I_\rho(0) = 0$. 
Recalling \eqref{coercive} and noting that $b_1 = 1/2$ in \eqref{1}, we see that 
for each $\rho \in \cX(\Rm)^\ast$ with $\| \rho \|_{\cX^\ast} \leq \mathcal{I}$, 
	\[
		\| u_\rho \|_\cX^2 \leq 4 \left[ 1 + 2 \| \rho \|_{\cX^\ast} \| u_\rho \|_{\cX} \right] 
		\leq 4 + 8 \mathcal{I} \| u_\rho \|_{\cX} .
	\]
Hence, minimizers are uniformly bounded in $\cX(\Rm)$ when $\|\rho\|_{\cX^*} \le \mathcal{I}$ and by virtue of Proposition \ref{prop-cX-emb}, 
\eqref{uniform-bdd-urho} holds.

	Let $\Omega'' \subset \Omega'$ satisfy \eqref{assu_dist}. 
Notice that \eqref{uniform-bdd-urho} implies that 
for each $x,o \in \Rm$ and each $\rho \in \cX(\Rm)^\ast$ with $\| \rho \|_{\cX^\ast} \leq \mathcal{I}$, 
	\[
		\left( \ell^\rho_o \right)^2 (x) = r_o^2(x) - \left| u_\rho(x) - u_\rho(o) \right|^2 \geq r_o^2(x) - 4 \mathcal{J}^2.
	\]
Hence, for any $x \in \Rm \setminus \Omega'$ and $o \in \Omega''$, 
	\[
		\left( \ell^\rho_o(x) \right)^2 \geq 4 \mathcal{J} \e + \e^2,
	\]
which implies $L_\e^\rho(\Omega'') \Subset \Omega'$. 
\end{proof}

\begin{proof}[Proof of Theorem \ref{teo_BI_global}]
Define $p_1$ as in \eqref{def_p_1} for $m\ge 3$, and choose $\{\rho_j,u_j,w_j\}$ as in Subsection \ref{subsec_strategy}. Under the assumptions of Theorem \ref{teo_BI_global}, 
in view of Proposition \ref{prop_extrinsicglobal}, 
there exists $\mathcal{J}= \mathcal{J} (m,\mathcal{I},p) $ such that $\|u_j\|_\infty \le \mathcal{J}$ and 
$L^{\rho_j}_{\e} (\Omega'') \Subset \Omega'$ 
for any $\e>0$ with $\di_\delta (\Omega'', \Rm \setminus \Omega' ) \geq 2 \mathcal{J}+ \e$. 
Then the local uniform higher integrability of $\{w_j\}$ and the fact that $u_\rho$ solves \eqref{borninfeld} 
directly follow from Theorems \ref{teo_removable} and \ref{teo_higherint}.
\end{proof} 

\begin{proof}[Proof of Theorem \ref{teo_BI_global_consing}]
The proof follow verbatim that of Theorem \ref{teo_BI_local_sm}, 
with the help of the $L^\infty$ estimates in Proposition \ref{prop_extrinsicglobal}, and is left to the reader. 
\end{proof}

\section{Weak solutions with light segments}\label{sec_counterexamples}

In this section we construct the example in Theorem \ref{ex-light-seg_intro}. First, for $\ell \in \{1, \ldots, m-2\}$ we write $x \in \R^m$ as 
	\[
	x = (y,z,x_m), \qquad \text{with} \quad y \in \R^{m-\ell}, \ \ z \in \R^{\ell-1}.
	\]
If $\ell = 1$, then the variable $z$ can be omitted, which allows for some computational simplifications. The idea is to consider the function 
	\begin{equation}\label{def-U}
	U(x) = \big( 1 - \eps^{2\kappa}|y|^{2\kappa}\big) x_m  
	\end{equation}
for $\kappa \ge 1$. Notice that the set $\{|y| = 0\}$ is an $\ell$-dimensional subspace made up of light segments, but $U$ does not satisfy a spacelike boundary condition in any bounded smooth domain $\Omega$ containing the origin. For this reason, for $\e>0$ we fix cut-off functions $\vartheta_\e$, $\zeta_\e$ and  $A_\e$ as follows: 
\begin{itemize} 
\item $\vartheta_\e \in C^\infty_c(\R)$ is defined by $\vartheta_\e(t) \doteq \vartheta_1(\e t)$ and $\vartheta_1(t)$ satisfies 
	\begin{equation}\label{prop-eta}
		\begin{aligned}
			&\vartheta_1 (t) \in C^\infty_c(\R), \quad \vartheta_1'(t) \leq 0 \quad \text{for $t \geq 0$}, \quad 
			\mathrm{supp}\, \vartheta_1 \subset [-2,2], \\
			&
			\vartheta_1 (t) \equiv 1 \ \text{for $0 \leq t \leq 1$}, \quad 
			\vartheta_1 (t) = 1 - \frac{e^2}{2} \exp \left( - \frac{1}{t-1} \right) \ \text{for $1 < t \leq \frac{3}{2}$}.
		\end{aligned}
	\end{equation}
\item $\zeta_\e \in C^\infty_c(\R)$ satisfies 
	\begin{equation}\label{eq:34.4}
		\zeta_\e \equiv 1 \quad \text{on} \ \left[ - \frac{1}{2\e} , \frac{1}{2\e} \right], 
		\quad \zeta_\e \equiv 0 \quad \text{on} \ \R \setminus \left( -\frac{1}{\e}, \frac{1}{\e}\right), 
		\quad \| \zeta_\e' \|_{L^\infty(\R)} \leq 4 \e. 
	\end{equation}
\item Having chosen a function $a_\e \in C^\infty_c(\R)$ with 
	\begin{equation}\label{eq:34.5}
		\begin{aligned}
			&a_\e (-t) = a_\e (t), \quad a_\e (t) = \left\{\begin{aligned}
				&1 & &\text{if}\ t \in [0,\e],\\
				&0 & &\text{if} \ t \in [2\e, \infty),
			\end{aligned}\right.
		\\
			&
			a_\e'(t) < 0 \quad \text{if} \ t \in (\e,2\e), 
			\quad a_\e(t) = 1 - d_\e \exp \left( - \frac{1}{t-\e} \right) \quad \text{if $t \in  \left( \e , \frac{3\e}{2} \right]$},
		\end{aligned}
	\end{equation}
where $d_\e > 0$ is chosen so that $a_\e ( 3\e/2 ) = 1/2$, $A_\e$ is defined by  
	\begin{equation}\label{eq:Ae}
		A_\e (t) \doteq \int_0^t a_\e (s) \, \rd s \in C^\infty(\R).
	\end{equation}
\end{itemize}
For $\kappa \geq 1$, we then define $U_\e(y,z,x_m)$ by 
\[
		U_\e(y,z,x_m) \doteq \left( 1 - \e^{2\kappa} |y|^{2\kappa} \right) \zeta_\e(|y|) \vartheta_\e(|z|) \zeta_\e (x_m) A_\e(x_m), 
\]
If $\ell = 1$, $\vartheta_\e(|z|)$ is replaced by $1$. 
Notice that $U_\e \in C^2_c(\R^m)$ and $U_\e \in C^\infty_c(\R^m)$ if $\kappa \in \N$. 
Remark that 
	\[
		U_\e(0,z,x_m) = x_m \quad \text{if $|z| \leq \frac{1}{\e}$ and $|x_m| \leq \e$},
	\]
and a direct computation shows that $|DU_\eps|< 1$ on the complement of the above set, see below. Hence, the union of the light segments of $U_\e$ is the $\ell$-dimensional compact cylinder 
	\[
	(0,z,x_m) \in \{0\} \times \overline{B}^{\ell-1}_{1/\eps} \times [-\eps,\eps].
	\] 
%
We hereafter denote with $W_\e$, $\rho_{U_\e}$ and $\SF_{U_\e}$ the energy density, the mean curvature and the second fundamental form of the graph of $U_\eps$. 
Theorem \ref{ex-light-seg_intro} follows from the next one applied with $\kappa = 1$: 
	\begin{theorem}\label{ex-light-subsp}
		Assume $m \ge 3$, $1 \leq \ell \leq m-2$ and $\kappa \in [1, m-\ell )$. Then 
		\begin{equation}\label{Ue-mc-sf}
			W_\e \in L^q_\loc (\R^m) \quad \text{and} \quad 
			\rho_{U_\e}, \ \left\| \SF_{U_\e} \right\| \in L^q(\R^m) \quad \text{for all $q < \frac{m-\ell}{\kappa}$},
		\end{equation}
	and $U_\e$ satisfies 
		\[
			\int_{\R^m} \frac{ DU_\e \cdot D \eta }{\sqrt{1 - |DU_\e|^2 }}  \rd x 
			= \int_{\R^m} \rho_{U_\e} \eta \, \rd x 
			\quad \text{for each $\eta \in C^\infty_c (\R^m)$}. 
		\]
	\end{theorem}

%

	\begin{proof}
For \eqref{Ue-mc-sf}, since $|\rho_{U_\e}| \leq C \|\SF_{U_\e}\|$ it is enough to estimate $W_\e$ and $\| \SF_{U_\e} \|$. 

For computational reasons, with a slight abuse of notation we write $U_\eps$ as a function of the triple $(r,s,x_m)$, with $r = |y|$, and $s = |z|$: 
	\[
	U_\e(r,s,x_m) = u_\e(r,x_m) \vartheta_\e(s),
	\]
where we set 
	\[
	u_\eps(r,x_m) = \left( 1 - \e^{2\kappa} r^{2\kappa} \right) \zeta_\e(r)\zeta_\e (x_m) A_\e(x_m).
	\]
It is readily checked that for a function $u(r,s,x_m)$ it holds  
	\begin{equation}\label{eq:Du}
		Du = u_r \frac{y}{|y|} + u_s \frac{z}{|z|} + u_m e_m 
	\end{equation}
and 
	\begin{equation}\label{eq:D^2u}
		D^2u = \begin{pmatrix}
			u_{rr} \frac{y}{|y|} \otimes \frac{y}{|y|} + \frac{u_r}{r} \left( I_{m-\ell} - \frac{y}{|y|} \otimes \frac{y}{|y|} \right) 
			& u_{rs} \frac{y}{|y|} \otimes \frac{z}{|z|} 
			& u_{rm} \frac{y}{|y|}
			\\
			u_{rs} \frac{z}{|z|} \otimes \frac{y}{|y|} & u_{ss} \frac{z}{|z|} \otimes \frac{z}{|z|} 
			+ \frac{u_s}{s} \left( I_{\ell-1} - \frac{z}{|z|} \otimes \frac{z}{|z|} \right) 
			& u_{sm} \frac{z}{|z|} 
			\\
			u_{rm} \frac{y^T}{|y|} & u_{sm} \frac{z^T}{|z|} & u_{mm}
		\end{pmatrix},
	\end{equation}
where $I_k$ is the identity matrix of size $k$. Since the matrix
	\[
		u_{rr} \frac{y}{|y|} \otimes \frac{y}{|y|} + \frac{u_r}{r} \left( I_{m-\ell} - \frac{y}{|y|} \otimes \frac{y}{|y|} \right)
	\]
has eigenvalues $u_{rr}$ and $u_r/r$ with  multiplicities $1$ and $m-\ell-1$ respectively, we see that 
	\begin{equation}\label{eq:normD^2u}
		\begin{aligned}
			\left| D^2 u \right|^2 
			&= u_{rr}^2 + (m-\ell-1) \frac{u_r^2}{r^2} + u_{ss}^2 + \left( \ell - 1 \right) \frac{u_s^2}{s^2} 
			+ u_{mm}^2 + 2 u_{rs}^2 + 2 u_{rm}^2 + 2 u_{sm}^2.
		\end{aligned}
	\end{equation}
Also, from \eqref{eq:Du} and \eqref{eq:D^2u} it follows that
	\begin{equation}\label{eq:D^2uDu}
		\begin{aligned}
			D^2u \left( Du , \cdot \right) 
			=& \left[ u_{rr} u_r  + u_{rs} u_s  + u_{rm} u_m \right] \frac{y}{|y|}
			+ \left[ u_{rs} u_r + u_{ss} u_s + u_{sm} u_m \right] \frac{z}{|z|} 
			\\
			&
			+ \left[ u_{rm} u_r + u_{sm} u_s + u_{mm} u_m \right] e_m
		\end{aligned}
	\end{equation}
and that
	\begin{equation}\label{eq:D^2DuDu}
		\begin{aligned}
			D^2u \left( Du, Du \right)
			&=
			u_{rr} u_r^2 + 2 u_{rs} u_r u_s + 2 u_{rm} u_r u_m + u_{ss} u_s^2 + 2 u_{sm} u_s u_m + u_{mm} u_m^2.
		\end{aligned}
	\end{equation}
For $u(|y|,x_m)$, \eqref{eq:Du}--\eqref{eq:D^2DuDu} also hold with $\ell =1$ 
and $u_s, u_{rs}, u_{ss}, u_{ms} = 0$.\par

Computing the gradient of $U_\e$, we obtain 
	\begin{equation}\label{DUe}
		\left| DU_\e \right|^2 = \left[ (u_\e)_r^2 + (u_\e)_m^2 \right] \vartheta_\e^2 + u_\e^2 (\vartheta_\e')^2 
		= \left| Du_\e \right|^2 \vartheta_\e^2 + u_\e^2 (\vartheta_\e')^2,
	\end{equation}
and moreover
	\begin{equation}\label{eq:34.7}
		\begin{aligned}
			(u_\e)_r &= \left[ \zeta_\e'(r) \left( 1 -\e^{2\kappa} r^{2\kappa} \right) 
			- \zeta_\e (r) 2\kappa \e^{2\kappa} r^{2\kappa -1} \right] \zeta_\e(x_m) A_\e(x_m), 
			\\
			(u_\e)_m &= \zeta_\e (r) (1-\e^{2\kappa} r^{2\kappa} ) 
			\left[ \zeta_\e'(x_m) A_\e(x_m) + \zeta_\e (x_m) a_\e(x_m) \right].
		\end{aligned}
	\end{equation}
Hereafter, $C$ and $C_\e$ will denote constants whose value may change from line to line, 
with $C_\e$ possibly depending on $\e$. From \eqref{eq:34.5}, we see that 
	\begin{equation}\label{eq:34.6}
		\left| A_\e (x_m) \right| \leq 2\e \quad \text{for all $x_m \in \R$}.
	\end{equation}
Hence, using also \eqref{prop-eta}, notice that 
	\begin{equation}\label{eq_uevarte}
		\left| u_\e(r,x_m) \right| \leq 2\e, \quad 0 \leq \vartheta_\e (s) \leq 1, \quad \left| \vartheta_\e'(s) \right| \leq C \e.
	\end{equation}

	We first consider the region 
	\[
		\Omega_0 \doteq \left\{|x_m| \geq \frac{3\e}{2} \right\} \subset \R^m.
	\]
Since $a_\e(x_m) \leq 1/2$ and $|\xi_\e|\leq 1$ due to \eqref{eq:34.4}, if $\e>0$ is small, then \eqref{eq:34.6}, \eqref{eq:34.4} and \eqref{eq_uevarte} give 	
	\begin{equation}\label{eq_lower}
	\begin{aligned}
	W_\e^{-2} & = 1 - |DU_\e|^2 \ge 1 - \left| Du_\e \right|^2 \vartheta_\e^2 - C \e^2 \\
	& \geq 1 - C\e^2  - \left( a_\e(x_m) \right)^2 \geq 1/2.
	\end{aligned}
	\end{equation}
Since $U_\e \in C^2_c(\R^m)$, we get $W_\e \leq \sqrt{2}$ and $\left\| \SF_{U_\e}\right\| \leq C$ on $\Omega_0$. Similarly, we study the region 
	\[
	\Omega_1 \doteq \left\{|x_m| \leq \frac{3\e}{2}, \ |y| \ge \frac{1}{2\e} \right\}.
	\]
For $\delta_\kappa \doteq 2^{-2\kappa} > 0$ and $|y| \geq 1/ (2\e)$, 
	\[
		0 \leq \zeta_\e (r) \left( 1 - \e^{2\kappa} |y|^{2\kappa } \right) \leq 1 - \delta_\kappa. 
	\]
Thus, by \eqref{eq:34.4}, \eqref{eq:34.6}, \eqref{eq:34.7} and $0 \leq a(x_m) \leq 1$, if $\e$ is small enough, then for some constant $\gamma_\kappa > 0$, 
	\[
	\begin{aligned}
	W_\e^{-2} & \ge 1 - \left| Du_\e \right|^2 \vartheta_\e^2 - C \e^2 \geq 1 - \left| Du_\e \right|^2 - C\e^2 \\
	& \geq 1- C \e^2 - \left( 1 - \delta_\kappa \right)^2 \left[ C \e^2 + 1 \right] \geq \gamma_\kappa^2 > 0.
	\end{aligned}	
	\]
Therefore, $W_\e$ and thus $\left\| \SF_{U_\e}\right\|$ are bounded on $\Omega_1$, too. Summarizing, 
	\begin{equation}\label{eq:34.8}
		W_\e \leq C, \quad \left\| \SF_{U_\e}\right\| \leq C  
		\quad \text{on } \, \left\{|x_m| \geq \frac{3\e}{2}\right\} \cup \left\{|y| \ge \frac{1}{2\e} \right\}. 
	\end{equation}

	Next, we shall check the integrability of $W_\eps$ and $\SF_{U_\e}$ on 
	\[
		\begin{aligned}
			\Omega_2 &\doteq \left\{|y| < \frac{1}{2\e}\right\} \cup \left\{|z| < \frac{1}{\e}, \ |x_m| \leq \e\right\},
			\\
			\Omega_3 &\doteq \left\{|y| < \frac{1}{2\e}\right\} \cup \left\{|z| < \frac{1}{\e}, \ \e \leq |x_m| \leq \frac{3\e}{2} \right\}.
		\end{aligned}
	\]
By \eqref{prop-eta}, we have $U_\e(r,s,x_m) = u_\e(r,x_m) = \big(1-\e^{2\kappa}r^{2\kappa}\big)A_\e(x_m)$ in a neighborhood of $\Omega_2 \cup \Omega_3$. In particular,
	\begin{equation}\label{eq:34.11}
		\begin{aligned}
			(u_\e)_r &= -2\kappa \e^{2\kappa} r^{2\kappa-1} A_\e(x_m), 
			&
			(u_\e)_m &= \left( 1 - \e^{2\kappa} r^{2\kappa} \right) a_\e(x_m),
			\\
			(u_\e)_{rr} &= - 2\kappa (2\kappa -1) \e^{2\kappa} r^{2\kappa -2} A_\e(x_m), 
			&
			(u_\e)_{rm} &= -2\kappa \e^{2\kappa} r^{2\kappa -1} a_\e(x_m), \quad 
			\\
			(u_\e)_{mm} &= \left( 1 - \e^{2\kappa} r^{2\kappa} \right) a_\e'(x_m).
		\end{aligned}
	\end{equation}
The bounds $1/2 \leq a_\e(x_m) \leq 1$ following from \eqref{eq:34.5} lead to 
	\[
		\begin{aligned}
			W_\e^{-2} & = 1 - |Du_\e(x)|^2 = 1 - 4 \kappa^2 \e^{4\kappa} r^{4\kappa -2} A_\e^2(x_m) 
			- \left( 1 - \e^{2\kappa} r^{2\kappa} \right)^2 a_\e^2(x_m)
			\\
			&= \left( 1 - a_\e(x_m) \right) \left( 1 + a_\e(x_m) \right) 
			\\ & \quad  + \e^{2\kappa} r^{2\kappa} \left[ \left( 2 - \e^{2\kappa} r^{2\kappa} \right) a_\e^2(x_m) 
			- 4 \kappa^2 \e^{2\kappa} r^{2\kappa -2} A_\e^2(x_m) \right] 
			\\
			&\geq 1 - a_\e(x_m) + \e^{2\kappa} r^{2\kappa} \left[ \frac{1}{4}  - 16 \kappa^2 \e^{4} \right].
		\end{aligned}
	\]
Thus, for sufficiently small $\e>0$, we get
	\begin{equation}\label{eq:34.12}
		W_\e \leq C_\e \left( 1 - a_\e(x_m) + r^{2\kappa} \right)^{ -\frac{1}{2} } \qquad \text{on } \, \Omega_2 \cup \Omega_3.
	\end{equation}
In particular, using $0 \le a_\e(x_m) \le 1$ we deduce	
	\[
	W_\e \leq C_\e r^{-\kappa} \in L^q(\Omega_2 \cup \Omega_3) \qquad \text{for each $q < \frac{m-\ell}{\kappa}$}.
	\]

	Regarding $\SF_{U_\eps}$, since $\kappa \geq1$ and $U_\eps$ has bounded support, 
it follows from \eqref{eq:normD^2u}, \eqref{eq:D^2uDu}, \eqref{eq:D^2DuDu} and \eqref{eq:34.11} that for $u = U_\eps (= u_\eps)$
	\[
		\begin{aligned}
			\left| D^2 U_\e \right|  \leq 
			C \left\{ \left|u_{rr}\right| + \left| \frac{u_r}{r} \right| + \left| u_{rm} \right| + \left| u_{mm} \right| \right\} 
			&
			\leq C_\e \left( r^{2\kappa -2} + r^{2\kappa -1} + \left| a_\e'(x_m) \right| \right) \\
			&\leq  C_\e \left( r^{2\kappa -2} + \left| a_\e'(x_m) \right| \right), 
			\\
			\left| D^2 U_\e (D U_\e, \cdot) \right|
			\leq \left| u_{rr} u_r + u_{rm} u_m \right| + \left| u_{rm} u_r + u_{mm} u_m \right|
			&\leq C_\e \left( r^{ 4\kappa -3 } +  r^{ 2\kappa -1 } + \left| a_\e'(x_m) \right| \right) \\
			& \leq C_\e \left(  r^{2\kappa -1} + \left| a_\e'(x_m) \right| \right), \\
			\left| D^2 U_\e (D U_\e,D U_\e) \right| 
			\leq \left| u_{rr} u_r^2 + 2 u_{rm} u_r u_m + u_{mm} u_m^2 \right| 
			&\leq C_\e \left( r^{6\kappa -4} + r^{4\kappa -2} + \left| a_\e'(x_m) \right| \right).
		\end{aligned}
	\]
By using \eqref{eq:34.12}, \eqref{norm_second} and $W_\e \geq1$, we deduce 
	\begin{equation}\label{bdd-u-sf}
		\begin{aligned}
			\left\| \SF_{U_\e} \right\| &\leq  W_\e \left| D^2 U_\e \right| + 2 W_\e^2 \left| D^2 U_\e \left( D U_\e,\cdot \right) \right| 
			+ W_\e^3 \left| D^2 U_\e \left( D U_\e, D U_\e \right) \right| \\
			 & \leq C_\e \left[  r^{\kappa -2} + r^{-1} + W_\e^3 \left|a_\e'(x_m) \right| \right] \leq C_\e \left( r^{-1} + W_\e^3\left|a_\e'(x_m) \right|\right).
		\end{aligned}
	\end{equation}
Whence, to prove that $\|\SF_{U_\e}\| \in L^q(\Omega_2 \cup \Omega_3)$ for $q < (m-\ell)/\kappa$, taking into account \eqref{eq:34.12} and that $a_\e' = 0$ on $[0,\e]$ it suffices to show 
	\begin{equation}\label{eq:34.15}
	 \left( 1 - a_\e(x_m) + |y|^{2\kappa} \right)^{ -\frac{3}{2} }\left| a_\e'(x_m) \right| 
	 \in L^q \left( \left\{ |y| < \frac{1}{2\e}, \ |z| < \frac{1}{\e} , \  \e \le |x_m| \le \frac{3\e}{2} \right\} \right) 
	\end{equation}
for each $q < \frac{m-\ell}{\kappa}$. 
Notice that it is enough to check it for $\frac{m- \ell}{3\kappa} < q < \frac{m- \ell}{\kappa}$ and for $\e \le x_m \le 3\e/2$, since $a_\e$ is even. Due to \eqref{eq:34.12} and since we can reduce to integrate in the variables $(y, x_m)$, using polar coordinates we get 
	\begin{equation}\label{we-ae}
		\begin{aligned}
			&\int_\e^{ \frac{3}{2} \e } \rd x_m \int_{|y| \le 1/(2\e)} \left( 1 - a_\e(x_m) + |y|^{2\kappa} \right)^{-\frac{3q}{2}} \left| a_\e'(x_m) \right|^q \rd y
			\\
			\leq \ & C_\e \int_{ \e }^{ \frac{3\e}{2} } \rd x_m 
			\int_{ 0 }^{ \frac{1}{2\e} } 
			\left| a_\e'(x_m) \right|^q \left( 1 - a_\e(x_m) + r^{2\kappa} \right)^{ - \frac{3q}{2} } r^{m-\ell - 1} \, \rd r
			\\
			\leq \ & 
			C_\e \int_{ \e }^{ \frac{3\e}{2} } \rd x_m 
			\int_{0}^{ (1 - a_\e(x_m))^{1/(2\kappa)} } 
			\left| a_\e'(x_m) \right|^q \left( 1 - a_\e(x_m) \right)^{ - \frac{3q}{2} } r^{m-\ell-1} \, \rd r 
			\\
			& \quad 
			+ C_\e \int_{ \e }^{ \frac{3\e}{2} } \rd x_m 
			\int_{ (1 - a_\e(x_m))^{1/(2\kappa)}}^{ \frac{1}{2\e} } 
			\left| a_\e'(x_m) \right|^q r^{ -3q \kappa + m- \ell -1 } \, \rd r
			\\
			\leq \ &
			C_\e \int_{ \e }^{ \frac{3\e}{2} } 
			\left| a_\e'(x_m) \right|^q \left( 1 - a_\e(x_m) \right)^{ - \frac{3q}{2} + \frac{m- \ell}{2\kappa } } \rd x_m.
		\end{aligned}
	\end{equation}
Recalling $a_\e(x_m) = 1 - d_\e \exp \left( - \frac{1}{x_m-\e} \right) $ in \eqref{eq:34.5}, we have 
	\[
		\left| a_\e'(x_m) \right|^q \left( 1 - a_\e(x_m) \right)^{ \frac{m- \ell -3q \kappa }{2\kappa} } 
		\leq C_\eps \left( x_m-\e \right)^{-2q} \exp \left( \frac{\kappa q - (m-\ell) }{2 \kappa (x_m - \e)} \right).
	\]
Hence, if $ \frac{m-\ell}{3\kappa} <  q < \frac{m-\ell}{\kappa} $, then 
	\[
		\int_\e^{\frac{3\e}{2}} \left| a_\e'(x_m) \right|^q 
		\left( 1 - a_\e(x_m) \right)^{ - \frac{3q}{2} + \frac{m-\ell}{2\kappa } } \rd x_m < \infty.
	\]
Thus, $\|\SF_{U_\e}\| \in L^q(\Omega_2 \cup \Omega_3)$ holds for each $q < (m-\ell)/\kappa$, as required.

	If $\ell = 1$, that is, if the variable $z$ is missing, we have therefore concluded the desired integrability properties of $W_\e$ and $\|\SF_{U_\e}\|$, since so far we only used that $0 \le \vartheta_\e \le 1$. The reader may therefore skip to the end of the proof, where we check that $U_\e$ is a weak solution. To conclude for $\ell \ge 2$, we shall  check the integrability of $\| \SF_{U_\e} \|$ on $\Omega_4 \cup \Omega_5$, where  
	\[
		\begin{aligned}
			\Omega_4 &\doteq \left\{|y| \leq \frac{1}{2\e}, \ \frac{1}{\e} < |z| \leq \frac{3}{2\e}, \ 
			|x_m| \leq \frac{3\e}{2} \right\}
			,\\
			\Omega_5 &\doteq \left\{|y| \leq \frac{1}{2\e}, \ \frac{3}{2\e} \leq |z| \leq \frac{2}{\e}, \ 
			|x_m| \leq \frac{3\e}{2} \right\}.
		\end{aligned}
	\]
This is achieved by similar estimates, though computationally more demanding.\par 
We first prove that $|DU_\e|<1$ on $\Omega_4 \cup \Omega_5$. 
Since $U_\e(r,s,x_m) = (1-\e^{2\kappa} r^{2\kappa} ) \vartheta_\e (s) A_\e(x_m)$ on $\Omega_4 \cup \Omega_5$, 
	\begin{equation}\label{DUe-D^2Ue}
		\begin{aligned}
			(U_\e)_r &= 
			- 2\kappa \e^{2\kappa} r^{2\kappa -1} \vartheta_\e A_\e, \quad 
			(U_\e)_s = \left( 1 - \e^{2\kappa} r^{2\kappa} \right) \vartheta_\e' A_\e, \quad 
			(U_\e)_m = \left( 1 - \e^{2\kappa} r^{2\kappa} \right) \vartheta_\e a_\e,
			\\
			(U_\e)_{rr} &= -2\kappa (2\kappa -1) \e^{2\kappa} r^{2\kappa -2} \vartheta_\e A_\e, \quad 
			(U_\e)_{rs} = -2\kappa \e^{2\kappa} r^{2\kappa -1} \vartheta_\e' A_\e, \\
			(U_\e)_{rm} &= -2\kappa \e^{2\kappa} r^{2\kappa -1} \vartheta_\e a_\e, \quad 
			(U_\e)_{ss} = \left(1-\e^{2\kappa} r^{2\kappa} \right) \vartheta_\e'' A_\e, \\
			(U_\e)_{sm} &= \left( 1 - \e^{2\kappa} r^{2\kappa} \right) \vartheta_\e' a_\e, \quad 
			(U_\e)_{mm} = \left( 1 - \e^{2\kappa} r^{2\kappa} \right) \vartheta_\e a_\e'.
		\end{aligned}
	\end{equation}
Thus, 
	\[
		\begin{aligned}
			W_\e^{-2} = \ & 1 - \left| DU_\e\right|^2 
			\\
			= \ & 1 - 4\kappa^2 \e^{4\kappa} r^{4\kappa -2} \vartheta_\e^2 A_\e^2 
			- \left( 1 - 2\e^{2\kappa} r^{2\kappa} + \e^{4\kappa} r^{4\kappa} \right) 
			\left[ (\vartheta_\e')^2 A_\e^2 + \vartheta_\e^2 a_\e^2 \right]
			\\
			= \ & 
			1 - (\vartheta_\e')^2 A_\e^2 - \vartheta_\e^2 a_\e^2 
			+ \e^{2\kappa} r^{2\kappa} 
			\left[ \left( 2 - \e^{2\kappa} r^{2\kappa} \right) 
			\left\{ (\vartheta_\e')^2 A_\e^2 + \vartheta_\e^2 a_\e^2 \right\} 
			- 4\kappa^2 \e^{2\kappa}r^{2\kappa -2} \vartheta_\e^2 A_\e^2
			\right].
		\end{aligned}
	\]
By 
	\[
		\left| A_\e(x_m) \right| \leq 2\e, \quad \frac{1}{2} \leq a_\e(x_m) \leq 1, \quad 
		\e r = \left| \e y \right| \leq \frac{1}{2} \quad 
		\text{for each $(y,z,x_m) \in \Omega_4 \cup \Omega_5$},
	\]
if $\e>0$ is sufficiently small, then 
	\[
		\left( 2 - \e^{2\kappa} r^{2\kappa} \right) \vartheta_\e^2 a_\e^2 
		- 4 \kappa^2 \e^{2\kappa} r^{2\kappa -2} \vartheta_\e^2 A_\e^2 
		\geq \frac{1}{8} \vartheta_\e^2.
	\]
Therefore, for every $(y,z,x_m) \in \Omega_4 \cup \Omega_5$, 
	\begin{equation}\label{1-DUe^2}
		W_\e^{-2} \geq 1 - (\vartheta_\e'(|z|))^2 A_\e^2(x_m) - \vartheta_\e^2(|z|) a_\e^2(x_m) 
		+ \frac{1}{8} \e^{2\kappa} |y|^{2\kappa} \vartheta_\e^2(|z|).
	\end{equation}
When $(y,z,x_m) \in \Omega_5$, by $3/2 \leq \e |z| \leq 2$ and \eqref{prop-eta}, we see that 
	\[
		\left( \vartheta_\e'(|z|) \right)^2 \leq C \e^2, \quad 
		\vartheta_\e^2 (|z|) \leq \vartheta_\e^2 \left( \frac{3}{2\e} \right) 
		= \frac{1}{4},
	\]
which implies that if $\e$ is sufficiently small, then for all $(y,z,x_m) \in \Omega_5$, 
	\[
		W_\e^{-2} = 1 - \left| DU_\e\right|^2 
		\geq 1 - C \e^4 - \frac{1}{4} \geq \frac{1}{2}.
	\]
Hence,  
	\begin{equation}\label{bdd-O5}
		W_\e, \ \left\| \SF_{U_\e} \right\| \in L^\infty ( \Omega_5 ).
	\end{equation}

	On the other hand, when $(y,z,x_m) \in \Omega_4$, 
we have $\vartheta_\e(|z|) \geq 1/2$, and \eqref{1-DUe^2} yields 
	\[
		W_\e^{-2} \geq 1 - 4\e^2 (\vartheta_\e'(|z|))^2 - \vartheta_\e^2(|z|) a_\e^2(x_m) 
		+ \frac{\e^{2\kappa} |y|^{2\kappa} }{32}.
	\]
Thus, to show $|DU_\e| < 1$, it suffices to prove 
	\begin{equation}\label{eta1-in}
		4 \e^2 \left( \vartheta_\e'(s) \right)^2 + \vartheta_\e^2 (s) 
		= 
		4\e^4 \left( \vartheta_1'( \e s ) \right)^2 + \vartheta_1^2 (\e s) < 1
		\quad \text{for each $\frac{1}{\e} < s \leq \frac{3}{2\e}$}. 
	\end{equation}
To this end, from \eqref{prop-eta} and 
	\[
		\vartheta_1'(t) = - \frac{e^2}{2} (t-1)^{-2} \exp \left( - (t-1)^{-1} \right),
	\]
it follows that for $1 < t \leq \frac{3}{2}$
	\[
		\begin{aligned}
			&4\e^4 \left( \vartheta_1'( t ) \right)^2 + \vartheta_1^2 (t)
			\\
			= \ &
			\e^4 e^4 (t-1)^{-4} \exp \left( -2 \left( t- 1 \right)^{-1} \right) + 
			\left[ 1 - \frac{e^2}{2} \exp \left( - (t-1)^{-1} \right) \right]^2
			\\
			= \ & 1 - e^2 
			\left[ 1 - \frac{e^2}{4} \exp \left( - (t-1)^{-1} \right) - \e^4 e^2 (t-1)^{-4} \exp \left( -(t-1)^{-1} \right)  \right] 
			\exp \left( - (t-1)^{-1} \right).
		\end{aligned}
	\]
Since 
	\[
		1 - \frac{e^2}{4} \exp \left( - (t-1)^{-1} \right) \geq 1 - \frac{e^2}{4} e^{ -2 } = \frac{3}{4}
		\quad \text{for every $1 < t \leq \frac{3}{2}$},
	\]
for sufficiently small $\e>0$, 
	\begin{equation}\label{eta1-in2}
		4 \e^4 \left( \vartheta_1'(t) \right)^2 + \vartheta_1^2(t) \leq 1 - \frac{e^2}{2} \exp \left( - (t-1)^{-1} \right) < 1.
	\end{equation}
Hence, $|DU_\e| < 1$ on $\Omega_4$. In addition, by $1 - 4\e^2 (\vartheta_\e'(|z|))^2 - \vartheta_\e^2(|z|) a_\e^2(x_m) \geq 0$, 
we have 
	\begin{equation}\label{bdd-We-O5}
		\begin{aligned}
		W_\e(y,z,x_m)
		&\leq 
		C_\e \left[ 1 - 4\e^2 \left( \vartheta_\e'(|z|) \right)^2 - \vartheta_\e^2 (|z|) a_\e^2(x_m) + |y|^{2\kappa} \right]^{-1/2 }
		\\
		&\leq C_\e |y|^{-\kappa}
		\end{aligned}
		\quad \text{on } \, \Omega_4.
	\end{equation}
Thus, $W_\e \in L^q (\Omega_4)$ for $q < \frac{m-\ell}{\kappa}$. To show $\| \SF_{U_\e} \| \in L^q(\Omega_4)$, by $\kappa \geq 1$, \eqref{eq:normD^2u}, \eqref{eq:D^2uDu}, \eqref{eq:D^2DuDu} 
and \eqref{DUe-D^2Ue} we deduce that, for $(y,z,x_m) \in \Omega_4$, 
	\begin{equation}\label{D2U-D2DUDU}
		\begin{aligned}
			\left| D^2 U_\e \right| & \leq C_\e \left\{ |y|^{2\kappa -2} + \left| \vartheta_\e''(|z|) \right| 
			+ \left| \vartheta_\e'(|z|) \right| + \left| a_\e'(x_m) \right| \right\},
			\\
			\left| D^2 U_\e (D U_\e, \cdot) \right| 
			& \leq 
			C_\e \left\{ |y|^{2\kappa -1} + \left| \vartheta_\e'(|z|) \right| + \left| a_\e'(x_m) \right| \right\},
			\\
			\left| D^2 U_\e (D U_\e,D U_\e) \right| 
			& \leq 
			C_\e \left\{ |y|^{4\kappa -2} + \left( \vartheta_\e'(|z|) \right)^2 + \left| a_\e'(x_m) \right|  \right\}.
		\end{aligned}
	\end{equation}
Due to \eqref{bdd-We-O5}, we verify that for all $q < (m-\ell) /\kappa$
	\begin{equation}\label{O5-r}
		W_\e (y,z,x_m) |y|^{2\kappa -2} + W_\e^2 (y,z,x_m)  |y|^{2\kappa -1} 
		+ W_\e^3 (y,z,x_m) |y|^{4\kappa -2} 
		\leq C_\e |y|^{-1} \in L^q(\Omega_4).
	\end{equation}

	On the other hand, by \eqref{eta1-in} and \eqref{eta1-in2}, we notice that 
	\[
		1 - 4\e^2 \left( \vartheta_\e'(|z|) \right)^2 - \vartheta_\e^2(|z|) 
		\geq \frac{e^2}{2} \exp \left( - \left( \e |z| -1 \right)^{-1} \right),
	\]
which yields 
	\[
		W_\e (y,z,x_m) \leq C_\e \exp \left( \frac{1}{2} \left( \e |z| - 1 \right)^{-1} \right) 
		\quad \text{on } \, \Omega_4.
	\]
From \eqref{prop-eta}, 
	\[
		\left| \vartheta_\e''(|z|) \right| + \left| \vartheta_\e'(|z|) \right| \leq C_\e \left( \left| \e z \right| - 1 \right)^{-4} 
		\exp \left( - \left( \e |z| - 1 \right)^{-1} \right).
	\]
Hence, 
	\begin{equation}\label{O5-eta-1}
		\begin{aligned}
			& W_\e (y,z,x_m) \left\{  \left| \vartheta_\e''(|z|) \right| + \left| \vartheta_\e'(|z|) \right| \right\} 
			+ W_\e^3 (y,z,x_m) \left( \vartheta_\e'(|z|) \right)^2 
			\\
			\leq \ & C_\e \left( \e |z| - 1 \right)^{-4} \exp \left( - \frac{1}{2} \left( \e |z| -1 \right)^{-1} \right) 
			\in L^\infty(\Omega_4).
		\end{aligned}
	\end{equation}
Moreover, 
	\begin{equation}\label{O5-eta-2}
		\begin{aligned}
			& W_\e^2 (y,z,x_m) \left| \vartheta_\e'(|z|) \right| 
			\\
			= \ & W_\e^{2- \kappa^{-1} } (y,z,x_m) W_\e^{  \kappa^{-1} } (y,z,x_m) \left| \vartheta_\e'(|z|) \right| 
			\\
			\leq \ & C_\e \exp \left( \frac{2-\kappa^{-1}}{2} \left( \e |z| -1 \right)^{-1} \right) 
			\left( C_\e |y|^{-\kappa} \right)^{ \kappa^{-1} } 
			\left( \e|z| -1 \right)^{-2} \exp \left( - \left( \e |z| - 1 \right)^{-1} \right)  
			\\
			= \ & C_\e \left( \e |z| - 1 \right)^{-2} \exp \left( - \frac{1}{2\kappa} \left( \e|z| - 1 \right)^{-1} \right) |y|^{-1} 
			\in L^q(\Omega_4) \quad \text{if $q < \frac{m-\ell}{\kappa}$}. 
		\end{aligned}
	\end{equation}
By \eqref{D2U-D2DUDU}, \eqref{O5-r}, \eqref{O5-eta-1}, \eqref{O5-eta-2} and $W_\e \geq 1$, to show 
$\| \SF_{U_\e} \| \in L^q(\Omega_4)$ for $q < (m-\ell)/\kappa$, it remains to prove 
	\begin{equation}\label{int-W3-a'}
		W_\e^3 (x,y,z_m) \left| a_\e'(x_m) \right| \in L^q(\Omega_4) \quad \text{for each $q < \frac{m-\ell}{\kappa}$}. 
	\end{equation}
Since $a_\e'(x_m) = 0$ for $|x_m| \leq \frac{\e}{2}$ and $a_\e$ is even, 
we may suppose $\frac{\e}{2} < x_m \leq \frac{3\e}{2}$. In this case, from \eqref{eq:34.5} and \eqref{prop-eta}, 
notice that 
		\begin{align*}
			&1 - 4\e^2 \left( \vartheta_\e'(|z|) \right)^2 - \vartheta_\e^2(|z|) a_\e^2(x_m)  
			\\
			= \ & 
			\left[ 1 + \vartheta_\e (|z|)  a_\e (x_m) \right] \left[ 1 - \vartheta_\e (|z|)  a_\e (x_m) \right] 
			- 4\e^4 \left( \vartheta_1'(\e|z| ) \right)^2
			\\
			\geq \ & 
			1 - \vartheta_\e (|z|) a_\e (x_m) - 4\e^4 \left( \vartheta_1'(\e|z|) \right)^2
			\\
			\geq\ & 
			1 - \left[ 1 - \frac{e^2}{2} \exp \left( - \left( \e |z| -1 \right)^{-1} \right) \right] 
			\left[ 1 - d_\e \exp \left( - \left(x_m - \e \right)^{-1} \right) \right] 
			- 4\e^4 \left( \vartheta_1'(\e |z|) \right)^2
			\\
			\geq \ & c_0 \left\{ \exp \left( - \left( \e |z| -1 \right)^{-1} \right) + \exp \left( - \left( x_m - \e  \right)^{-1} \right) \right\}
			\doteq c_0 R^2 (|z|,x_m).
		\end{align*}
Thus, by \eqref{bdd-We-O5}, 
	\[
		W_\e(y,z,x_m) \leq C_\e \left\{ R^2(|z|,x_m)  + |y|^{2\kappa}
		 \right\}^{ - \frac{1}{2} }.
	\]
Then we proceed as in \eqref{we-ae} and for $ \frac{m-\ell}{3\kappa} < q < \frac{m-\ell}{\kappa}$, we obtain 
		\begin{align*}
			& \int_{\e}^{\frac{3\e}{2}} \rd x_m \int_{ \frac{1}{\e} < |z| < \frac{3}{2\e} } \rd z 
			\int_{|y| \leq \frac{1}{2\e}} \left( W_\e^3 (y,z,x_m) \left| a_\e'(x_m) \right| \right)^q \rd y
			\\
			\leq \ & 
			C_\e \int_{\e}^{\frac{3\e}{2}} \rd x_m \int_{ \frac{1}{\e} < |z| < \frac{3}{2\e} } \rd z 
			\int_{ |y| \leq R^{1/\kappa}(|z|,x_m) } R^{-3q} (|z|,x_m) \left| a_\e'(x_m) \right|^q \rd y 
			\\
			& \quad 
			+ C_\e \int_{\e}^{\frac{3\e}{2}} \rd x_m \int_{ \frac{1}{\e} < |z| < \frac{3}{2\e} } \rd z 
			\int_{ R^{1/\kappa}(|z|,x_m) \leq |y| \leq \frac{1}{2\e} } 
			|y|^{-3\kappa q} \left| a_\e'(x_m) \right|^q \rd y
			\\
			\leq \ & 
			C_\e \int_{\e}^{\frac{3\e}{2}} \rd x_m \int_{ \frac{1}{\e} < |z| < \frac{3}{2\e} } 
			R^{-3q + \frac{m-\ell}{\kappa} } (|z|,x_m)   \left| a_\e'(x_m) \right|^q \rd z
			\\
			\leq \ &
			C_\e \int_{0}^{ \frac{\e}{2} } \rd t
			\int_{1 }^{\frac{3}{2}} 
			\left\{ \exp \left( - \frac{1}{s-1} \right) + \exp \left( - \frac{1}{t} \right) 
			\right\}^{ \frac{m-\ell - 3\kappa q}{2\kappa} } t^{-2q} \exp \left( - \frac{q}{t} \right)  \rd s
			\\
			= \ & 
			C_\e \int_{0}^{ \frac{\e}{2} } \rd t
			\int_{0 }^{\frac{1}{2}} 
			\left\{ \exp \left( - \frac{1}{s} \right) + \exp \left( - \frac{1}{t} \right) 
			\right\}^{ \frac{m-\ell - 3\kappa q}{2\kappa} } t^{-2q} \exp \left( - \frac{q}{t} \right)  \rd s
			\\
			\leq \ & 
			C_\e \int_{0}^{ \frac{\e}{2} } \rd t 
			\int_0^{t} \exp \left( \frac{3\kappa q - m + \ell }{2\kappa t} \right) t^{-2q} \exp \left( - \frac{q}{t} \right) \rd s
			\\
			& \quad 
			+ C_\e \int_{0}^{ \frac{\e}{2} } \rd t \int_{t}^{\frac{1}{2}} 
			\exp \left( \frac{3\kappa q - m + \ell }{2\kappa s} \right) t^{-2q} \exp \left( - \frac{q}{t} \right) \rd s
			\\
			\leq \ & 
			C_\e \int_{0}^{ \frac{\e}{2} } 
			t^{-2q} \exp \left( \frac{\kappa q - m + \ell }{2\kappa t} \right)
			\rd t < \infty.
		\end{align*}
Hence, \eqref{int-W3-a'} holds and \eqref{Ue-mc-sf} follows. 

	Finally, we prove that $u$ is a weak solution. Let $\eta \in \lip_c(\Omega)$. First, observe that our estimates guarantee that
	\[
		W_\e(y,z,x_m) \leq C_\e |y|^{-\kappa} \quad \text{for each $(y,z,x_m) \in \R^m$}.
	\]
Hence, $W_\e \in L^1(\R^m)$. From $\rho_{U_\e} \in L^q(\R^m)$ and the dominated convergence theorem, 
it follows that 
	\begin{equation}\label{eq:34.2}
		\begin{aligned}
			\int_{\R^m} \rho_{U_\e}  \eta \, \rd x 
			= \lim_{\tau \to 0} \int_{\{|y| > \tau\} } \rho_{U_\e} \eta  \, \rd x 
			= - \lim_{\tau \to 0} \int_{\{ |y| > \tau \} } \diver \left( W_\e DU_\e \right) \eta  \, \rd x. 
		\end{aligned}
	\end{equation}
Integration by parts gives  
	\begin{equation}\label{eq:34.3}
		\begin{aligned}
			- \int_{\{ |y| > \tau \} } \diver \left( W_\e D U_\e \right) \eta \, \rd x 
			&= \int_{\{ |y| = \tau \} } W_\e \eta D U_\e \cdot \frac{y}{|y|} \, \rd \haus^{m-1}_\delta 
			+ \int_{\{ |y| > \tau \}} W_\e D U_\e \cdot D \eta \, \rd x.
		\end{aligned}
	\end{equation}
By \eqref{eq:34.7} and 
	\[
	\left| D U_\e \cdot \frac{y}{|y|} \right| = |(U_\e)_r| = |(u_\e)_r\vartheta_\e| \le C\tau^{2\kappa -1} \quad 
		\text{if $|y|=\tau$}, 
	\]
it follows from the estimate for $W_\e$ and the assumption $\ell \le m-2$ that 
	\[
		\limsup_{\tau \to 0} \int_{\{ |y| = \tau \} } \left| W_\e \eta D U_\e \cdot \frac{y}{|y|} \right| \rd \haus^{m-1}_\delta
		\leq \lim_{\tau \to 0} C \tau^{-\kappa} \tau^{2\kappa -1} \tau^{m-\ell -1} = 0.
	\]
Finally, since $W_\e \in L^1$, it follows from \eqref{eq:34.2} and \eqref{eq:34.3} that 
	\[
		\int_{\R^m} \rho_{U_\e} \eta \, \rd x = \int_{\R^m} W_\e D U_\e \cdot D \eta \, \rd x,
	\]
and we complete the proof. 
\end{proof}

\appendix 

\section{The Born-Infeld model}

We here recall the Born-Infeld model of electromagnetism \cite{borninfeld_1,borninfeld_2}. Concise but informative introductions  can be found in \cite{bpd,bpdr}, see also \cite{Yang, Kiessling, Kiessling_legacy, bialin} for a thorough account of the physical literature. We remark that the Born-Infeld model also proved to be relevant in the theory of superstrings and membranes, see \cite{gibbons,Yang} and the references therein.

As outlined in the Introduction, one of the main concerns of the theory was to overcome the failure of the principle of finite energy occurring in Maxwell's model, which we now describe. In a spacetime $(N^4,g)$ with metric $g = g_{ab}\di y^a \otimes \di y^b$ of signature $(-,+,+,+)$ ($g_{00}<0$), 
the electromagnetic field is described as a closed $2$-form $F = \frac{1}{2} F_{ab} \di y^a \wedge \di y^b$ which, 
according to the Lagrangian formulation of Maxwell's theory due to Schwarzschild (cf. \cite{Schw} and \cite[p.88]{Pauli}), 
in the absence of charges and currents, is required to be stationary for the action
	\[	
	\mathscr{L}_{\M} \doteq \int_{N^4} \mathsf{L}_{\M} \sqrt{-|g|} \di y   \qquad \text{with} \quad \mathsf{L}_\M \doteq -\frac{F^{ab}F_{ab}}{4}, 
	\]
where $|g|$ is the determinant of $g$ and 
$F^{ab} \doteq g^{ac}g^{bd} F_{cd}$. 
The presence of a vector field $J$ describing charges and currents is taken into account by adding the action
	\[
	\mathscr{L}_{J} \doteq \int_{N^4} \mathsf{L}_J \sqrt{-|g|} \di y, \qquad \mathsf{L}_J = J^a\Phi_a,
	\]
where we assumed that $F$ is globally exact and we set $F = \di \Phi$. 
By its very definition, the energy-impulse tensor $T$ associated to $\mathscr{L}_\M + \mathscr{L}_J$ has components
	\[
	T_{ab} 
	= \frac{-2}{\sqrt{-|g|}} \frac{\partial ((\mathsf{L}_\M+ \mathsf{L}_J) \sqrt{-|g|})}{\partial g^{ab}} 
	=  F_{ac}F_{bp}g^{cp} - \frac{1}{4}F^{cp}F_{cp}g_{ab}  + J^c \Phi_c g_{ab}
	\]
and in particular $T_{00}$ describes the energy density. 
In Minkowski space $\mathbb{L}^4$, by writing  in Cartesian coordinates $\{x^a\}$ 
the electromagnetic tensor in terms of the electric and magnetic fields ${\bf E} = E_j \di x^j$ and ${\bf B} = B_j \di x^j$ as   
	\[
	F = \sum_{j =1}^3 E_j \di x^j \wedge \di x^0 + B_1 \di x^2 \wedge \di x^3 + B_2 \di x^3 \wedge \di x^1 + B_3 \di x^1 \wedge \di x^2,  
	\]
the vector potential as $\Phi = -\varphi \di x^0 + {\bf A} = -\varphi \di x^0 + A_j \di x^j$ and 
$J = \rho \partial_{x^0} + {\bf J} = \rho \partial_{x^0} + J^j \partial_{x^j}$, the Maxwell Lagrangian and energy densities become
	\[
	\mathsf{L}_\M + \mathsf{L}_J 
	= \frac{1}{2}\big(|{\bf E}|^2 - |{\bf B}|^2 \big) - \rho \varphi + {\bf A}({\bf J}), 
	\qquad T_{00} = \frac{1}{2} \big(|{\bf E}|^2 + |{\bf B}|^2 \big) + \rho \varphi - {\bf A}({\bf J}).
	\]
Restricting to the electrostatic case with no current density (${\bf B} = 0$, ${\bf E}$ independent of $x^0$, ${\bf J} = 0$), 
from ${\bf E} = -\di \varphi$ the potential $\varphi$ turns out to be stationary for the reduced action
	\[
	J_\rho(v) \doteq \frac{1}{2} \int_{\R^3} |Dv|^2 \di x - \langle \rho, v \rangle,
	\]
where $\langle \rho, v \rangle$ is the duality pairing given, for smooth $\rho$, by integration. 
However, for $\rho = \delta_{x_0}$ the Dirac delta centered at a point $x_0$, 
the Newtonian potential $\bar u_\rho = \mathrm{const} \cdot |x-x_0|^{2-m}$ solving the Euler-Lagrange equation $-\Delta \bar u_\rho = \rho$ for $J_\rho$ has infinite energy on punctured balls centered at $x_0$:
	\[
	\int_{B_R\backslash B_\eps} T_{00} \di x 
	= \frac{1}{2} \int_{B_R\backslash B_\eps} |D \bar u_\rho|^2 \di x \to \infty \qquad \text{as } \, \eps \to 0,
	\]	
a fact of serious physical concern (cf. \cite{borninfeld_2}). The problem also persists for certain sources $\rho \in L^1(\R^m)$, see \cite{fop,bpd}. 
To avoid it, Born and Infeld in \cite{borninfeld_1, borninfeld_2} proposed an alternative Lagrangian density $\mathsf{L}_{\BI}$, defined by
	\[
	\disp \mathsf{L}_{\BI}\sqrt{-|g|} \doteq \sqrt{ - |g|} - \sqrt{ - |g + F|}
	\]
(we follow the convention in \cite{Yang}, which changes signs in $\mathsf{L}_\BI$ with respect to \cite{borninfeld_2}). Indeed, the authors also indicated a further Lagrangian, already suggested in \cite{born}, see \cite[(2.27) and (2.28)]{borninfeld_2}. However, later works in \cite{boillat,plebanski} pointed out a distinctive feature of $\mathsf{L}_\BI$, that of generating an electrodynamics free from the phenomenon of birefringence (cf. \cite{Kiessling_legacy, bialin}). 
%
%
%
%
Computing the determinants in the expression of $\mathsf{L}_\BI$, Born and Infeld obtained
	\[
	\disp \mathsf{L}_{\BI} = 1- \sqrt{1 + \mathbb{F} - \mathbb{G}^2}, \qquad \text{where } \ \ \mathbb{F} \doteq \frac{F^{ab}F_{ab}}{2}, \ \ \mathbb{G} \doteq \frac{F_{ab}(\star F)^{ab}}{4}
%
%
	\]
and $\star$ is the Hodge dual, so $\mathsf{L}_\BI$ is asymptotic to $\mathsf{L}_\M$ for small $F$ and flat $g$.
In Minkowski space and Cartesian coordinates $\{x^a\}$, $\mathsf{L}_\BI$ becomes
	\begin{equation}\label{eq_BI_Minko}
	\mathsf{L}_{\BI} = 1- \sqrt{ 1 - |{\bf E}|^2 + |{\bf B}|^2 - ({\bf B}\cdot {\bf E})^2}.
	\end{equation}
Here, we set the maximal field strength to be $1$ for convenience. 
The energy-impulse tensor associated to $\mathscr{L}_\BI + \mathscr{L}_J$, and its component $T_{00}$ in Cartesian coordinates, 
are thus
	\[
	\begin{array}{lcl}
	T_{ab} & = & \disp \mathsf{L}_{\BI} g_{ab} + \frac{F_{ac}F_{bp}g^{cp} - \mathbb{G}^2g_{ab}}{\sqrt{1 + \mathbb{F} - \mathbb{G}^2}} + J^c \Phi_c g_{ab}, \\[0.5cm]
	T_{00} 
	& = & \disp \frac{1 + |{\bf B}|^2}{\sqrt{ 1 - |{\bf E}|^2 + |{\bf B}|^2 - ({\bf B}\cdot {\bf E})^2}} -1 + \rho \varphi - {\bf A}({\bf J}).
	\end{array}
	\]
In the electrostatic case, the potential $u_\rho$ generated by $\rho$ is therefore required to minimize the action $I_\rho$ in \eqref{def_Irho_intro} on $\Omega = \R^3$ among weakly spacelike functions with a suitable decay at infinity, and its energy density is given by  
	\begin{equation*}
	T_{00} = \disp \frac{1}{\sqrt{ 1 - |Du_\rho|^2}} -1 + \rho u_\rho = w_\rho -1 + \rho u_\rho,
	\end{equation*}
making apparent the link between $w_\rho$ and $T_{00}$ mentioned in the Introduction. Whence, the integrability \eqref{eq_ham_Rm} proved in \cite[Proposition 2.7]{bpd} can be rephrased as the remarkable property 
	\begin{equation}\label{desirable}
	T_{00}- \rho u_\rho \in L^1(\R^3)
	\end{equation}  
for $\rho$ in a large class of distributions including any finite measure on $\R^3$. The case $\rho = \delta_{x_0}$ was previously considered by Born and Infeld in \cite{borninfeld_2} to support the consistency of their theory. For solutions in bounded domains, Proposition \ref{lem_basicL2} guarantees the same desirable property:  $T_{00}- \rho u_\rho \in L^1_\loc(\Omega)$, provided that the boundary datum $\phi$ is not too degenerate.


\section*{Acknowledgements}
The authors would like to express their gratitude to the anonymous referees 
for their comments, especially to the one who suggested the example in Remark \ref{rem:smooth-mc-ls} and the application of Theorem \ref{teo_nolight} in the proof of Theorem \ref{teo_bocofo} $(ii)$.
They wish to thank D. Bonheure and A. Iacopetti for detailed discussions about the results in \cite{bcf,bpd,bi_ARMA,bi_new}, S. Terracini for sharing her insights on the problem, and L. Maniscalco, who carefully read the manuscript and spotted a mistake in Proposition \ref{1001} in an earlier version of the paper. 
This work initiated when J.B. and N.I. visited Scuola Normale Superiore in 2017. 
They would like to express their gratitude for their hospitality and support during their visit to Scuola Normale Superiore. 
J.B. was supported by the National Research Foundation of Korea (NRF) grant funded by the Korea government (MSIT) (No. NRF-2019R1A5A1028324)
N.I. was supported by JSPS KAKENHI Grant Numbers JP 16K17623, 17H02851, 19H01797 and 19K03590. 
A.M. has been supported by the project {\em Geometric problems with loss of compactness} 
from the Scuola Normale Superiore and by GNAMPA as part of INdAM.

\vspace{0.4cm}

\noindent \textbf{Conflict of Interest.} The authors have no conflict of interest.

\medskip

\noindent
\textbf{Data Availability Statement.} Data sharing is not applicable to this article as no datasets were generated or analyzed during the current study.



\end{document}